\documentclass[11pt]{article}

\usepackage{fullpage}
\usepackage[ruled, linesnumbered, vlined, commentsnumbered]{algorithm2e}
\usepackage{graphicx,subfig} 
\usepackage{amsfonts,amssymb,amsmath,amsthm,amsopn}	
\usepackage{hhline,booktabs,diagbox,colortbl,multirow,tabularx,threeparttable} 
\usepackage[listings,skins,breakable]{tcolorbox}

\usepackage{enumerate}
\usepackage{authblk}
\usepackage{footnote}
\usepackage{hyperref}
\usepackage{prettyref}
\usepackage{cite}
\usepackage{setspace}
\usepackage{color}
\usepackage{xcolor}  
\usepackage{scrpage2}
\usepackage{lscape}
\usepackage{geometry}

\usepackage{tikz}
\usepackage{pgfplots}
\usetikzlibrary{positioning,shapes,shadows,arrows,calc}
\tikzstyle{component}=[rectangle, draw=black, rounded corners, fill=blue!40, drop shadow, text centered, anchor=north, text=white, minimum height=1cm]
\tikzstyle{arrow}=[->, thick]

\pgfplotsset{compat=1.12}
\usetikzlibrary{intersections}
\usetikzlibrary{pgfplots.statistics}
\usepgfplotslibrary{fillbetween}

\definecolor{myblue}{RGB}{34,31,217}
\definecolor{mycyan}{gray}{.7}
\definecolor{Gray}{gray}{0.9}

\newtheorem{remark}{Remark}
\newtheorem{theorem}{Theorem}

\newtheorem{corollary}{Corollary}
\newtheorem{definition}{Definition}
\newtheorem{lemma}{Lemma}

\DeclareMathOperator*{\argmax}{argmax}
\DeclareMathOperator*{\argmin}{argmin}

\newcommand{\pref}{\prettyref}

\newrefformat{fig}{Fig.~\ref{#1}}
\newrefformat{tab}{Table~\ref{#1}}
\newrefformat{sec}{Section~\ref{#1}}
\newrefformat{app}{Appendix~\ref{#1}}
\newrefformat{alg}{Algorithm~\ref{#1}}
\newrefformat{property}{Property~\ref{#1}}
\newrefformat{theorem}{Theorem~\ref{#1}}
\newrefformat{corollary}{Corollary~\ref{#1}}
\newrefformat{proposition}{Proposition~\ref{#1}}
\newrefformat{def}{Definition~\ref{#1}}
\newrefformat{eq}{equation~(\ref{#1})}

\title{\vspace{-1ex}\LARGE\textbf{Knee Point Identification Based on Trade-Off Utility}\footnote{This manuscript is submitted for potential publication. Reviewers can use this version in peer review.}}

\author[1]{\normalsize Ke Li}
\author[2]{\normalsize Haifeng Nie}
\author[2]{\normalsize Huiru Gao}
\author[3,4]{\normalsize Xin Yao}
\affil[1]{\normalsize Department of Computer Science, University of Exeter, EX4 4QF, Exeter, UK}
\affil[2]{\normalsize College of Computer Science and Engineering, University of Electronic Science and Technology of China, 611731, Chengdu, China}
\affil[3]{\normalsize Shenzhen Key Lab of Computational Intelligence, Department of Computer Science and Engineering, Southern University of Science and Technology, Shenzhen 518055, China}
\affil[4]{\normalsize CERCIA, School of Computer Science, University of Birmingham, Edgbaston, Birmingham, B15 2TT, UK}
\affil[$\ast$]{\normalsize Email: \texttt{k.li@exeter.ac.uk}}

\date{}

\begin{document}
\maketitle

\vspace{-3ex}
{\normalsize\textbf{Abstract: }Knee points, characterised as their smallest trade-off loss at all objectives, are attractive to decision makers in multi-criterion decision-making. In contrast, other Pareto-optimal solutions are less attractive since a small improvement on one objective can lead to a significant degradation on at least one of the other objectives. In this paper, we propose a simple and effective knee point identification method based on trade-off utility, dubbed KPITU, to help decision makers identify knee points from a given set of trade-off solutions. The basic idea of KPITU is to sequentially validate whether a solution is a knee point or not by comparing its trade-off utility with others within its neighbourhood. In particular, a solution is a knee point if and only if it has the best trade-off utility among its neighbours. Moreover, we implement a GPU version of KPITU that carries out the knee point identification in a parallel manner. This GPU version reduces the worst-case complexity from quadratic to linear. To validate the effectiveness of KPITU, we compare its performance with five state-of-the-art knee point identification methods on 134 test problem instances. Empirical results fully demonstrate the outstanding performance of KPITU especially on problems with many local knee points. At the end, we further validate the usefulness of KPITU for guiding EMO algorithms to search for knee points on the fly during the evolutionary process.}



\section{Introduction}
\label{Introduction}

The multi-objective optimisation problem (MOP) considered in this paper is formulated as:
\begin{equation}
\begin{aligned}
& \mathrm{minimize} \quad \mathbf{F}(\mathbf{x}) = (f_1(\mathbf{x}),\cdots,f_m(\mathbf{x}))^T  \label{equ:MOOP} \\
& \mathrm{subject \ to} \quad  \mathbf{x} \in \Omega  \\
  \end{aligned},
\end{equation}
where $\mathbf{x} = (x_1,\cdots,x_n)^T$ is a decision (variable) vector (also known as a candidate solution) from the decision (variable) space $\Omega=\Pi_{i=1}^n[a_i,b_i]\subseteq\mathbb{R}^n$. $\mathbf{F}(\mathbf{x})$ constitutes a vector of objective functions which are conflicting to each other in the objective space $\mathbb{R}^m$. In multi-objective optimisation, a solution $\mathbf{x}^1\in\Omega$ is said to be better than $\mathbf{x}^2\in\Omega$ (also known as $\mathbf{x}^1$ dominates $\mathbf{x}^2$, denoted as $\mathbf{x}^1\preceq\mathbf{x}^2$) if and only if $f_i(\mathbf{x}^1)\leq f_i(\mathbf{x}^2)$ for all $i\in\{1,\cdots,m\}$ and $f_i(\mathbf{x}^1)<f_i(\mathbf{x}^2)$ for at least one $i\in\{1,\cdots,m\}$. A solution $\mathbf{x}^{\ast}$ is Pareto optimal in case there does not exist a solution $\mathbf{x}\in\Omega$ that dominates $\mathbf{x}^{\ast}$. The set of all Pareto optimal solutions is called the Pareto set (PS). $ PF = \{\mathbf{F}(\mathbf{x})|\mathbf{x} \in PS\} $ is called the Pareto front (PF).

Due to the population-based property, evolutionary algorithms (EAs) have been widely accepted as a major approach for multi-objective optimisation. In other words, EAs are able to approximate a set of non-dominated solutions simultaneously thus provide decision makers (DMs) a diverse set of trade-off alternatives in decision-making. However, having excessive trade-off alternatives is a double-edged sword because the quality of decision-making will be compromised in the presence of a large number of alternatives~\cite{Jacoby74,LiDY18,Li19,LiCSY19}. This will be further aggravated with the increase of the number of objectives. To alleviate the DMs' cognitive burden, it might be more acceptable to pre-screen the trade-off solution set obtained by an evolutionary multi-objective optimisation (EMO) algorithm and only choose a few trade-off solutions, commonly referred to as \textit{knee} points, before handing over to DMs for multi-criterion decision-making (MCDM).

Generally speaking, one of the key characteristics of the knee point, which makes it intriguing to DMs, is a small improvement on one objective can lead to a significant degradation on at least one of the other objectives. Although knee points are of extreme importance for MCDM, the relevant research is surprisingly lukewarm in the EMO community, comparing to the flourishing developments of various EMO algorithms. In the past two decades, there have been some efforts devoted to the development of effective methods for identifying knee points among a set of trade-off alternatives obtained by an EMO algorithm. Their ideas can be divided into two categories: one is mainly based on the trade-off information collected from pair-wise comparisons among solutions~\cite{DebG11,ChiuYJ16,BhattacharjeeSR17}; whilst the other is according to the geometry characteristics of the approximated PF~\cite{Das99,SchutzeLC08,YuJO18}. It is worth noting that some methods (e.g.,~\cite{DebG11}) were merely designed for the two-objective scenario. In addition, some methods (e.g.,~\cite{ChiuYJ16} and~\cite{BrankeDDO04}) just consider the global information of the approximated PF thus they can only identify one knee point, also known as the global knee point. However, it is not uncommon that there exist more than one knee region. On the other hand, instead of identifying knee points after running an EMO algorithm, there is another line of research that aims to guide the population towards the knee points on the fly during the evolutionary process~\cite{Ikeda2001,BrankeDDO04,RachmawatiS06,RachmawatiS06a,BechikhSG10,BechikhSG11,ZhangTJ15,SudengW15,Ramirez-Atencia17,YuJO19}. Note that most, if not all, of these methods have at least one parameter, the setting of which depends on the shape of the approximated PF. To facilitate the study and understanding of knee points, there are some benchmark test problems with understandable and controllable setup of knee points developed in the literature~\cite{BrankeDDO04,YuJO18,YuJO19a}.

In this paper, we propose a simple and effective knee point identification (KPI) method based on trade-off utility, dubbed KPITU, to help DMs select knee point(s) among a set of trade-off solutions. Our major contributions are outlined as follows. 
\begin{itemize}
    \item The basic idea of KPITU is to sequentially locate knee point(s) based on trade-off utility obtained by pair-wise comparisons among neighbouring solutions. To this end, we first use a set of evenly distributed weight vectors to divide the objective space into several subregions, each of which is associated with solution(s) from the given trade-off solution set. Afterwards, for each solution, we compare it with its neighbouring solutions whilst this solution is a knee point if and only if its trade-off utility is better than other solutions in its neighbourhood.
    \item In addition to the sequential implementation which has a quadratic complexity, we develop a GPU version of KPITU that implements the knee point search at each solution in a parallel manner. As a result, the worst case time complexity of this GPU version of KPITU is linear.
    \item The effectiveness of the proposed KPITU is compared with five state-of-the-art KPI methods on 134 test problem instances. In addition to a visual comparison, we use a metric to carry out a quantitative evaluation of the performance of different KPI methods.
    \item Experimental results demonstrate that KPITU is scalable to any number of objectives. Its performance is comparable to the other five state-of-the-art KPI methods when there is only one global knee point whilst its superiority becomes more evident with the increase of the number of knee points.
    \item Last but not the least, KPITU is incorporated into the environmental selection of NSGA-II~\cite{DebAPM02} to demonstrate its usefulness as an operator to guide an EMO algorithm to search for knee points on the fly during the evolutionary process.
\end{itemize}

The rest of this paper is organised as follows. \pref{sec:literature} provides a literature review on the existing developments related to knee points. \pref{sec:proposal} describes the implementation of our proposed KPITU for finding knee points. \pref{sec:experiments} compares and analyses the performance of our proposed KPITU against five state-of-the-art algorithms. Finally, \pref{sec:conclusion} concludes this paper and shed some lights on future directions.


\section{Related Works}
\label{sec:literature}

\begin{figure*}[htbp]
    \centering
    \subfloat[KPI method based on cone-domination.]{\includegraphics[width=.33\linewidth]{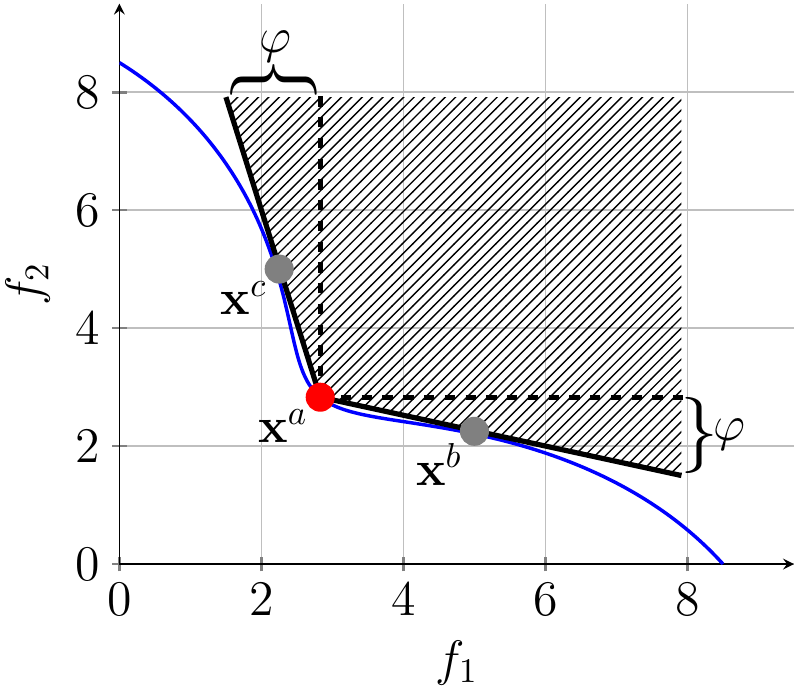}}
    \subfloat[KPI method based on EMU.]{\includegraphics[width=.33\linewidth]{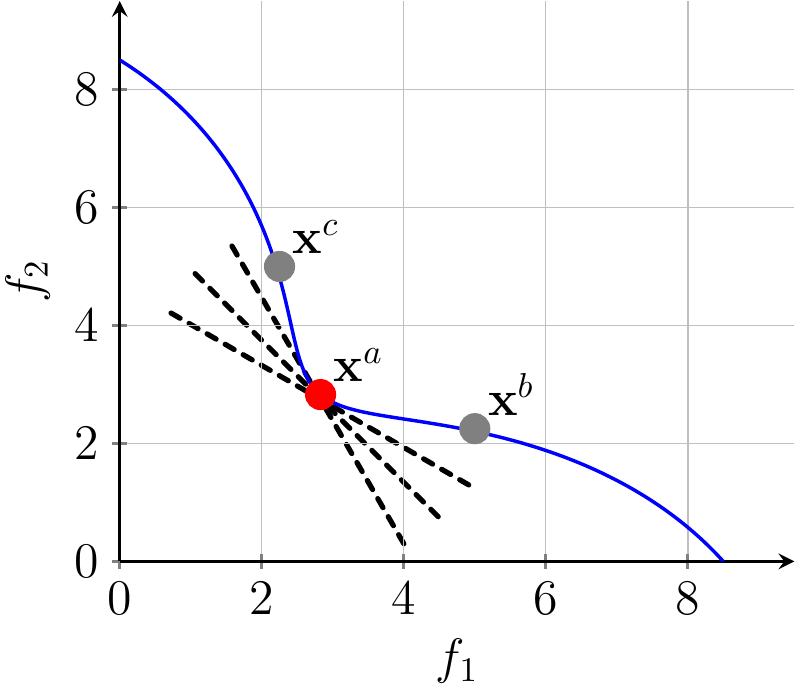}}
    \subfloat[KPI method based on hyperplane.]{\includegraphics[width=.33\linewidth]{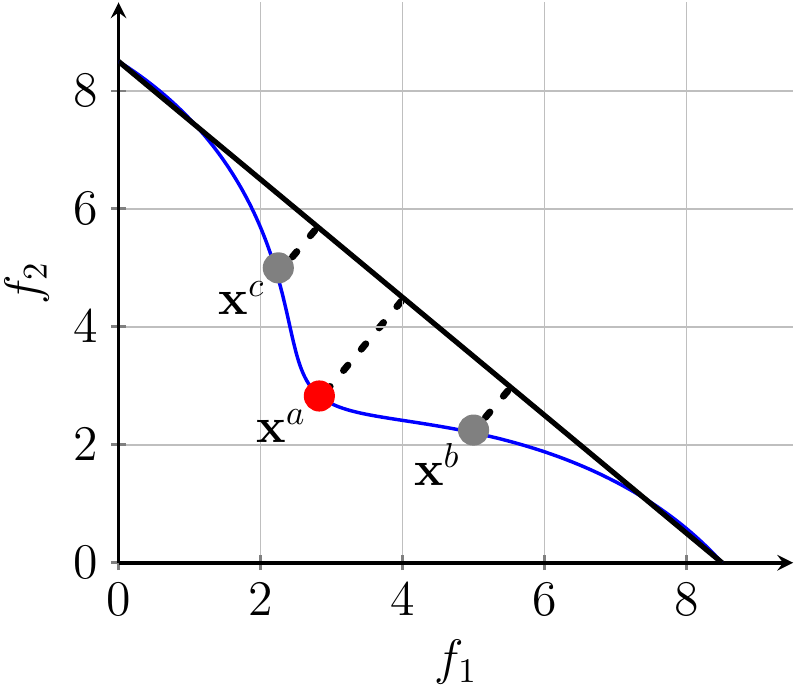}}
    \caption{Illustration of working mechanisms of different KPI methods.}
    \label{fig:knee_methods}
\end{figure*}

In this section, we provide a systematic overview of the existing developments of KPI either posteriori or progressively during the evolutionary process. Our literature review is organised according to the criteria used to identify knee point(s), i.e., trade-off information and geometry characteristics.

\subsection{KPI Based on Trade-Off Information}
\label{sec:trade-off}

One of the key characteristics of knee point is its large trade-off offset at different objectives. Thus, it is nature to take advantages of the trade-off information to find knee point(s).

\subsubsection{Approaches Based on Modified Dominance Relation}
\label{sec:domination}

The original Pareto dominance is a strict partial order relation. It implies an equal importance and the same amount of trade-off at all objectives thus entails no preference at any objective. Considering the example shown in~\pref{fig:knee_methods}(a), $\mathbf{x}^a$ is the knee point. However, $\mathbf{x}^a$, $\mathbf{x}^b$ and $\mathbf{x}^c$ will be treated as equally good according to the Pareto dominance relation. Therefore, there is no selection pressure to guide the search towards the knee point. A natural idea to amend this is to relax the dominance conditions by amending extra trade-off rates to different objectives. For example, Ram\'irez-Atencia et al. proposed a cone-domination~\cite{Ramirez-Atencia17} where a weighted function of the objectives is defined as:
\begin{equation}
    \mathbf{\Omega}(\mathbf{x})=
\begin{bmatrix}
    1 & a_{12} & \dots & a_{1m} \\
    a_{21} & 1 & \dots  & a_{2m} \\
    \vdots & \vdots & \ddots & \vdots \\
    a_{m1} & a_{m2} & \dots  & 1
\end{bmatrix}\cdot\mathbf{F}(\mathbf{x}),
\end{equation}
where $a_{ij}$, both $i$ and $j\in\{1,\cdots,m\}$, is the amount of gain in the $j$-th objective function for a loss of one unit in the $i$-th objective function. By leveraging this definition, the cone-domination is simply defined by replacing the original objective functions $\mathbf{F}(\mathbf{x})$ by these weighted functions $\mathbf{\Omega}(\mathbf{x})$. As demonstrated in~\pref{fig:knee_methods}(a), the effect of such a relaxed dominance relation can be understood as an extension of the coverage of dominance area to an angle $\varphi$ where $\tan\frac{\varphi-90}{2}=a_{ij}$. By doing so, we can see that $\mathbf{x}^b$ and $\mathbf{x}^c$ are cone-dominated by $\mathbf{x}^a$ thus this cone-domination can provide a necessary selection pressure to drive the search towards knee point(s) with prescribed trade-off offsets at different objectives. Similar ideas were also explored in~\cite{SudengW15} and~\cite{YuJO19}, respectively.

\subsubsection{Approaches Based on Utility Function}
\label{sec:utility}

Utility function, or as known as scalarising function, have been widely used in the MCDM community to aggregate a MOP into a single-objective optimisation problem. Trade-off among different objectives is usually represented as a weight vector when used in a utility function. For example, Branke et al.~\cite{BrankeDDO04} proposed an individual marginal utility function:
\begin{equation}
    \hat{U}(\mathbf{x}_i,\lambda)=
    \begin{cases}
        \min\limits_{j\neq i}U(\mathbf{x}_j,\lambda)-U(\mathbf{x}_i,\lambda), & \text{if } i=\argmin U(\mathbf{x}_j,\lambda) \\
        0, & \text{otherwise}
    \end{cases},
\end{equation}
where $\lambda\in[0,1]$ is a weight vector specified by the DM to elicit the trade-off among objectives and $U(\mathbf{x},\lambda)=\lambda f_1(\mathbf{x})+(1-\lambda)f_2(\mathbf{x})$ in the two-objective case as an example. Based on this definition, the expected marginal utility (EMU) is calculated as an integral of the individual marginal utility function across a set of different weight vectors. Solution(s) having the largest EMU is(are) recognised as the knee point(s). In other words, the knee point is the solution that finds the minimum value of the utility function for the largest number of different weight vectors as the example shown in~\pref{fig:knee_methods}(b). In particular, dash lines represent the contours of utility functions with different  different weight vector settings. It is not difficult to envisage that the EMU method is scalable to problems with a large number of objectives. However, as discussed in~\cite{BhattacharjeeSR17}, the EMU becomes less discriminative with the increase of the number of objectives. Accordingly, Bhattacharjee et al. proposed an improved version of EMU, dubbed EMU$^r$, that is not only able to identify knee point(s) but also can discriminate their positions (i.e., whether they are on the internal or peripheral part of the PF). In~\cite{RachmawatiS06a} and \cite{RachmawatiS09}, Rachmawati and Srinivasan proposed a metric $\mu(\mathbf{x}^i,\mathrm{S})$ to evaluate the value of a solution $\mathbf{x}^i$, in terms of trade-off performance, with respect to the trade-off solution set $\mathrm{S}$ as:
\begin{equation}
    \mu(\mathbf{x}^i,S)=\min\limits_{\mathbf{x}^j\in\mathrm{S},\mathbf{x}^i\npreceq\mathbf{x}^j,\mathbf{x}^j\npreceq\mathbf{x}^i}T(\mathbf{x}^i,\mathbf{x}^j),
\end{equation}
where
\begin{equation}
    T(\mathbf{x}^i,\mathbf{x}^j)=\frac{\sum_{i=1}^m\max[0,\frac{f_i(\mathbf{x}^j)-f_i(\mathbf{x}^i)}{f_i^{\max}-f_i^{\min}}]}{\sum_{i=1}^m\max[0,\frac{f_i(\mathbf{x}^i)-f_i(\mathbf{x}^j)}{f_i^{\max}-f_i^{\min}}]},
\end{equation}
where $f_i^{\max}$ and $f_i^{\min}$ are the maximum and the minimum values at the $i$-th objective function of $\mathrm{S}$. $\mu(\mathbf{x}^i,\mathrm{S})$ evaluates the least amount of improvement per unit deterioration obtained by exchanging any alternative solution $\mathbf{x}^j$ in $\mathrm{S}$ with $\mathbf{x}^i$. In particular, the solution $\mathbf{x}^i$ is recognised as a knee point in case $\mu(\mathbf{x}^i,\mathrm{S})$ is the local maximum over $\mathrm{S}$. In~\cite{BechikhSG11}, Bechikh et al. applied this metric in the environmental selection of an EMO algorithm to guide the search towards knee point(s).

\subsection{KPI Based on Geometry Characteristics}
\label{sec:geomety}

Different from the other part of the PF, the knee region, where knee point(s) locate, has a distinctive geometry characteristic, i.e., it incurs an abrupt change in the curvature of the PF manifold. This characteristic has also been widely used to find knee point(s).

\subsubsection{Approaches Based on Angle Selection}
\label{sec:angle}

In~\cite{BrankeDDO04}, Branke et al. proposed a reflex angle based selection mechanism to replace the crowding distance calculation in the original NSGA-II. Specifically, the reflex angle is defined as the angle formed by two lines that cover the neighbouring area of the underlying solution. As for a given set of trade-off solutions, the one with the largest reflex angle is recognised as the knee point. In fact, the basic idea of the reflex angle is geometrically similar to the effect of the cone-domination discussed in~\pref{fig:knee_methods}(a). Since the reflex angle only take the local information of the PF into consideration, Deb and Gupta proposed a bend angle based selection mechanism to remedy this issue~\cite{DebG11}. Note that both the reflex angle and bend angle based selection mechanisms assume a normalised objective space, which can hardly be met in real-world applications. Deb and Gupta~\cite{DebG11} proposed a $(\alpha,\beta)$-selection that takes user prescribed trade-off information into consideration. By doing so, the selection pressure can be adjusted according to the scales of different objectives.
\vspace{-.19em}

\subsubsection{Approaches Based on Hyperplane}
\label{sec:hyperplane}

In order to identify the maximum convex bulge at the PF, Das~\cite{Das99} proposed to first setup a convex hull of individual minima (CHIM), i.e., a hyperplane formed by the individual minimum at each objective in the normalised objective space as shown in~\pref{fig:knee_methods}(c). Given a trade-off solution set, the solution having the largest perpendicular distance to this hyperplane is recognised as the knee point. Based on the same idea, Sch\"utze et al. proposed two improved strategies to find the maximal convex bulge, i.e., knee point(s)~\cite{SchutzeLC08}. In~\cite{BechikhSG10} and~\cite{ZhangTJ15}, the environmental selection of an EMO algorithm is guided by the \lq knee points\rq\ of the current approximated PF. In particular, the \lq knee points\rq\ are identified by using the method proposed by Das. Instead of calculating the perpendicular distance towards the hyperplane, Yu et al.~\cite{YuJO18} proposed to first project the trade-off solutions onto the hyperplane. Afterwards, the knee point(s) are identified according to the density of those projected solutions on the hyperplane. In~\cite{ChiuYJ16}, Chiu et al. proposed to use the minimum Manhattan distance (MMD) to a hyperplane to identify knee point(s) from the given trade-off solution set $\mathrm{S}$. Specifically,
\begin{equation}\label{eq:mmd}
	\mathbf{x}^{\ast}=\argmin_{\mathbf{x}\in\mathrm{S}}\|\mathbf{F}(\mathbf{x})-\mathbf{z}^{\ast}\|_1,
\end{equation}
where $\mathbf{x}^{\ast}$ is the knee point, $\mathbf{z}^{\ast}=(z_1^{\ast},\cdots,z_m^{\ast})$ is the ideal point where $z_i^{\ast}=\min\limits_{\mathbf{x}\in\mathrm{S}}f_i(\mathbf{x})$, $i\in\{1,\cdots,m\}$ and $\|\cdot\|_1$ is the Manhattan distance. Different from the other methods, the hyperplane used in the MMD is formed by $\mathbf{z}^{\ast}$ with a prescribed orientation, i.e., the predefined importance of different objectives. Different from the aforementioned methods, Bhattacharjee et al.~\cite{BhattacharjeeSR16} proposed a method to estimate the local curvature of the underlying PF shape within a neighborhood to identify knee points. Instead of calculating the distance towards a hyperplane, this method uses a general format of a hyperplane to facilitate the local curvature estimation.

\subsection{Remarks on Existing Methods}
\label{sec:comments}

In the following paragraphs, we make some remarks on the aforementioned KPI methods.
\begin{remark}
    Some KPI methods (e.g.,~\cite{DebG11,SchutzeLC08} and \cite{BrankeDDO04}) are merely designed for the two-objective scenarios and are not scalable to higher dimensional cases. It is controversial to use computational methods to identify knee point(s) in the two-objective case where the PF manifold can be easily visualised to assist the decision-making.
\end{remark}
\begin{remark}
    Most, if not all, KPI methods are designed for problems with a convex PF shape given that the knee point(s) are presumably lying in the convex bulge. Although the KPI methods proposed in~\cite{BhattacharjeeSR17} and~\cite{YuJO18} are supposed to identify knee point(s) on problems with both convex and concave PF shapes, their performance depend on some parameters.
\end{remark}
\begin{remark}
    Followed the previous issue, many existing KPI methods have at least one control parameter, the setting of which is problem dependent. Unfortunately, there is no thumb rule to set those parameter(s) in order to control the behaviour of the corresponding KPI methods, including the number of knee regions identified (e.g.,~\cite{BhattacharjeeSR17} and~\cite{YuJO18}) and the selection pressure towards the knee points (e.g.,~\cite{SudengW15,Ramirez-Atencia17,YuJO19}).
\end{remark}
\begin{remark}
    Some KPI methods have an assumption of only one global knee point (e.g.,~\cite{ChiuYJ16,DebG11} and~\cite{BrankeDDO04}). However, it is not uncommon that there exist more than one knee region for problems with complex PF shapes.
\end{remark}
Bearing these aforementioned issues in mind, this paper proposes a simple and effective KPI method based on trade-off utility. In particular, the trade-off utility measures the trade-off offsets of a pair of non-dominated solutions at different objectives. The proposed KPITU is parameterless and is able to identify as many knee point(s) as possible on problems with various PF shapes and is scalable to any number of objectives.

\section{Proposed Algorithm}
\label{sec:proposal}

This section delineates the technical details of our proposed KPITU to identify the knee point(s) of a given trade-off solution set $\mathrm{S}=\{\mathbf{x}^i\}_{i=1}^N$. First, we start with the foundation to understand knee point from a trade-off perspective, i.e., trade-off utility which is the main crux of KPITU. Thereafter, we explain the algorithmic implementation of KPITU step by step. At the end, we explore a GPU implementation of KPITU that accelerates the KPI process in a parallel manner.

\subsection{Understanding Knee Point from a Trade-Off Perspective}
\label{sec:TOU}

Trade-off is one of the key characteristics of a pair of non-dominated solutions. It is characterised as an advantage at one objective accompanied by an inadequacy of at least one other objective. On one hand, trade-off is the source of incomparability among non-dominated solutions. On the other hand, trade-off can be accountable for decision-making. Specifically, given a pair of non-dominated solutions, the one whose net gain of improvement in advantageous objectives exceeds the deterioration in inferior objectives is preferred by the DM. Let us consider the example shown in~\pref{fig:knee_methods} again, a DM will prefer $\mathbf{x}^a$ over $\mathbf{x}^b$ and $\mathbf{x}^c$ without any \textit{a priori} preference information on the underlying problem, e.g., the importance of different objectives. In other words, $\mathbf{x}^a$ is more likely to be a knee point than $\mathbf{x}^b$ and $\mathbf{x}^c$. Conceptually, the knee point is the one that has the largest improvement versus deterioration rate. Based on this intuition, we define the trade-off utility as follows.

\begin{definition}\label{def:knee_dominance}
    Given two non-dominated solutions $\mathbf{x}^1$ and $\mathbf{x}^2$ of a trade-off solution set $\mathrm{S}$, the trade-off utility of $\mathbf{x}^1$ over $\mathbf{x}^2$ is defined as: 
    \begin{equation}
		\mathbb{U}(\mathbf{x}^1,\mathbf{x}^2)=\mathbb{G}(\mathbf{x}^1,\mathbf{x}^2)+\mathbb{D}(\mathbf{x}^1,\mathbf{x}^2),
        \label{eq:tradeoff_utility}
    \end{equation}
    where
    \begin{equation}\label{eq:gain}
        \mathbb{G}(\mathbf{x}^1,\mathbf{x}^2)=\sum_{i=1}^m[\frac{f_i(\mathbf{x}^1)-f_i(\mathbf{x}^2)}{z^{\max}_i-z^{\min}_i}]^+,
    \end{equation}
    where $z^{\max}_i=\max\limits_{\mathbf{x}\in\mathrm{S}}f_i(\mathbf{x})$ and $z^{\min}_i=\min\limits_{\mathbf{x}\in\mathrm{S}}f_i(\mathbf{x})$, $i\in\{1,\cdots,m\}$ and
    \begin{equation}
        [x]^+=
        \begin{cases}
            x, & \text{if } x<0\\
            0, & \text{otherwise}
        \end{cases},
    \end{equation}
    and
    \begin{equation}\label{eq:loss}
        \mathbb{D}(\mathbf{x}^1,\mathbf{x}^2)=\sum_{i=1}^m[\frac{f_i(\mathbf{x}^1)-f_i(\mathbf{x}^2)}{z^{\max}_i-z^{\min}_i}]^-,
    \end{equation}
    where 
    \begin{equation}
        [x]^-=
        \begin{cases}
            x, & \text{if } x>0\\
            0, & \text{otherwise}
        \end{cases}.
    \end{equation}
\end{definition} 
In particular, $\mathbb{G}(\mathbf{x}^1,\mathbf{x}^2)$ represents the net gain of improvement of $\mathbf{x}^1$ over $\mathbf{x}^2$ in the advantageous objectives whilst $\mathbb{D}(\mathbf{x}^1,\mathbf{x}^2)$ evaluates the deterioration of $\mathbf{x}^1$ with respect to $\mathbf{x}^2$ in the inferior objectives. Based on the trade-off utility, we have the knee-dominance concept as follows. 
\begin{definition}\label{def:knee_dominance}
    Given two non-dominated solutions $\mathbf{x}^1$ and $\mathbf{x}^2$, $\mathbf{x}^1$ is defined to knee-dominate $\mathbf{x}^2$, denoted as $\mathbf{x}^1\preceq_k\mathbf{x}^2$, in case $\mathbb{U}(\mathbf{x}^1,\mathbf{x}^2)<0$; $\mathbf{x}^1$ is defined to be non-knee-dominated with $\mathbf{x}^2$, denoted as $\mathbf{x}^1\cong_k\mathbf{x}^2$, in case $\mathbb{U}(\mathbf{x}^1,\mathbf{x}^2)=0$; otherwise, $\mathbb{U}(\mathbf{x}^1,\mathbf{x}^2)>0$.
\end{definition}
The basic idea of the knee-dominance is to provide the ability to discriminate non-dominated solutions from a trade-off perspective. In the following paragraphs, we will discuss some interesting properties of this knee-dominance relation.

\vspace{0.5em}
\noindent\textbf{Property 1.} \textit{The knee-dominance is an irreflexive relation on the given trade-off solution set $\mathrm{S}$.}
\begin{proof}
    We would like to show that $\forall\mathbf{x}\in\mathrm{S}, \mathbf{x}\npreceq_k\mathbf{x}$.
    
    Because $\mathbb{U}(\mathbf{x},\mathbf{x})=\mathbb{G}(\mathbf{x},\mathbf{x})+\mathbb{D}(\mathbf{x},\mathbf{x})=0$, $\mathbf{x}\cong_k\mathbf{x}$. Thus, knee-dominance is irreflexive.
\end{proof}
\noindent\textbf{Property 2.} \textit{The knee-dominance is an asymmetric relation on the given trade-off solution set $\mathrm{S}$.}
\begin{proof} 
    We would like to show that if $\mathbf{x}^1\preceq_k\mathbf{x}^2$ then $\mathbf{x}^2\npreceq_k\mathbf{x}^1$.

    If $\mathbf{x}^1\preceq_k\mathbf{x}^2$ then we have $\mathbb{U}(\mathbf{x}^1,\mathbf{x}^2)=\mathbb{G}(\mathbf{x}^1,\mathbf{x}^2)+\mathbb{D}(\mathbf{x}^1,\mathbf{x}^2)<0$. It is easy to verify that $\mathbb{G}(\mathbf{x}^1,\mathbf{x}^2)=-\mathbb{D}(\mathbf{x}^2,\mathbf{x}^1)$ and $\mathbb{D}(\mathbf{x}^1,\mathbf{x}^2)=-\mathbb{G}(\mathbf{x}^2,\mathbf{x}^1)$. Thus, we have $\mathbb{U}(\mathbf{x}^2,\mathbf{x}^1)=\mathbb{G}(\mathbf{x}^2,\mathbf{x}^1)+\mathbb{D}(\mathbf{x}^2,\mathbf{x}^1)=-[\mathbb{G}(\mathbf{x}^1,\mathbf{x}^2)+\mathbb{D}(\mathbf{x}^1,\mathbf{x}^2)]>0$. That is to say, $\mathbf{x}^2\npreceq_k\mathbf{x}^1$.
\end{proof} 
\noindent\textbf{Property 3.} \textit{The knee-dominance is a transitive relation on the given trade-off solution set $\mathrm{S}$.}
\begin{proof}
    We would like to show that if $\mathbf{x}^1\preceq_k\mathbf{x}^2$ and $\mathbf{x}^2\preceq_k\mathbf{x}^3$, then $\mathbf{x}^1\preceq_k\mathbf{x}^3$.

	According to~\pref{def:knee_dominance}, we can rewrite $\mathbb{U}(\mathbf{x}^1,\mathbf{x}^2)=\sum_{i=1}^m[\frac{f_i(\mathbf{x}^1)-f_i(\mathbf{x}^2)}{z^{\max}_i-z^{\min}_i}]$. Since we have $\mathbb{U}(\mathbf{x}^1,\mathbf{x}^2)=\sum_{i=1}^m[\frac{f_i(\mathbf{x}^1)-f_i(\mathbf{x}^2)}{z^{\max}_i-z^{\min}_i}]<0$ and $\mathbb{U}(\mathbf{x}^2,\mathbf{x}^3)=\sum_{i=1}^m[\frac{f_i(\mathbf{x}^2)-f_i(\mathbf{x}^3)}{z^{\max}_i-z^{\min}_i}]<0$, we can derive that $\mathbb{U}(\mathbf{x}^1,\mathbf{x}^3)=\sum_{i=1}^m[\frac{f_i(\mathbf{x}^1)-f_i(\mathbf{x}^2)+f_i(\mathbf{x}^2)-f_i(\mathbf{x}^3)}{z^{\max}_i-z^{\min}_i}]<0$. As a result, $\mathbf{x}^1\preceq_k\mathbf{x}^3$.
\end{proof} 
Note that $\mathbf{x}^1\preceq_k\mathbf{x}^2$ means $\mathbf{x}^1$ has a larger chance to become a knee point than $\mathbf{x}^2$. In other words, a DM prefers $\mathbf{x}^1$ over $\mathbf{x}^2$ in decision-making. Based on the knee-dominance concept, we have the following theorem to support the identification of the knee point.

\begin{theorem}
    Given a set of trade-off solution set $\mathrm{S}$, $\mathbf{x}^{\ast}\in\mathrm{S}$ is the knee point if and only if there does not exist another solution $\mathbf{x}^{\prime}\in\mathrm{S}$ that $\mathbf{x}^{\prime}\preceq_k\mathbf{x}^{\ast}$.
    \label{theorem:global_knee}
\end{theorem}
\begin{proof}
    This theorem can be proved by contradiction. If $\exists\mathbf{x}^{''}\in \mathrm{S}$ such that $\mathbf{x}^{''}\preceq_k\mathbf{x}^{\ast}$. Then $\mathbf{x}^{''}$ has a larger chance to become a knee point $\mathbf{x}^{\ast}$. This contradicts the assertion that $\mathbf{x}^{\ast}$ is the knee point.
\end{proof}
\begin{remark}
    It is worth noting that~\pref{theorem:global_knee} is valid for the problem that has only one global knee point. As discussed in some recent studies (e.g.,~\cite{YuJO19a} and~\cite{BhattacharjeeSR16}), it is not uncommon that a problem has more than one knee point, as known as local knee points. In this case, \pref{theorem:global_knee} does not work appropriately for identifying all local knee points. Considering the example shown in~\pref{fig:subregion}(a), both $\mathbf{x}^1$ and $\mathbf{x}^2$ are knee points of the underlying PF. However, since $\mathbb{U}(\mathbf{x}^3,\mathbf{x}^2)=\mathbb{G}(\mathbf{x}^3,\mathbf{x}^2)+\mathbb{D}(\mathbf{x}^3,\mathbf{x}^2)<0$, we have $\mathbf{x}^3\preceq_k\mathbf{x}^2$ whereas $\mathbf{x}^3$ is obviously not a knee point.
\end{remark}

To enable the identification of local knee points, we restrict the comparison of knee-dominance within the neighbourhood of each solution in $\mathrm{S}$. Henceforth, the first (probably the most important) hassle is the setup of an appropriate neighbourhood. In this paper, we borrow the idea from MOEA/D-M2M~\cite{LiuGZ14}, a decomposition-based EMO algorithm, to use a set of evenly distributed weight vectors $\mathrm{W}=\{\mathbf{w}^i\}_{i=1}^{\hat{N}}$, generated by the Das and Dennis' method~\cite{DasD98}, to divide the objective space into several subregions $\Upsilon=\{\Delta^i\}_{i=1}^{\hat{N}}$ where
\begin{equation}
    \Delta^i=\{\mathbf{F}(\mathbf{x})\in\mathbb{R}^m|\langle\mathbf{F}(\mathbf{x}),\mathbf{w}^i\rangle\leq\langle\mathbf{F}(\mathbf{x}),\mathbf{w}^j\rangle\},
\end{equation} 
where $j\in\{1,\cdots,\hat{N}\}$ and $\langle\mathbf{F}(\mathbf{x}),\mathbf{w}\rangle$ is the acute angle between $\mathbf{F}(\mathbf{x})$ and $\mathbf{w}$. In particular, $\hat{N}\leq N$ is the largest possible number of weight vectors that can be generated by the Das and Dennis' method. Based on the closeness of weight vectors, we can build the neighborhood for each subregion.
\begin{definition}\label{def:neighour_subregion}
    For each subregion $\Delta^i$, $i\in\{1,\cdots,\hat{N}\}$, its neighborhood is defined as:
    \begin{equation}
        \Omega^i=\{\Delta^k|k=\argmin_{j\in\{1,\cdots,\hat{N}\}}\langle\mathbf{w}^i,\mathbf{w}^j\rangle\}.
    \end{equation}
\end{definition}
\noindent As the example shown in~\pref{fig:subregion}(b), $\Omega^2=\{\Delta^1,\Delta^2,\Delta^3\}$. After the setup of subregions and their neighborhoods, each solution $\mathbf{x}$ of $\mathrm{S}$ is associated with a unique subregion whose index is determined as:
\begin{equation}
    k=\argmin\limits_{i\in\{1,\cdots,\hat{N}\}}\langle\mathbf{\overline{F}}(\mathbf{x}),\mathbf{w}^i\rangle,
    \label{eq:association}
\end{equation}
where $\mathbf{\overline{F}}(\mathbf{x})$ is the normalized objective vector of $\mathbf{x}$, and its $i$-th objective function is calculated as:
\begin{equation}
    \overline{f}_i(\mathbf{x})=\frac{f_i(\mathbf{x})-z^{\ast}_i}{z^{\max}_i-z_i^{\min}},
\end{equation}
where $i\in\{1,\cdots,m\}$. After associating solutions with subregions, we can build the neighbourhood of each solution.
\begin{definition}\label{def:neighbor_sol}
    For each solution $\mathbf{x}^i$, $i\in\{1,\cdots,N\}$, its neighbourhood is defined as:
    \begin{equation}
        \Psi^i=\{\mathbf{x}|\mathbf{x}\in\Delta\ \text{where}\ \Delta\in\Omega^j\ \text{and}\ \mathbf{x}\neq\mathbf{x}^i\},
    \end{equation}
    where $j$ is the index of the subregion with which $\mathbf{x}^i$ is associated.
\end{definition}

\begin{figure}[htbp]
    \centering
    \subfloat[Counter example of \pref{theorem:global_knee}.]{\includegraphics[width=.4\linewidth]{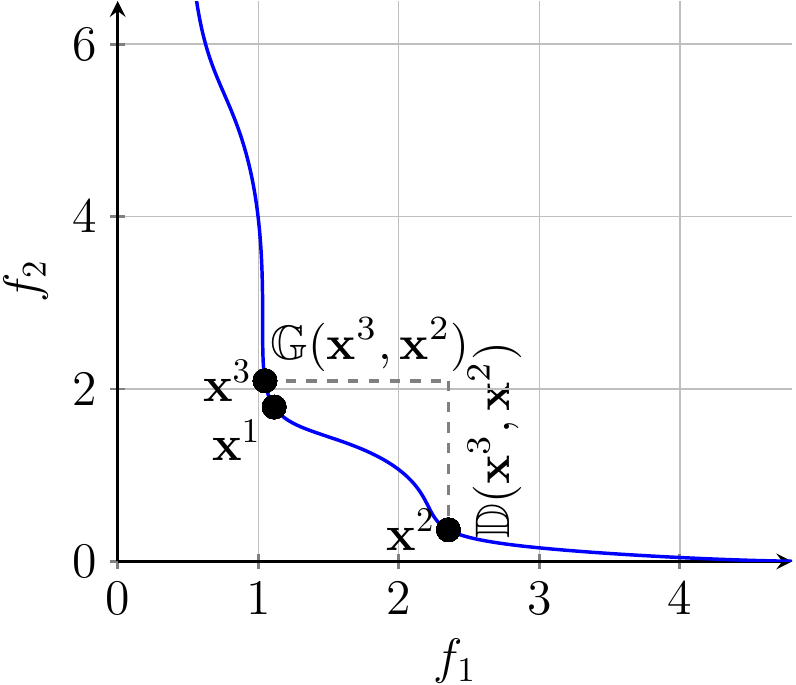}}
    \subfloat[Illustraion of neighborhood.]{\includegraphics[width=.4\linewidth]{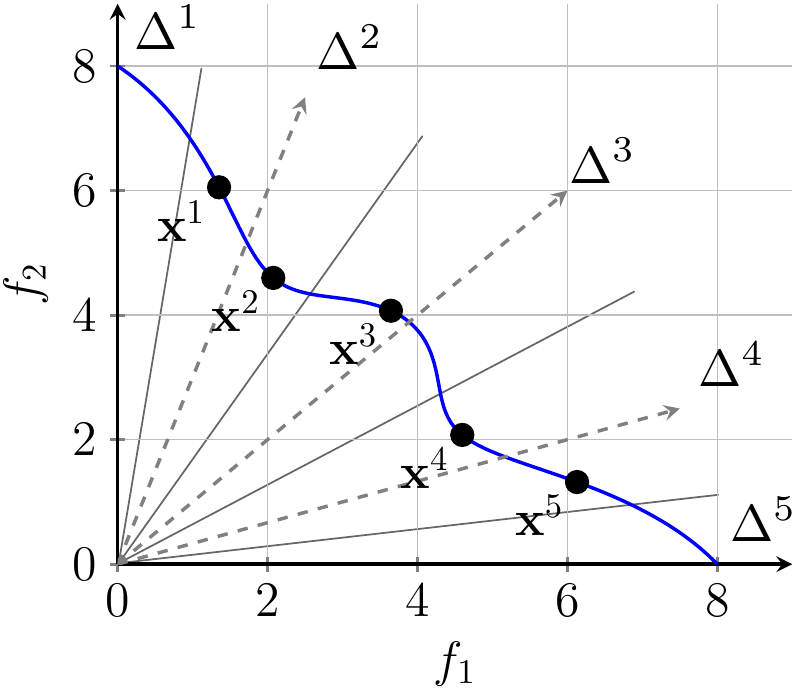}}
    \caption{Illustrative examples of some important concepts.}
    \label{fig:subregion}
\end{figure}

\noindent As for the example shown in~\pref{fig:subregion}(b), $\Psi^2=\{\mathbf{x}^1,\mathbf{x}^3\}$. Based on the definition of neighbourhood of a solution in~\pref{def:neighbor_sol}, we have the following corollaries and theorem to support the identification of local knee point(s).
\begin{corollary}
    Given a set of trade-off solution set $\mathrm{S}$, $\mathbf{x}^{\ast}\in\mathrm{S}$ is a local knee point if and only if there does not exist another solution $\mathbf{x}^{\prime}\in\Psi^{\ast}$ that $\mathbf{x}^{\prime}\preceq_k\mathbf{x}^{\ast}$.
    \label{corollary:local_knee}
\end{corollary}
The proof of this corollary can be easily derived from that of~\pref{theorem:global_knee} by restricting the comparison within $\Psi^{\ast}$.
\begin{corollary}
    If there is only one global knee point, it is the local knee point identified by \pref{corollary:local_knee}.
\end{corollary}
\begin{proof}
    This corollary can be proved by contradiction. Let us assume that $\mathbf{x}^{\ast}$ is the global knee point but is not the local knee point identified by \pref{corollary:local_knee}. According to~\pref{def:neighbor_sol}, we can build a neighbourhood $\Psi^{\ast}$ of $\mathbf{x}^{\ast}$. Since $\mathbf{x}^{\ast}$ is not the local knee point, there should exist another solution $\mathbf{x}^{\prime}\in\Psi^{\ast}$ such that $\mathbf{x}^{\prime}\preceq_k\mathbf{x}^{\ast}$. Obviously, this contradicts the assertion that $\mathbf{x}^{\ast}$ is the global knee point.
\end{proof}

\begin{definition}\label{def:accumulative}
    If there is a set of local knee points $\mathrm{K}=\{\mathbf{x}^1_k,\cdots,\mathbf{x}^{|K|}_k\}$, the accumulative trade-off utility of a local knee point $\mathbf{x}^i_k$, $i\in\{1,\cdots,|K|\}$, is defined as:
    \begin{equation}
        \mathbb{K}(\mathbf{x}^i_k)=\sum_{j=1,j\neq i}^{|K|}\mathbb{U}(\mathbf{x}^i_k,\mathbf{x}^j_k).
    \end{equation}
\end{definition}

\begin{lemma}\label{lemma:accumulative}
    Given a pair of local knee points $\mathbf{x}^i_k$ and $\mathbf{x}^j_k$ from a set of local knee points $\mathrm{K}=\{\mathbf{x}^1_k,\cdots,\mathbf{x}^{|K|}_k\}$, if $\mathbf{x}^i_k\preceq_k\mathbf{x}^j_k$, then we have $\mathbb{K}(\mathbf{x}^i_k)<\mathbb{K}(\mathbf{x}^j_k)$.
\end{lemma}
\begin{proof}
    To facilitate the notation, let us first prove this lemma for $\mathbf{x}^1_k$ and $\mathbf{x}^2_k$. According to \pref{def:accumulative}, we have:
    \begin{align*}
        \mathbb{K}(\mathbf{x}^1_k) &= \sum_{i=2}^{|K|}\mathbb{U}(\mathbf{x}^1_k,\mathbf{x}^i_k)
                                   = \sum_{i=2}^{|K|}\bigg\{\sum_{j=1}^m[\frac{f_j(\mathbf{x}^1_k)-f_j(\mathbf{x}^i_k)}{z^{\max}_j-z_j^{\min}}]\bigg\}\\
                                   &= \sum_{j=1}^m[\frac{f_j(\mathbf{x}^1_k)-f_j(\mathbf{x}^2_k)}{z^{\max}_j-z_j^{\min}}]+\cdots+\sum_{j=1}^m[\frac{f_j(\mathbf{x}^1_k)-f_j(\mathbf{x}^{|K|}_k)}{z^{\max}_j-z_j^{\min}}],
    \end{align*}
\begin{align*}
        \mathbb{K}(\mathbf{x}^2_k) &= \sum_{i=1}^{|K|}\mathbb{U}(\mathbf{x}^2_k,\mathbf{x}^i_k)
                                   = \sum_{i=1}^{|K|}\bigg\{\sum_{j=1}^m[\frac{f_j(\mathbf{x}^2_k)-f_j(\mathbf{x}^i_k)}{z^{\max}_j-z_j^{\min}}]\bigg\}\\
                                   &= \sum_{j=1}^m[\frac{f_j(\mathbf{x}^2_k)-f_j(\mathbf{x}^1_k)}{z^{\max}_j-z_j^{\min}}]+\cdots+\sum_{j=1}^m[\frac{f_j(\mathbf{x}^2_k)-f_j(\mathbf{x}^{|K|}_k)}{z^{\max}_j-z_j^{\min}}].
    \end{align*}
    Then we can derive
\begin{align*}
    \mathbb{K}(\mathbf{x}^1_k)-\mathbb{K}(\mathbf{x}^2_k)&=|K|\sum_{j=1}^m\frac{f_j(\mathbf{x}^1_k)}{z^{\max}_j-z_j^{\min}}-|K|\sum_{j=1}^m\frac{f_j(\mathbf{x}^2_k)}{z^{\max}_j-z_j^{\min}}\\
                                   &= |K|\bigg\{\sum_{j=1}^m[\frac{f_j(\mathbf{x}^1_k)-f_j(\mathbf{x}^2_k)}{z^{\max}_j-z_j^{\min}}]\bigg\}.
\end{align*}
Since $\mathbf{x}^1_k\preceq_k\mathbf{x}^2_k$, we have $\mathbb{U}(\mathbf{x}^1_k,\mathbf{x}^2_k)=\sum_{i=1}^m[f_i(\mathbf{x}^1_k)-f_i(\mathbf{x}^2_k)]<0$. Accordingly, $\mathbb{K}(\mathbf{x}^1_k)-\mathbb{K}(\mathbf{x}^2_k)<0$. It is not difficult to envisage that this conclusion can be generalised to any pair of local knee points.
\end{proof}

\begin{theorem}\label{theorem:best_knee}
    Based on~\pref{lemma:accumulative}, given a set of local knee points $\mathrm{K}=\{\mathbf{x}^1_k,\cdots,\mathbf{x}^{|K|}_k\}$, the best one is $\argmin_{\mathbf{x}^i_k\in\mathrm{K}}(\mathbb{K}(\mathbf{x}^i_k))$.
\end{theorem}

\begin{remark}
    \pref{theorem:best_knee} is essentially equivalent to \pref{theorem:global_knee}. Besides, it also provides the foundation to prioritise the importance of different local knee points. Specifically, local knee points can be sorted according to their accumulative trade-off utility. In particular, they are incomparable if and only if they have the same accumulative trade-off utility.
\end{remark}


\begin{algorithm}[t]
    \caption{$\mathtt{KPITU}(\mathrm{S})$: main procedure of keen point identification based on trade-off utility}
    \label{alg:KPITU}
    \KwIn{Trade-off solution set $\mathrm{S}=\{\mathbf{x}^i\}_{i=1}^N$}
    \KwOut{Local knee point set $\mathrm{K}$}

    $\mathrm{K}\leftarrow\emptyset$;\\
    Generate a set of weight vectors $\mathrm{W}\leftarrow\{\mathbf{w}^i\}_{i=1}^{\hat{N}}$;\\
    $\Psi\leftarrow\mathtt{NeighbourSearch}(\mathrm{S},\mathrm{W})$;\\
    \For{$i\leftarrow 1$ \KwTo $N$}{
        $flag\leftarrow 1$;\\
        \ForEach{$\mathbf{x}\in\Psi^i$}{
            \If{$\mathbb{U}(\mathbf{x}^i,\mathbf{x})>1$}{
                $flag\leftarrow 0$; \textbf{break};
            } 
        }
        \If{$flag==1$}{
            $\mathrm{K}\leftarrow\mathrm{K}\bigcup\{\mathbf{x}^i\}$;
        }
    }
    $\mathrm{K}\leftarrow\mathtt{Sort}(\mathrm{K})$;\\
    \Return $\mathrm{K}$
\end{algorithm}

\subsection{Algorithmic Implementation of KPITU}
\label{sec:algorithm}

Based on the foundation discussed in~\pref{sec:TOU}, this section provides the algorithmic implementation of our proposed KPITU step by step. The pseudo code of the main procedure of KPITU is given in~\pref{alg:KPITU}. Some important parts are discussed further in the following paragraphs.

\begin{itemize}
    \item The input of KPITU is a set of trade-off solutions $\mathrm{S}$ whilst its output is a set of knee points $\mathrm{K}$ sorted based on their trade-off utility.

    \item As discussed in~\pref{sec:TOU}, the KPI process is restricted to the neighbourhood of each solution in $\mathrm{S}$ in order to find as many knee points as possible. \pref{alg:neighbor} gives the pseudo code of the process for identifying the neighbourhood of each solution in $\mathrm{S}$. In particular, lines 2 to 8 are the algorithmic implementation of \pref{def:neighour_subregion} for identifying the neighbourhood of each subregion; lines 9 to 16 implement the solution association given in~\pref{eq:association} and lines 17 to 21 provide the algorithmic implementation of \pref{def:neighbor_sol} for identifying the neighbourhood of each solution in $\mathrm{S}$.
    \item Lines 4 to 10 of~\pref{alg:KPITU} are the algorithmic implementation of~\pref{corollary:local_knee}.
    \item If our KPITU identifies more than one local knee point, it is arguable that these knee points in $\mathrm{K}$ are of equal importance given the difference of their trade-off utility. Based on~\pref{theorem:best_knee}, \pref{alg:sort} gives the procedure that sorts the local knee points in $\mathrm{K}$ according to their accumulative trade-off utility.

    \item In KPITU, there are three major parts. The first one is the process to identify the neighbourhood of each solution in $\mathrm{S}$ given in~\pref{alg:neighbor}. Specifically, lines 2 to 8 of~\pref{alg:neighbor} cost at most $\hat{N}^2$ comparisons; lines 9 to 14 cost at most $N\times\hat{N}$ comparisons; and lines 15 to 19 cost at most $N\times\hat{N}\times|\Delta|$ comparisons. Since $\hat{N}\leq N$, the time complexity of $\mathtt{NeighbourSearch}(\mathrm{S},\mathrm{W})$ is $\mathcal{O}(N^2)$. As for the main iteration of $\mathtt{KPITU(S)}$ shown in~\pref{alg:KPITU}, lines 4 to 10 cost $N\times|\Psi^i|$ comparisons. As for the procedure $\mathtt{Sort}(\mathrm{K})$, the calculation of accumulative trade-off utility requires $\mathcal{O}(|\mathrm{K}|^2)$ computations whilst the worst case complexity of sorting $\mathrm{K}$ is $\mathcal{O}(|\mathrm{K}|\log|\mathrm{K}|)$. In summary, the time complexity of KPITU is $\mathcal{O}(N^2)$.
\end{itemize}

\begin{algorithm}[t]
    \caption{$\mathtt{NeighbourSearch}(\mathrm{S},\mathrm{W})$: the procedure used to identify the neighbourhood of each solution in $\mathrm{S}$}
    \label{alg:neighbor}
    \KwIn{Trade-off solution set $\mathrm{S}$, weight vector set $\mathrm{W}$}
    \KwOut{Neighborhood set $\Psi$}
    $\Psi\leftarrow\emptyset$;\\
    \For{$i\leftarrow 1$ \KwTo $\hat{N}$}{
        $\Omega^i\leftarrow\emptyset$; $d_{\min}\leftarrow\infty$;\\
        \For{$j\leftarrow 1$ \KwTo $\hat{N}$ $\&\&$ $j\neq i$}{
            \uIf{$\langle\mathbf{w}^i,\mathbf{w}^j\rangle<d_{\min}$}{
                $d_{\min}\leftarrow\langle\mathbf{w}^i,\mathbf{w}^j\rangle$; $\Omega^i\leftarrow\emptyset$; $\Omega^i\leftarrow\{j\}$;
            }\ElseIf{$\langle\mathbf{w}^i,\mathbf{w}^j\rangle==d_{\min}$}{
                $\Omega^i\leftarrow\Omega^i\bigcup\{j\}$;
            }
        }
    }
    \For{$i\leftarrow 1$ \KwTo $\hat{N}$}{
    	$\Delta^i\leftarrow\emptyset$;
    }
    \For{$i\leftarrow 1$ \KwTo $N$}{
        $d_{\min}\leftarrow\infty$;\\
        \For{$j\leftarrow 1$ \KwTo $\hat{N}$}{
            \If{$\langle\overline{\mathbf{F}}(\mathbf{x}^i),\mathbf{w}^j\rangle<d_{\min}$}{
                $d_{\min}\leftarrow\langle\overline{\mathbf{F}}(\mathbf{x}^i),\mathbf{w}^j\rangle$; $index\leftarrow j$;
            }
        }
        $\Delta^{index}\leftarrow\Delta^{index}\bigcup\{\mathbf{x}^i\}$;
    }
    \For{$i\leftarrow 1$ \KwTo $N$}{
        \For{$j\leftarrow 1$ \KwTo $\hat{N}$}{
            \If{$\mathbf{x}^i\in\Delta^j$}{
                \ForEach{$\Delta\in\Omega^j$}{
                    $\Psi^i\leftarrow\{\mathbf{x}|\mathbf{x}\in\Delta\ \text{where } \mathbf{x}\neq\mathbf{x}^i\}$;
                }
            }
        }
    }
    \Return $\Psi\leftarrow\{\Psi^1,\cdots,\Psi^N\}$
\end{algorithm}

\begin{algorithm}[t]
    \caption{$\mathtt{Sort}(\mathrm{K})$: the procedure used to sort the knee points in $\mathrm{K}$ according to their accumulative trade-off utility}
    \label{alg:sort}
    \KwIn{Knee point set $\mathrm{K}$}
    \KwOut{Sorted knee point set $\mathrm{K}$}
    \For{$i\leftarrow 1$ \KwTo $|\mathrm{K}|$}{
        $\mathbb{K}(\mathbf{x}^i)\leftarrow 0$;\\
        \ForEach{$\mathbf{x}\in\Psi^i$ $\&\&$ $\mathbf{x}\neq\mathbf{x}^i$}{
            $\mathbb{K}(\mathbf{x}^i)\leftarrow\mathbb{K}(\mathbf{x}^i)+\mathbb{U}(\mathbf{x}^i,\mathbf{x})$;
        }
    }
    Sort $\mathrm{K}$ in descending order of $\mathbb{K}(\mathbf{x})$ where $\mathbf{x}\in\mathrm{K}$;\\
    \Return Sorted $\mathrm{K}$
\end{algorithm}

\subsection{Parallel Implementation of KPITU}
\label{sec:GPU}

According to the time complexity analysis discussed in~\pref{sec:algorithm}, we find that the most time consuming parts of KPITU are $\mathtt{NeighbourSearch}(\mathrm{S},\mathrm{W})$ and the implementation of~\pref{corollary:local_knee} (lines 4 to 10), the complexity of which are quadratic of $N$.

In principle, both the search of neighbourhood for each candidate solution (i.e., $\mathtt{NeighbourSearch}(\mathrm{S},\mathrm{W})$) and the identification of knee point among each neighbourhood (i.e., lines 4 to 10) are implemented in a sequential manner. They require comparisons of either distances or trade-off utility of each solution with others in $\mathrm{S}$. However, it is not difficult to see that such comparisons are independent from different solutions of $\mathrm{S}$. Thus, it is a natural idea to parallelise these two processes across different solutions of $\mathrm{S}$. Accordingly, this reduces the time complexity of KPITU to be linear, i.e., $\mathcal{O}(N)$. Bearing this consideration in mind, we implement KPITU under a GPU environment (NVIDIA GeForce RTX 2080Ti, 11GB GDDR6) where each computing unit carries out each for-loop of~\pref{alg:KPITU}.


\section{Empirical Studies}
\label{sec:experiments}

To validate the effectiveness of our proposed KPITU, we test and compare its performance on benchmark problems with two to ten objectives against five state-of-the-art KPI methods. Before we present the empirical results, we delineate the experimental setup including the peer algorithms, the benchmark problems and the performance indicator.

\subsection{Experimental Setup}
\label{sec:setup}

\subsubsection{Peer Algorithms}
\label{sec:algorithms}

According to the categorisation in~\pref{sec:literature}, we choose five state-of-the-art KPI methods, including cone-domination (CD)~\cite{Ramirez-Atencia17}, EMU$^r$~\cite{BhattacharjeeSR17}, reflex angle (RA)~\cite{BrankeDDO04}, CHIM~\cite{Das99} and MMD~\cite{ChiuYJ16}, from each of those four classes to compare with our proposed KPITU. Note that the dominance angle $\varphi$ of CD is set to 135 and the number of weight vectors used in EMU$^r$ is set to be $1/6$ of the cardinality of the underlying solution set.

\subsubsection{Benchmark Test Problems}
\label{sec:problems}

All benchmark test problems proposed in this literature for testing KPI, including DO2DK, DEB2DK and DEB3DK~\cite{BrankeDDO04}, CKP~\cite{YuJO18} and PMOP~\cite{YuJO19a}, are used in our experiments. DO2DK, DEB2DK, and CKP only have two objectives; DEB3DK is a three-objective benchmark test problem whilst PMOP (i.e., PMOP1 to PMOP14) is scalable to higher dimensional cases. In particular, the number of objectives is set as $m\in\{2,3,5,8,10\}$ in our experiments. Moreover, DO2DK, DEB2DK and DEB3DK have a convex PF shape; CKP is designed to have a concave PF whilst the shape of PMOP is tuneable to be linear, concave and convex. Furthermore, the number of knee points is controllable in our experiments where we consider the scenarios of having only one global knee point and with $2^{m-1}$ local knee points as suggested in~\cite{YuJO19a}. More detailed information of these benchmark test problems can be found from their corresponding references (i.e., \cite{BrankeDDO04,YuJO18} and \cite{YuJO19a}). Since all benchmark test problems have analytical forms, we sample a set of representative points from the PF of the corresponding problem formulation to setup the baseline for KPI, denoted as $\mathrm{P}$. In particular, $|\mathrm{P}|$ is set to 200,676,$4,096$,$5,000$,$8,000$ and $10,000$ for different settings of $m$\footnote{The data sets along with the source codes are downloable from our project page \url{https://github.com/COLA-Laboratory/kpi}.}.

\subsubsection{Performance Indicator}
\label{sec:indicator}

In addition to a visual comparison of different KPI methods, which is widely used in the KPI literature, we use the following performance indicator $\mathbb{I}(\mathrm{S})$ to have a quantitative evaluation:
\begin{equation}
    \mathbb{I}(\mathrm{S})=\frac{1}{|\mathrm{S}|}\sum_{i=1}^{|\mathrm{S}|}dist(\mathbf{x}^i,\mathrm{S}^{\ast}),
\end{equation}
where $\mathrm{S}$ is the set of knee points identified by a KPI method, $\mathrm{S}^{\ast}$ is the set of true knee points of the underlying benchmark test problem and $dist(\mathbf{x}^i,\mathrm{S}^{\ast})$ is the Euclidean distance between $\mathbf{x}^i\in\mathrm{S}$ and its nearest neighbour in $\mathrm{S}^{\ast}$. Essentially, $\mathbb{I}(\mathrm{S})$ is similar to the IGD metric~\cite{BosmanT03} widely used in the EMO literature~\cite{LiZZL09,LiZLZL09,LiKWCR12,LiKCLZS12,CaoKWL12,LiKWTM13,LiK14,CaoKWL14,WuKZLWL15,LiKZD15,LiKD15,LiDZ15,LiODY16,LiDZZ17,WuKJLZ17,WuLKZZ17,LiWKC13,ChenLY18,KumarBCLB18,WuLKZ20,ChenLBY18,LiCFY19,WuLKZZ19,LiuLC19,LiXT19,ZouJYZZL19,BillingsleyLMMG19,GaoNL19}. In particular, $\mathbb{I}(S)=0$ if and only if the corresponding KPI method perfectly identifies all knee points of the underlying benchmark test problem. It is worth noting that each KPI method is only required to run once since all of them are deterministic method.

\subsection{Experimental Results}
\label{sec:results}

In this section, we will present and analyse the experimental results of our proposed KPITU against five state-of-the-art KPI methods. In particular, the discussion is separated into three parts according to the characteristics of knee point(s) of the underlying benchmark test problems.

\subsubsection{Results on Problems with Only One Global Knee Point}
\label{sec:global_knee_results}

\begin{table*}[htbp]
    \centering
    \caption{Comparison Results of KPITU with Other Five KPI Methods on Problems with Only One Global Knee Point}
    \resizebox{\textwidth}{!}{
        \begin{tabular}{c|c|cccccc||c|c|cccccc}
            \cline{2-8}\cline{10-16}
            & $m$     & \texttt{KPITU} & \texttt{CD}    & \texttt{EMU$^r$}  & \texttt{RA}    & \texttt{CHIM}  & \texttt{MMD}   &       & $m$     & \texttt{KPITU} & \texttt{CD}    & \texttt{EMU$^r$}  & \texttt{RA}    & \texttt{CHIM}  & \texttt{MMD} \\
            \hline
            \texttt{DO2DK} & 2     & \cellcolor[rgb]{ .749,  .749,  .749}\textbf{2.210E-2} & \cellcolor[rgb]{ .749,  .749,  .749}\textbf{2.210E-2} & 2.281E+0 & 1.948E-1 & \cellcolor[rgb]{ .749,  .749,  .749}\textbf{2.210E-2} & \cellcolor[rgb]{ .749,  .749,  .749}\textbf{2.210E-2} & \texttt{CKP}   & 2     & \cellcolor[rgb]{ .749,  .749,  .749}\textbf{0.000E+0} & 2.009E-1 & \cellcolor[rgb]{ .749,  .749,  .749}\textbf{0.000E+0} & 1.670E-1 & \cellcolor[rgb]{ .749,  .749,  .749}\textbf{0.000E+0} & \cellcolor[rgb]{ .749,  .749,  .749}\textbf{0.000E+0} \\
            \hline
            \hline
            \texttt{DEB2DK} & 2     & \cellcolor[rgb]{ .749,  .749,  .749}\textbf{0.000E+0} & 1.590E-1 & \cellcolor[rgb]{ .749,  .749,  .749}\textbf{0.000E+0} & \cellcolor[rgb]{ .749,  .749,  .749}\textbf{0.000E+0} & \cellcolor[rgb]{ .749,  .749,  .749}\textbf{0.000E+0} & \cellcolor[rgb]{ .749,  .749,  .749}\textbf{0.000E+0} & \texttt{DEB3DK} & 3     & \cellcolor[rgb]{ .749,  .749,  .749}\textbf{0.000E+0} & 1.453E-1 & \cellcolor[rgb]{ .749,  .749,  .749}\textbf{0.000E+0} & \diagbox{}{} & 1.299E-1 & 1.299E-1 \\
            \hline
            \hline
            \multirow{5}[1]{*}{\texttt{PMOP1}} & 2     & \cellcolor[rgb]{ .749,  .749,  .749}\textbf{0.000E+0} & 4.840E-2 & \cellcolor[rgb]{ .749,  .749,  .749}\textbf{0.000E+0} & \cellcolor[rgb]{ .749,  .749,  .749}\textbf{0.000E+0} & \cellcolor[rgb]{ .749,  .749,  .749}\textbf{0.000E+0} & \cellcolor[rgb]{ .749,  .749,  .749}\textbf{0.000E+0} & \multirow{5}[1]{*}{\texttt{PMOP7}} & 2     & \cellcolor[rgb]{ .749,  .749,  .749}\textbf{0.000E+0} & 2.320E+0 & 2.365E+0 & \cellcolor[rgb]{ .749,  .749,  .749}\textbf{0.000E+0} & \cellcolor[rgb]{ .749,  .749,  .749}\textbf{0.000E+0} & \cellcolor[rgb]{ .749,  .749,  .749}\textbf{0.000E+0} \\
            & 3     & \cellcolor[rgb]{ .749,  .749,  .749}\textbf{0.000E+0} & 7.380E-2 & 8.580E-2 & \diagbox{}{} & 4.460E-2 & 4.460E-2 &       & 3     & \cellcolor[rgb]{ .749,  .749,  .749}\textbf{1.059E-1} & 5.879E-1 & 5.476E-1 & \diagbox{}{} & 1.222E-1 & 1.222E-1 \\
            & 5     & \cellcolor[rgb]{ .749,  .749,  .749}\textbf{3.120E-2} & 2.967E-1 & \cellcolor[rgb]{ .749,  .749,  .749}\textbf{3.120E-2} & \diagbox{}{} & \cellcolor[rgb]{ .749,  .749,  .749}\textbf{3.120E-2} & \cellcolor[rgb]{ .749,  .749,  .749}\textbf{3.120E-2} &       & 5     & \cellcolor[rgb]{ .749,  .749,  .749}\textbf{6.504E-1} & 9.780E-1 & 8.582E-1 & \diagbox{}{} & \cellcolor[rgb]{ .749,  .749,  .749}\textbf{6.504E-1} & \cellcolor[rgb]{ .749,  .749,  .749}\textbf{6.504E-1} \\
            & 8     & \cellcolor[rgb]{ .749,  .749,  .749}\textbf{4.177E-1} & 4.366E+0 & 9.628E+0 & \diagbox{}{} & \cellcolor[rgb]{ .749,  .749,  .749}\textbf{4.177E-1} & \cellcolor[rgb]{ .749,  .749,  .749}\textbf{4.177E-1} &       & 8     & \cellcolor[rgb]{ .749,  .749,  .749}\textbf{1.925E-1} & 6.218E-1 & 5.095E-1 & \diagbox{}{} & \cellcolor[rgb]{ .749,  .749,  .749}\textbf{1.925E-1} & \cellcolor[rgb]{ .749,  .749,  .749}\textbf{1.925E-1} \\
            & 10    & \cellcolor[rgb]{ .749,  .749,  .749}\textbf{3.026E+0} & 6.471E+0 & 8.964E+0 & \diagbox{}{} & \cellcolor[rgb]{ .749,  .749,  .749}\textbf{3.026E+0} & \cellcolor[rgb]{ .749,  .749,  .749}\textbf{3.026E+0} &       & 10    & \cellcolor[rgb]{ .749,  .749,  .749}\textbf{2.361E-1} & 8.687E-1 & 9.214E-1 & \diagbox{}{} & \cellcolor[rgb]{ .749,  .749,  .749}\textbf{2.361E-1} & \cellcolor[rgb]{ .749,  .749,  .749}\textbf{2.361E-1} \\
            \hline\hline
            \multirow{5}[1]{*}{\texttt{PMOP2}} & 2     & 1.203E+0 & \cellcolor[rgb]{ .749,  .749,  .749}\textbf{1.660E-1} & 3.453E-1 & 4.092E-1 & 1.203E+0 & 1.203E+0 & \multirow{5}[1]{*}{\texttt{PMOP8}} & 2     & \cellcolor[rgb]{ .749,  .749,  .749}\textbf{0.000E+0} & 1.123E-1 & \cellcolor[rgb]{ .749,  .749,  .749}\textbf{0.000E+0} & 4.960E-1 & \cellcolor[rgb]{ .749,  .749,  .749}\textbf{0.000E+0} & \cellcolor[rgb]{ .749,  .749,  .749}\textbf{0.000E+0} \\
            & 3     & 8.481E-1 & \cellcolor[rgb]{ .749,  .749,  .749}\textbf{3.553E-1} & 9.514E-1 & \diagbox{}{} & 8.996E-1 & 8.996E-1 &       & 3     & \cellcolor[rgb]{ .749,  .749,  .749}\textbf{6.950E-2} & 9.530E-2 & 8.410E-2 & \diagbox{}{} & \cellcolor[rgb]{ .749,  .749,  .749}\textbf{6.950E-2} & \cellcolor[rgb]{ .749,  .749,  .749}\textbf{6.950E-2} \\
            & 5     & \cellcolor[rgb]{ .749,  .749,  .749}\textbf{4.719E-1} & 7.060E-1 & 7.821E-1 & \diagbox{}{} & 6.946E-1 & 6.946E-1 &       & 5     & \cellcolor[rgb]{ .749,  .749,  .749}\textbf{4.426E-1} & \cellcolor[rgb]{ .749,  .749,  .749}\textbf{4.426E-1} & 6.651E-1 & \diagbox{}{} & \cellcolor[rgb]{ .749,  .749,  .749}\textbf{4.426E-1} & \cellcolor[rgb]{ .749,  .749,  .749}\textbf{4.426E-1} \\
            & 8     & 5.992E-1 & \cellcolor[rgb]{ .749,  .749,  .749}\textbf{5.689E-1} & 7.679E-1 & \diagbox{}{} & 5.992E-1 & 5.992E-1 &       & 8     & 9.143E-1 & \cellcolor[rgb]{ .749,  .749,  .749}\textbf{6.319E-1} & 9.076E-1 & \diagbox{}{} & 9.143E-1 & 9.143E-1 \\
            & 10    & \cellcolor[rgb]{ .749,  .749,  .749}\textbf{5.772E-1} & 8.219E-1 & 7.682E-1 & \diagbox{}{} & \cellcolor[rgb]{ .749,  .749,  .749}\textbf{5.772E-1} & \cellcolor[rgb]{ .749,  .749,  .749}\textbf{5.772E-1} &       & 10    & \cellcolor[rgb]{ .749,  .749,  .749}\textbf{2.383E+0} & \cellcolor[rgb]{ .749,  .749,  .749}\textbf{2.383E+0} & \cellcolor[rgb]{ .749,  .749,  .749}\textbf{2.383E+0} & \diagbox{}{} & \cellcolor[rgb]{ .749,  .749,  .749}\textbf{2.383E+0} & \cellcolor[rgb]{ .749,  .749,  .749}\textbf{2.383E+0} \\
            \hline\hline
            \multirow{5}[1]{*}{\texttt{PMOP3}} & 2     & \cellcolor[rgb]{ .749,  .749,  .749}\textbf{0.000E+0} & 8.952E-1 & 2.062E+0 & \cellcolor[rgb]{ .749,  .749,  .749}\textbf{0.000E+0} & \cellcolor[rgb]{ .749,  .749,  .749}\textbf{0.000E+0} & \cellcolor[rgb]{ .749,  .749,  .749}\textbf{0.000E+0} & \multirow{5}[1]{*}{\texttt{PMOP9}} & 2     & \cellcolor[rgb]{ .749,  .749,  .749}\textbf{0.000E+0} & \cellcolor[rgb]{ .749,  .749,  .749}\textbf{0.000E+0} & \cellcolor[rgb]{ .749,  .749,  .749}\textbf{0.000E+0} & 6.800E-3 & \cellcolor[rgb]{ .749,  .749,  .749}\textbf{0.000E+0} & \cellcolor[rgb]{ .749,  .749,  .749}\textbf{0.000E+0} \\
            & 3     & \cellcolor[rgb]{ .749,  .749,  .749}\textbf{1.558E-1} & 9.505E-1 & 6.805E-1 & \diagbox{}{} & \cellcolor[rgb]{ .749,  .749,  .749}\textbf{1.558E-1} & \cellcolor[rgb]{ .749,  .749,  .749}\textbf{1.558E-1} &       & 3     & \cellcolor[rgb]{ .749,  .749,  .749}\textbf{7.900E-3} & 2.060E-2 & 1.714E-1 & \diagbox{}{} & \cellcolor[rgb]{ .749,  .749,  .749}\textbf{7.900E-3} & \cellcolor[rgb]{ .749,  .749,  .749}\textbf{7.900E-3} \\
            & 5     & 5.310E-1 & \cellcolor[rgb]{ .749,  .749,  .749}\textbf{4.693E-1} & 5.703E-1 & \diagbox{}{} & 5.310E-1 & 5.310E-1 &       & 5     & \cellcolor[rgb]{ .749,  .749,  .749}\textbf{6.240E-2} & 8.790E-2 & 7.085E-1 & \diagbox{}{} & \cellcolor[rgb]{ .749,  .749,  .749}\textbf{6.240E-2} & \cellcolor[rgb]{ .749,  .749,  .749}\textbf{6.240E-2} \\
            & 8     & \cellcolor[rgb]{ .749,  .749,  .749}\textbf{8.200E-1} & 8.258E-1 & 8.216E-1 & \diagbox{}{} & 8.280E-1 & 8.280E-1 &       & 8     & \cellcolor[rgb]{ .749,  .749,  .749}\textbf{4.860E-2} & 5.501E-1 & 6.700E-1 & \diagbox{}{} & \cellcolor[rgb]{ .749,  .749,  .749}\textbf{4.860E-2} & \cellcolor[rgb]{ .749,  .749,  .749}\textbf{4.860E-2} \\
            & 10    & 7.823E-1 & \cellcolor[rgb]{ .749,  .749,  .749}\textbf{7.617E-1} & 2.480E+0 & \diagbox{}{} & 7.823E-1 & 7.823E-1 &       & 10    & \cellcolor[rgb]{ .749,  .749,  .749}\textbf{9.480E-2} & 8.093E-1 & 1.310E+0 & \diagbox{}{} & \cellcolor[rgb]{ .749,  .749,  .749}\textbf{9.480E-2} & \cellcolor[rgb]{ .749,  .749,  .749}\textbf{9.480E-2} \\
            \hline\hline
            \multirow{5}[1]{*}{\texttt{PMOP4}} & 2     & \cellcolor[rgb]{ .749,  .749,  .749}\textbf{0.000E+0} & 1.095E+0 & \cellcolor[rgb]{ .749,  .749,  .749}\textbf{0.000E+0} & 3.083E-1 & \cellcolor[rgb]{ .749,  .749,  .749}\textbf{0.000E+0} & \cellcolor[rgb]{ .749,  .749,  .749}\textbf{0.000E+0} & \multirow{5}[1]{*}{\texttt{PMOP11}} & 2     & 2.000E-2 & \cellcolor[rgb]{ .749,  .749,  .749}\textbf{0.000E+0} & 2.930E-2 & 8.440E-2 & 2.000E-2 & 2.000E-2 \\
            & 3     & \cellcolor[rgb]{ .749,  .749,  .749}\textbf{0.000E+0} & \cellcolor[rgb]{ .749,  .749,  .749}\textbf{0.000E+0} & \cellcolor[rgb]{ .749,  .749,  .749}\textbf{0.000E+0} & \diagbox{}{} & \cellcolor[rgb]{ .749,  .749,  .749}\textbf{0.000E+0} & \cellcolor[rgb]{ .749,  .749,  .749}\textbf{0.000E+0} &       & 3     & \cellcolor[rgb]{ .749,  .749,  .749}\textbf{2.080E-2} & 6.043E-1 & 3.030E-2 & \diagbox{}{} & \cellcolor[rgb]{ .749,  .749,  .749}\textbf{2.080E-2} & \cellcolor[rgb]{ .749,  .749,  .749}\textbf{2.080E-2} \\
            & 5     & \cellcolor[rgb]{ .749,  .749,  .749}\textbf{4.833E-1} & 6.754E+0 & 7.918E-1 & \diagbox{}{} & 1.727E+0 & 1.727E+0 &       & 5     & \cellcolor[rgb]{ .749,  .749,  .749}\textbf{8.700E-3} & \cellcolor[rgb]{ .749,  .749,  .749}\textbf{8.700E-3} & \cellcolor[rgb]{ .749,  .749,  .749}\textbf{8.700E-3} & \diagbox{}{} & \cellcolor[rgb]{ .749,  .749,  .749}\textbf{8.700E-3} & \cellcolor[rgb]{ .749,  .749,  .749}\textbf{8.700E-3} \\
            & 8     & \cellcolor[rgb]{ .749,  .749,  .749}\textbf{1.828E+0} & 3.669E+0 & 3.669E+0 & \diagbox{}{} & 3.669E+0 & 3.669E+0 &       & 8     & 1.240E-2 & 1.240E-2 & \cellcolor[rgb]{ .749,  .749,  .749}\textbf{1.300E-3} & \diagbox{}{} & 1.240E-2 & 1.240E-2 \\
            & 10    & 1.903E+0 & \cellcolor[rgb]{ .749,  .749,  .749}\textbf{1.219E+0} & 6.457E+0 & \diagbox{}{} & 6.013E+0 & 6.013E+0 &       & 10    & \cellcolor[rgb]{ .749,  .749,  .749}\textbf{2.000E-4} & \cellcolor[rgb]{ .749,  .749,  .749}\textbf{2.000E-4} & \cellcolor[rgb]{ .749,  .749,  .749}\textbf{2.000E-4} & \diagbox{}{} & 2.900E-3 & 2.900E-3 \\
            \hline\hline
            \multirow{5}[1]{*}{\texttt{PMOP6}} & 2     & \cellcolor[rgb]{ .749,  .749,  .749}\textbf{0.000E+0} & \cellcolor[rgb]{ .749,  .749,  .749}\textbf{0.000E+0} & \cellcolor[rgb]{ .749,  .749,  .749}\textbf{0.000E+0} & \cellcolor[rgb]{ .749,  .749,  .749}\textbf{0.000E+0} & \cellcolor[rgb]{ .749,  .749,  .749}\textbf{0.000E+0} & \cellcolor[rgb]{ .749,  .749,  .749}\textbf{0.000E+0} & \multirow{5}[1]{*}{\texttt{PMOP12}} & 2     & \cellcolor[rgb]{ .749,  .749,  .749}\textbf{0.000E+0} & 1.757E-1 & 8.326E-1 & \cellcolor[rgb]{ .749,  .749,  .749}\textbf{0.000E+0} & \cellcolor[rgb]{ .749,  .749,  .749}\textbf{0.000E+0} & \cellcolor[rgb]{ .749,  .749,  .749}\textbf{0.000E+0} \\
            & 3     & \cellcolor[rgb]{ .749,  .749,  .749}\textbf{4.100E-2} & 3.230E-1 & 6.913E-1 & \diagbox{}{} & \cellcolor[rgb]{ .749,  .749,  .749}\textbf{4.100E-2} & \cellcolor[rgb]{ .749,  .749,  .749}\textbf{4.100E-2} &       & 3     & \cellcolor[rgb]{ .749,  .749,  .749}\textbf{4.070E-2} & 7.250E-2 & 1.036E-1 & \diagbox{}{} & \cellcolor[rgb]{ .749,  .749,  .749}\textbf{4.070E-2} & \cellcolor[rgb]{ .749,  .749,  .749}\textbf{4.070E-2} \\
            & 5     & \cellcolor[rgb]{ .749,  .749,  .749}\textbf{3.290E-2} & \cellcolor[rgb]{ .749,  .749,  .749}\textbf{3.290E-2} & 4.297E-1 & \diagbox{}{} & \cellcolor[rgb]{ .749,  .749,  .749}\textbf{3.290E-2} & \cellcolor[rgb]{ .749,  .749,  .749}\textbf{3.290E-2} &       & 5     & \cellcolor[rgb]{ .749,  .749,  .749}\textbf{2.580E-2} & 3.180E-2 & 6.464E-1 & \diagbox{}{} & 2.770E-2 & 2.770E-2 \\
            & 8     & \cellcolor[rgb]{ .749,  .749,  .749}\textbf{1.199E-1} & 7.373E-1 & 1.831E+0 & \diagbox{}{} & \cellcolor[rgb]{ .749,  .749,  .749}\textbf{1.199E-1} & \cellcolor[rgb]{ .749,  .749,  .749}\textbf{1.199E-1} &       & 8     & 8.000E-3 & 8.500E-3 & \cellcolor[rgb]{ .749,  .749,  .749}\textbf{7.200E-3} & \diagbox{}{} & 8.000E-3 & 8.000E-3 \\
            & 10    & \cellcolor[rgb]{ .749,  .749,  .749}\textbf{5.638E-1} & 5.865E-1 & 2.756E+0 & \diagbox{}{} & \cellcolor[rgb]{ .749,  .749,  .749}\textbf{5.638E-1} & \cellcolor[rgb]{ .749,  .749,  .749}\textbf{5.638E-1} &       & 10    & \cellcolor[rgb]{ .749,  .749,  .749}\textbf{5.000E-3} & 6.400E-3 & 1.040E-2 & \diagbox{}{} & \cellcolor[rgb]{ .749,  .749,  .749}\textbf{5.000E-3} & \cellcolor[rgb]{ .749,  .749,  .749}\textbf{5.000E-3} \\
            \hline
        \end{tabular}}%
        \label{tab:global_knee_results}
\end{table*}

Note that PMOP5 and PMOP10 are not considered in this experiment since they cannot be set to have only one global knee point. In addition, the discussion of experimental results on PMOP13 and PMOP14 is left in~\pref{sec:degenerate_knee_results} due to their degeneration characteristics.

\begin{figure}[htbp]
    \centering
    \includegraphics[width=.4\linewidth]{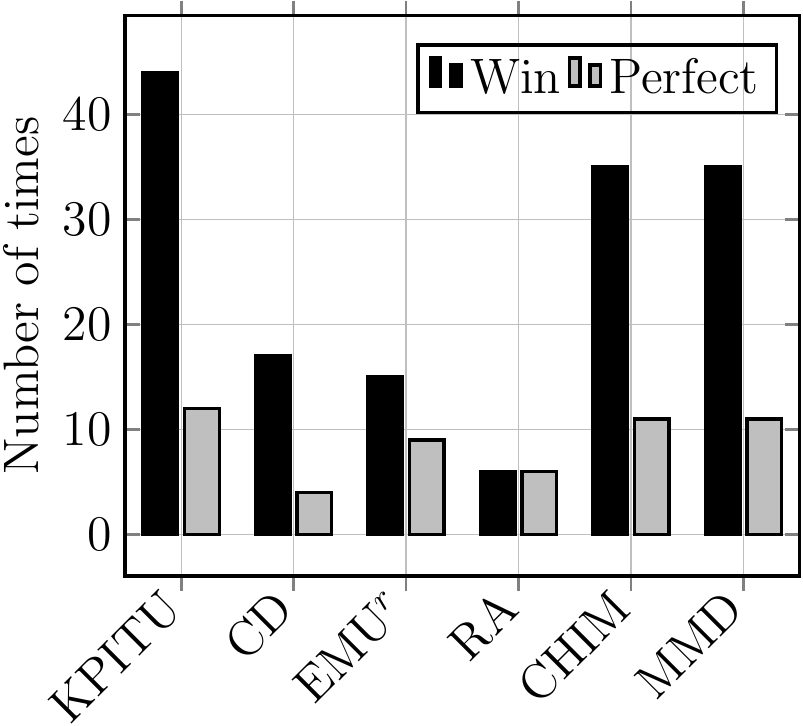}
    \caption{Bar graphs of the number of times that different algorithms obtain the best $\mathbb{I}(\mathrm{S})$ metric values and perfectly identify the global knee point.}
    \label{fig:bar_graph_g}
\end{figure}

According to the comparison results shown in~\pref{tab:global_knee_results} and the bar graphs shown in~\pref{fig:bar_graph_g}, we can clearly see that KPITU, CHIM, MMD are the most competitive algorithms. From some selected examples shown in Figs.~\ref{fig:PMOP7_M2_A1} to~\ref{fig:PMOP9_M10_A2}\footnote{The complete set of plots of population distribution can be found from the supplementary document \url{https://cola-laboratory.github.io/supp-kpi.pdf}.}, we can see that these three KPI methods are able to find a meaningful knee point even for 10-objective problems. In particular, KPITU is the best algorithm that obtains the best $\mathbb{I}(\mathrm{S})$ metric values on 44 out of 54 ($\approx$ 82\%) test problem instances and it perfectly identifies the knee point on 12 test problem instances. On the other hand, RA is the worst algorithm and it is only applicable for the two-objective problems. Its inferior performance might be because the selection angle is determined by neighbouring solutions. Therefore, the performance of RA highly depend on the solution distribution of the given trade-off solution set, i.e., $\mathrm{S}$ in our experiments. The performance of CD is not promising comparing to KPITU, CHIM and MMD. This can be partially attributed to the setting of its control parameter $\varphi$ which is problem dependent. Furthermore, since CD is an extension of the Pareto dominance, its performance degenerates with the increase of the number of objectives. As for EMU$^r$, its performance is similar to CD. This might be because the evaluation of EMU$^r$ depends on the choice of the weight vectors. It is interesting to notice that all KPI methods failed to find the knee point on the PMOP2 problem as an example shown in~\pref{fig:PMOP2_M2_A2}. In particular, we can see that KPITU, CHIM and MMD end up with the singular solution lying on the $f_2$ axis. According to the mechanism of KPITU, this solution should be a knee point as no other solution lying in its neighbourhood.

\begin{figure*}[htbp]
    \includegraphics[width=1.0\linewidth]{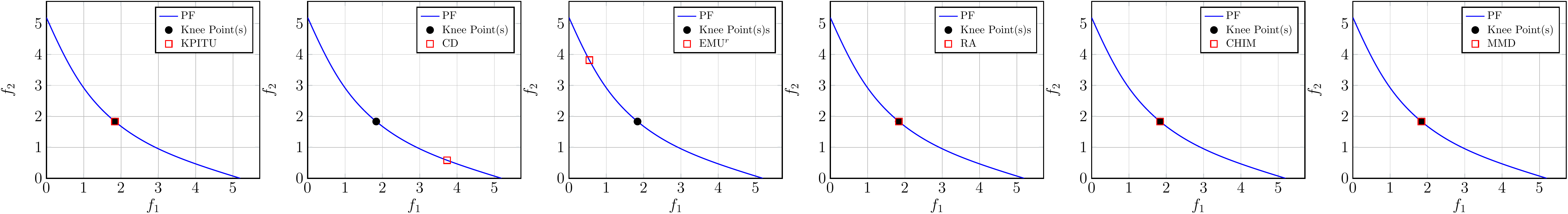}
    \caption{Knee points identified by different KPI methods on 2-objective PMOP7 with one global knee point.}
    \label{fig:PMOP7_M2_A1}
\end{figure*}

\begin{figure*}[htbp]
    \includegraphics[width=1.0\linewidth]{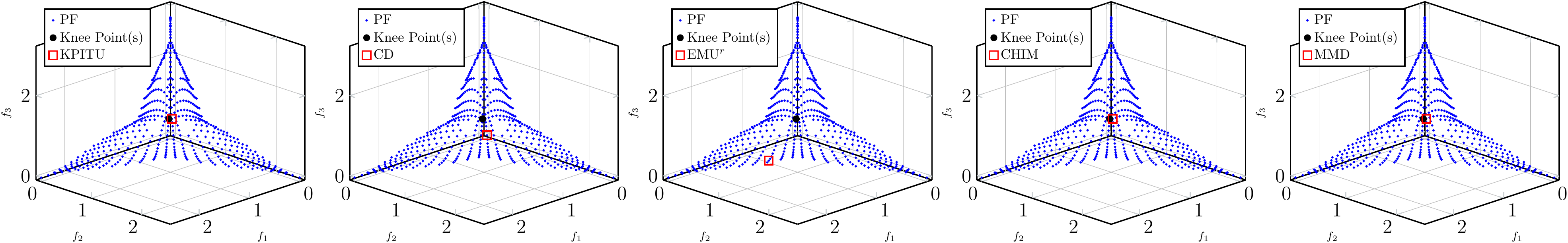}
    \caption{Knee points identified by different KPI methods on 3-objective PMOP6 with one global knee point.}
    \label{fig:PMOP6_M3_A2}
\end{figure*}

\begin{figure*}[htbp]
    \includegraphics[width=1.0\linewidth]{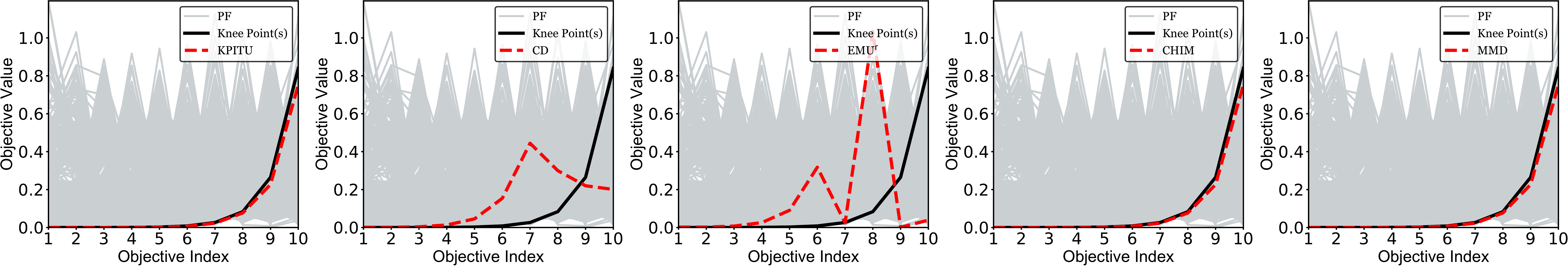}
    \caption{Knee points identified by different KPI methods on 10-objective PMOP9 with one global knee point.}
    \label{fig:PMOP9_M10_A2}
\end{figure*}

\begin{figure*}[htbp]
    \includegraphics[width=1.0\linewidth]{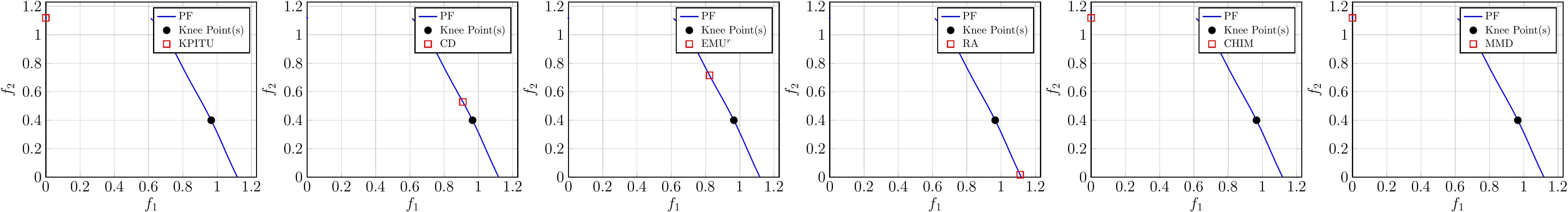}
    \caption{Knee points identified by different KPI methods on 2-objective PMOP2 with one global knee point.}
    \label{fig:PMOP2_M2_A2}
\end{figure*}

It is interesting, even surprising, to note that the performance of CHIM is exactly the same as that of MMD as shown in~\pref{tab:global_knee_results}. Let us analyse the working principles of these two KPI methods as follows. The basic idea of CHIM is to find the solution from $\mathrm{P}$ having the largest perpendicular distance to a hyperplane. Specifically, an affine hyperplane can be represented as:
\begin{equation}
    a_1f_1(\mathbf{x})+a_2f_2(\mathbf{x})+\cdots+a_mf_m(\mathbf{x})=b,
\end{equation}
where $\mathbf{x}\in\mathrm{P}$, $a_i>0$, $i\in\{1,\cdots,m\}$ and $b>0$. The perpendicular distance between $\mathbf{x}$ to the affine hyperplane is calculated as:
\begin{equation}
    \frac{|a_1f_1(\mathbf{x})+a_2f_2(\mathbf{x})+\cdots+a_mf_m(\mathbf{x})-b|}{\sqrt{a_1^2+a_2^2+\cdots+a_m^2}}.
\end{equation}
In practice, the objective values are normalised to $[0,1]$, thus we have $b=1$ and $a_i=1$ for $i\in\{1,\cdots,m\}$. Accordingly, the knee point is determined as:
\begin{equation}\label{eq:new_chim}
    \argmax_{\mathbf{x}\in\mathrm{P}}\frac{|f_1(\mathbf{x})+f_2(\mathbf{x})+\cdots+f_m(\mathbf{x})-1|}{\sqrt{m}}.
\end{equation}
As for MMD, the knee point is identified according to the~\pref{eq:mmd} where the ideal point $\mathbf{z}^{\ast}$ is set to the origin in practice. In this case, the \pref{eq:mmd} can be re-written as:
\begin{equation}\label{eq:new_mmd}
    \argmin_{\mathbf{x}\in S}\|\mathbf{F}(\mathbf{x})\|_1=|f_1(\mathbf{x})+f_2(\mathbf{x})+\cdots+f_m(\mathbf{x})|.
\end{equation}
Comparing equations~(\ref{eq:new_chim}) and~(\ref{eq:new_mmd}), we can infer that they are in effect identical except some constant terms. This explain the same behaviour of CHIM and MMD observed in~\pref{tab:global_knee_results}.

\subsubsection{Results on Problems with Local Knee Points}
\label{sec:local_knee_results}

\begin{table*}[htbp]
    \centering
    \caption{Comparison Results of KPITU with Other Four KPI Methods on Problems with Local Knee Points}
    \resizebox{\textwidth}{!}{
        \begin{tabular}{c|c|cccccc||c|c|cccccc}
            \cline{2-8}\cline{10-16}
            & $m$     & \texttt{KPITU} & \texttt{CD}    & \texttt{EMU$^r$}  & \texttt{RA}    & \texttt{CHIM}  & \texttt{MMD}   &       & $m$     & \texttt{KPITU} & \texttt{CD}    & \texttt{EMU$^r$}  & \texttt{RA}    & \texttt{CHIM}  & \texttt{MMD} \\
            \hline
            \texttt{DO2DK} & 2     & \cellcolor[rgb]{ .749,  .749,  .749}\textbf{1.250E-1} & 7.250E-1 & 4.116E-1 & 1.204E+0 & 1.228E+0 & 1.228E+0 & \texttt{CKP}   & 2     & \cellcolor[rgb]{ .749,  .749,  .749}\textbf{4.210E-1} & 2.643E+0 & 4.690E-1 & 1.625E+0 & 2.474E+0 & 2.474E+0 \\
            \hline\hline
            \texttt{DEB2DK} & 2     & \cellcolor[rgb]{ .749,  .749,  .749}\textbf{0.000E+0} & 1.488E+0 & \multicolumn{1}{r}{2.268E-1} & \multicolumn{1}{r}{1.428E-1} & \cellcolor[rgb]{ .749,  .749,  .749}\textbf{0.000E+0} & \cellcolor[rgb]{ .749,  .749,  .749}\textbf{0.000E+0} & \texttt{DEB3DK} & 3     & \cellcolor[rgb]{ .749,  .749,  .749}\textbf{3.549E-1} & 2.264E+0 & 3.656E-1 & \diagbox{}{} & 1.836E+0 & 1.836E+0 \\
            \hline\hline
            \multirow{5}[1]{*}{\texttt{PMOP1}} & 2     & \cellcolor[rgb]{ .749,  .749,  .749}\textbf{0.000E+0} & 6.580E-2 & 6.940E-2 & 1.050E-2 & \cellcolor[rgb]{ .749,  .749,  .749}\textbf{0.000E+0} & \cellcolor[rgb]{ .749,  .749,  .749}\textbf{0.000E+0} & \multirow{5}[1]{*}{\texttt{PMOP7}} & 2     & \cellcolor[rgb]{ .749,  .749,  .749}\textbf{0.000E+0} & 2.804E-1 & 3.453E-1 & 1.542E+0 & \cellcolor[rgb]{ .749,  .749,  .749}\textbf{0.000E+0} & \cellcolor[rgb]{ .749,  .749,  .749}\textbf{0.000E+0} \\
            & 3     & \cellcolor[rgb]{ .749,  .749,  .749}\textbf{0.000E+0} & 3.887E-1 & 6.650E-2 & \diagbox{}{} & \cellcolor[rgb]{ .749,  .749,  .749}\textbf{0.000E+0} & \cellcolor[rgb]{ .749,  .749,  .749}\textbf{0.000E+0} &       & 3     & \cellcolor[rgb]{ .749,  .749,  .749}\textbf{9.930E-2} & 2.357E-1 & 1.141E+0 & \diagbox{}{} & 9.324E-1 & 9.324E-1 \\
            & 5     & \cellcolor[rgb]{ .749,  .749,  .749}\textbf{2.492E-1} & 6.594E-1 & 9.071E-1 & \diagbox{}{} & 1.261E+0 & 1.261E+0 &       & 5     & \cellcolor[rgb]{ .749,  .749,  .749}\textbf{1.285E-1} & 5.797E-1 & 2.917E-1 & \diagbox{}{} & 8.692E-1 & 8.692E-1 \\
            & 8     & \cellcolor[rgb]{ .749,  .749,  .749}\textbf{6.549E-1} & 9.906E-1 & 2.409E+0 & \diagbox{}{} & 1.113E+0 & 1.113E+0 &       & 8     & \cellcolor[rgb]{ .749,  .749,  .749}\textbf{1.618E-1} & 2.506E-1 & 2.499E-1 & \diagbox{}{} & 6.472E-1 & 6.472E-1 \\
            & 10    & 1.017E+0 & \cellcolor[rgb]{ .749,  .749,  .749}\textbf{8.146E-1} & 2.151E+0 & \diagbox{}{} & 1.273E+0 & 1.273E+0 &       & 10    & \cellcolor[rgb]{ .749,  .749,  .749}\textbf{1.732E-1} & 1.750E-1 & 1.844E-1 & \diagbox{}{} & 2.171E-1 & 2.171E-1 \\
            \hline
            \hline
            \multirow{5}[1]{*}{\texttt{PMOP2}} & 2     & \cellcolor[rgb]{ .749,  .749,  .749}\textbf{4.032E-2} & 5.685E-1 & 1.235E-1 & 3.991E-1 & 4.068E-1 & 4.068E-1 & \multirow{5}[1]{*}{\texttt{PMOP8}} & 2     & \cellcolor[rgb]{ .749,  .749,  .749}\textbf{3.300E-2} & \cellcolor[rgb]{ .749,  .749,  .749}\textbf{3.300E-2} & 4.470E-2 & 6.080E-1 & \cellcolor[rgb]{ .749,  .749,  .749}\textbf{3.300E-2} & \cellcolor[rgb]{ .749,  .749,  .749}\textbf{3.300E-2} \\
            & 3     & \cellcolor[rgb]{ .749,  .749,  .749}\textbf{8.310E-2} & 4.176E-1 & 3.370E-1 & \diagbox{}{} & 8.223E-1 & 8.223E-1 &       & 3     & \cellcolor[rgb]{ .749,  .749,  .749}\textbf{3.921E-1} & 4.426E-1 & 4.456E-1 & \diagbox{}{} & 5.418E-1 & 5.418E-1 \\
            & 5     & \cellcolor[rgb]{ .749,  .749,  .749}\textbf{2.179E-1} & 2.406E-1 & 2.879E-1 & \diagbox{}{} & 3.093E-1 & 3.093E-1 &       & 5     & \cellcolor[rgb]{ .749,  .749,  .749}\textbf{5.988E-1} & 6.083E-1 & 6.196E-1 & \diagbox{}{} & 6.627E-1 & 6.627E-1 \\
            & 8     & 1.720E-1 & 2.055E-1 & 2.108E-1 & \diagbox{}{} & \cellcolor[rgb]{ .749,  .749,  .749}\textbf{1.644E-1} & \cellcolor[rgb]{ .749,  .749,  .749}\textbf{1.644E-1} &       & 8     & 4.416E-1 & \cellcolor[rgb]{ .749,  .749,  .749}\textbf{4.382E-1} & 4.388E-1 & \diagbox{}{} & 4.872E-1 & 4.872E-1 \\
            & 10    & \cellcolor[rgb]{ .749,  .749,  .749}\textbf{1.315E-1} & 1.505E-1 & 1.700E-1 & \diagbox{}{} & 1.544E-1 & 1.544E-1 &       & 10    & \cellcolor[rgb]{ .749,  .749,  .749}\textbf{4.239E-1} & 4.263E-1 & 4.855E-1 & \diagbox{}{} & 4.418E-1 & 4.418E-1 \\
            \hline
            \hline
            \multirow{5}[1]{*}{\texttt{PMOP3}} & 2     & 6.301E-1 & 1.432E+0 & \cellcolor[rgb]{ .749,  .749,  .749}\textbf{2.350E-1} & 4.594E-1 & 6.301E-1 & 6.301E-1 & \multirow{5}[1]{*}{\texttt{PMOP9}} & 2     & 2.820E-1 & 2.711E-1 & \cellcolor[rgb]{ .749,  .749,  .749}\textbf{2.610E-1} & 5.995E-1 & 6.096E-1 & 6.096E-1 \\
            & 3     & \cellcolor[rgb]{ .749,  .749,  .749}\textbf{2.029E-1} & 4.839E-1 & 1.040E+0 & \diagbox{}{} & 1.036E+0 & 1.036E+0 &       & 3     & \cellcolor[rgb]{ .749,  .749,  .749}\textbf{0.000E+0} & 6.310E-2 & 1.286E-1 & \diagbox{}{} & 3.206E-1 & 3.206E-1 \\
            & 5     & 4.636E-1 & 7.512E-1 & \cellcolor[rgb]{ .749,  .749,  .749}\textbf{2.833E-1} & \diagbox{}{} & 8.905E-1 & 8.905E-1 &       & 5     & \cellcolor[rgb]{ .749,  .749,  .749}\textbf{5.930E-2} & 1.656E-1 & 2.604E-1 & \diagbox{}{} & 2.002E-1 & 2.002E-1 \\
            & 8     & 5.023E-1 & \cellcolor[rgb]{ .749,  .749,  .749}\textbf{1.302E-1} & 5.103E-1 & \diagbox{}{} & 5.318E-1 & 5.318E-1 &       & 8     & 7.613E-1 & \cellcolor[rgb]{ .749,  .749,  .749}\textbf{6.262E-1} & 6.887E-1 & \diagbox{}{} & 6.772E-1 & 6.772E-1 \\
            & 10    & \cellcolor[rgb]{ .749,  .749,  .749}\textbf{6.930E-2} & 1.188E-1 & 2.884E-1 & \diagbox{}{} & 3.792E-1 & 3.792E-1 &       & 10    & \cellcolor[rgb]{ .749,  .749,  .749}\textbf{8.439E-1} & 9.118E-1 & 8.875E-1 & \diagbox{}{} & 8.804E-1 & 8.804E-1 \\
            \hline\hline
            \multirow{5}[1]{*}{\texttt{PMOP4}} & 2     & \cellcolor[rgb]{ .749,  .749,  .749}\textbf{0.000E+0} & 2.905E-1 & \cellcolor[rgb]{ .749,  .749,  .749}\textbf{0.000E+0} & 5.921E-1 & 5.445E-1 & 5.445E-1 & \multirow{5}[2]{*}{\texttt{PMOP10}} & 2     & \cellcolor[rgb]{ .749,  .749,  .749}\textbf{0.000E+0} & 2.037E-1 & 1.500E-2 & 7.751E-1 & \cellcolor[rgb]{ .749,  .749,  .749}\textbf{0.000E+0} & \cellcolor[rgb]{ .749,  .749,  .749}\textbf{0.000E+0} \\
            & 3     & \cellcolor[rgb]{ .749,  .749,  .749}\textbf{3.110E-2} & 7.777E-1 & 3.905E-1 & \diagbox{}{} & 9.645E-1 & 9.645E-1 &       & 3     & \cellcolor[rgb]{ .749,  .749,  .749}\textbf{3.480E-2} & 8.820E-2 & 5.140E-2 & \diagbox{}{} & 1.342E-1 & 1.342E-1 \\
            & 5     & \cellcolor[rgb]{ .749,  .749,  .749}\textbf{3.705E-1} & 6.180E-1 & 1.203E+0 & \diagbox{}{} & 1.293E+0 & 1.293E+0 &       & 5     & \cellcolor[rgb]{ .749,  .749,  .749}\textbf{1.194E-1} & 1.730E-1 & 1.730E-1 & \diagbox{}{} & 1.531E-1 & 1.531E-1 \\
            & 8     & 1.347E+0 & 2.099E+0 & 1.374E+0 & \diagbox{}{} & \cellcolor[rgb]{ .749,  .749,  .749}\textbf{1.180E+0} & \cellcolor[rgb]{ .749,  .749,  .749}\textbf{1.180E+0} &       & 8     & \cellcolor[rgb]{ .749,  .749,  .749}\textbf{5.460E-2} & 5.470E-2 & 8.060E-2 & \diagbox{}{} & 7.710E-2 & 7.710E-2 \\
            & 10    & 1.801E+0 & 2.174E+0 & 1.783E+0 & \diagbox{}{} & \cellcolor[rgb]{ .749,  .749,  .749}\textbf{1.620E+0} & \cellcolor[rgb]{ .749,  .749,  .749}\textbf{1.620E+0} &       & 10    & \cellcolor[rgb]{ .749,  .749,  .749}\textbf{2.600E-2} & 3.230E-2 & 3.730E-2 & \diagbox{}{} & 3.460E-2 & 3.460E-2 \\
            \hline\hline
            \multirow{5}[1]{*}{\texttt{PMOP5}} & 2     & \cellcolor[rgb]{ .749,  .749,  .749}\textbf{0.000E+0} & 1.201E-1 & 1.920E-2 & 4.580E-1 & 1.170E-1 & 1.170E-1 & \multirow{5}[1]{*}{\texttt{PMOP11}} & 2     & \cellcolor[rgb]{ .749,  .749,  .749}\textbf{1.150E-2} & 1.887E-2 & 1.370E-2 & 1.980E-1 & 1.942E-1 & 1.942E-1 \\
            & 3     & \cellcolor[rgb]{ .749,  .749,  .749}\textbf{1.260E-2} & 1.008E+0 & 1.660E-2 & \diagbox{}{} & 7.320E-2 & 7.320E-2 &       & 3     & \cellcolor[rgb]{ .749,  .749,  .749}\textbf{1.070E-2} & 2.370E-1 & 1.774E-1 & \diagbox{}{} & 2.327E-1 & 2.327E-1 \\
            & 5     & \cellcolor[rgb]{ .749,  .749,  .749}\textbf{1.480E-2} & 3.220E-2 & 6.500E-2 & \diagbox{}{} & 4.940E-2 & 4.940E-2 &       & 5     & 7.960E-2 & \cellcolor[rgb]{ .749,  .749,  .749}\textbf{7.230E-2} & 1.488E-1 & \diagbox{}{} & 1.867E-1 & 1.867E-1 \\
            & 8     & \cellcolor[rgb]{ .749,  .749,  .749}\textbf{4.800E-3} & 2.150E-2 & 5.400E-3 & \diagbox{}{} & 9.000E-3 & 9.000E-3 &       & 8     & 8.670E-2 & 8.730E-2 & \cellcolor[rgb]{ .749,  .749,  .749}\textbf{8.630E-2} & \diagbox{}{} & 9.320E-2 & 9.320E-2 \\
            & 10    & \cellcolor[rgb]{ .749,  .749,  .749}\textbf{3.600E-3} & 3.700E-3 & 4.100E-3 & \diagbox{}{} & 3.700E-3 & 3.700E-3 &       & 10    & 1.500E-2 & 1.470E-2 & \cellcolor[rgb]{ .749,  .749,  .749}\textbf{1.450E-2} & \diagbox{}{} & \cellcolor[rgb]{ .749,  .749,  .749}\textbf{1.450E-2} & \cellcolor[rgb]{ .749,  .749,  .749}\textbf{1.450E-2} \\
            \hline\hline
            \multirow{5}[1]{*}{\texttt{PMOP6}} & 2     & 6.190E-2 & 3.471E-1 & \cellcolor[rgb]{ .749,  .749,  .749}\textbf{2.920E-2} & 7.952E-1 & 8.215E-1 & 8.215E-1 & \multirow{5}[1]{*}{\texttt{PMOP12}} & 2     & 7.780E-2 & 1.401E-1 & \cellcolor[rgb]{ .749,  .749,  .749}\textbf{2.420E-2} & 8.350E-2 & 7.780E-2 & 7.780E-2 \\
            & 3     & \cellcolor[rgb]{ .749,  .749,  .749}\textbf{5.230E-2} & 2.542E-1 & \cellcolor[rgb]{ .749,  .749,  .749}\textbf{5.230E-2} & \diagbox{}{} & 6.637E-1 & 6.637E-1 &       & 3     & \cellcolor[rgb]{ .749,  .749,  .749}\textbf{3.200E-2} & 3.590E-2 & 3.590E-2 & \diagbox{}{} & 9.060E-2 & 9.060E-2 \\
            & 5     & \cellcolor[rgb]{ .749,  .749,  .749}\textbf{2.076E-1} & 3.922E-1 & 3.266E-1 & \diagbox{}{} & 2.600E-1 & 2.600E-1 &       & 5     & \cellcolor[rgb]{ .749,  .749,  .749}\textbf{4.400E-3} & 6.100E-3 & 4.900E-3 & \diagbox{}{} & 2.160E-2 & 2.160E-2 \\
            & 8     & 1.578E+0 & 1.671E+0 & \cellcolor[rgb]{ .749,  .749,  .749}\textbf{1.408E+0} & \diagbox{}{} & 1.596E+0 & 1.596E+0 &       & 8     & \cellcolor[rgb]{ .749,  .749,  .749}\textbf{1.000E-3} & 7.500E-3 & 3.800E-3 & \diagbox{}{} & 1.000E-2 & 1.000E-2 \\
            & 10    & \cellcolor[rgb]{ .749,  .749,  .749}\textbf{1.194E+2} & 1.220E+2 & 1.292E+2 & \diagbox{}{} & 1.250E+2 & 1.250E+2 &       & 10    & \cellcolor[rgb]{ .749,  .749,  .749}\textbf{6.000E-4} & 7.000E-4 & 7.000E-4 & \diagbox{}{} & 4.200E-3 & 4.200E-3 \\
            \hline
        \end{tabular}}%
        \label{tab:local_knee_results}%
\end{table*}

All benchmark test problems introduced in~\pref{sec:problems}, except PMOP13 and PMOP14, are able to set to have $2^{m-1}$ local knee points. From the comparison results shown in~\pref{tab:local_knee_results} and the bar graphs shown in~\pref{fig:bar_graph_l}, we clearly observe the superiority of our proposed KPITU against the other five state-of-the-art KPI methods on problems with more than one knee point. In particular, KPITU shows the best performance on 48 out of 64 ($\approx$ 75\%) test problem instances, 8 of which were perfectly approximated. Due to the increase of the number of knee points, we can envisage the decreased chance for perfectly approximating all local knee points comparing to the case of having only one global knee point. As discussed in~\pref{sec:global_knee_results}, the performance of CHIM and MMD are still the same when having many local knee points. However, different from the observations in~\pref{sec:global_knee_results}, they become less competitive with the increase of the number of knee points. In particular, they are similar to that of CD and EMU$^r$ in this case. This might be because both of CHIM and MMD are designed to find the point that maximises the distance to a hyperplane formed by extreme points. By this means, they normally ignore the local information along the PF. RA is still the worst KPI method and it is not scalable to problems with more than two objectives. Comparing to the results on finding the global knee point, the performance of EMU$^r$ becomes relatively more competitive on problems having many local knee points. This might be the benefit of using weight vectors which always leads to the identification of some local knee points, especially when the underlying benchmark test problem has a regular PF similar to a simplex. From two selected examples shown in~\pref{fig:PMOP5_M2_A2} and~\pref{fig:PMOP2_M3_A4}, we can see that only our KPITU make a reasonable approximation to all local knee points whilst the \lq knee points\rq\ found by the other KPI methods are deviated from the ground truth. Furthermore, from the example shown in~\pref{fig:PMOP3_M10_A4}, we also appreciate that approximating multiple knee points in the high-dimensional space becomes more difficult due to the sparse distribution of knee points along with the huge objective space.

\begin{figure}[htbp]
    \centering
    \includegraphics[width=.4\linewidth]{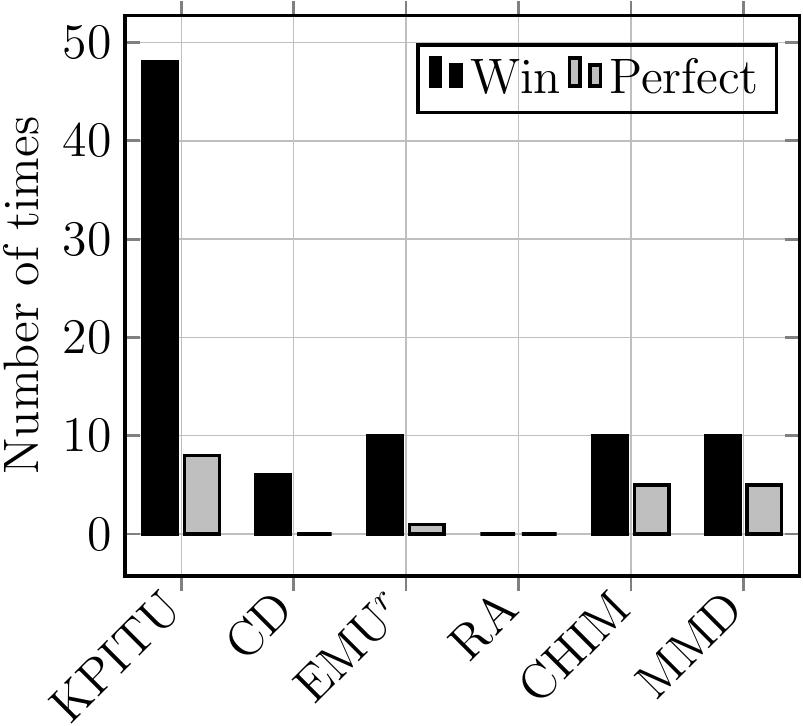}
    \caption{Bar graphs of the number of times that different algorithms obtain the best $\mathbb{I}(\mathrm{S})$ metric values and perfectly identify all local knee points.}
    \label{fig:bar_graph_l}
\end{figure}

\begin{figure*}[htbp]
    \includegraphics[width=1.0\linewidth]{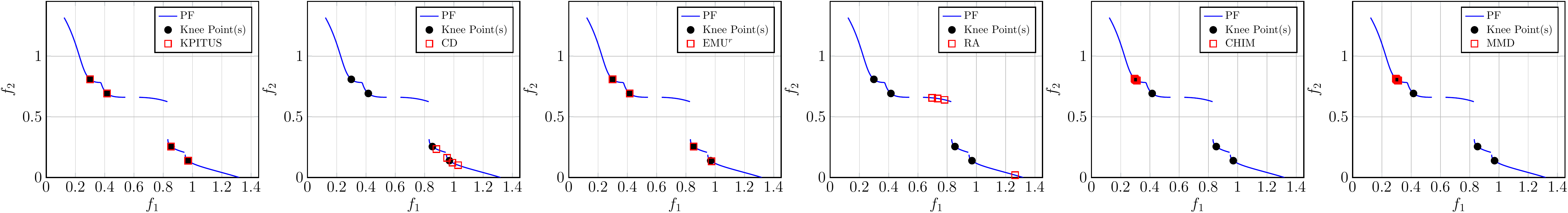}
    \caption{Knee points identified by different KPI methods on 2-objective PMOP5 with local knee points.}
    \label{fig:PMOP5_M2_A2}
\end{figure*}

\begin{figure*}[htbp]
    \includegraphics[width=1.0\linewidth]{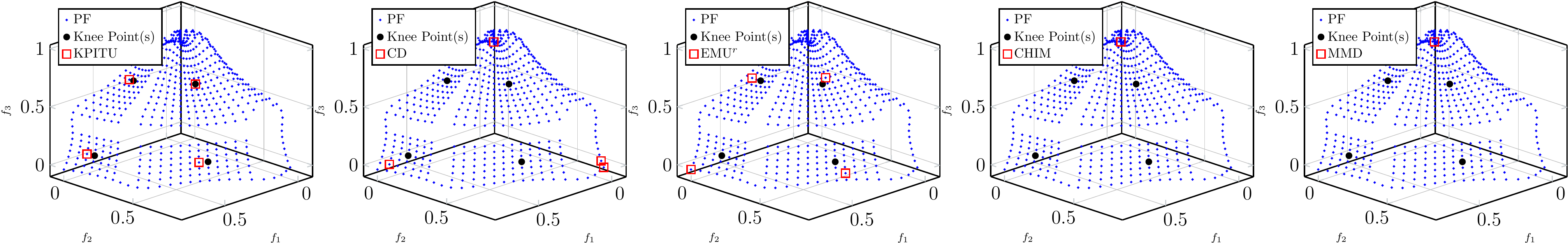}
    \caption{Knee points identified by different KPI methods on 2-objective PMOP2 with local knee points.}
    \label{fig:PMOP2_M3_A4}
\end{figure*}

\begin{figure*}[htbp]
    \includegraphics[width=1.0\linewidth]{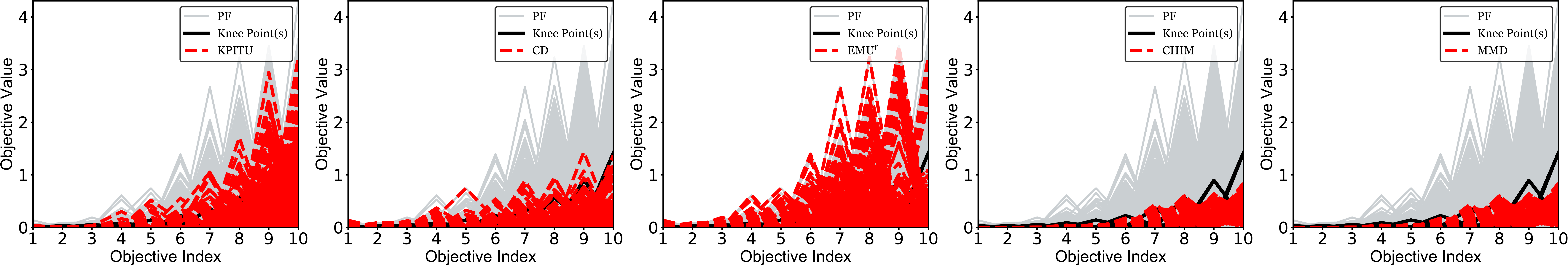}
    \caption{Knee points identified by different KPI methods on 10-objective PMOP3 with local knee points.}
    \label{fig:PMOP3_M10_A4}
\end{figure*}

\subsubsection{Results on Problems with Degenerate Knee Region(s)}
\label{sec:degenerate_knee_results}

In this subsection, we discuss the experimental results on PMOP13 and PMOP14 that have degenerate knee region(s). Different from the other benchmar test problems considered in the previous subsections, PMOP13 and PMOP14 are featured with an infinite number of knee points located at the corresponding knee region(s). As discussed in~\cite{YuJO19a}, the degenerate characteristics of PMOP13 and PMOP14 are caused by decoupling the trade-off relationship between some chosen objective functions with others.

From the comparison results given in~\pref{tab:PMOP1314} and the bar graphs shown in~\pref{fig:bar_graph_d}, it is clear to see that our proposed KPITU is the most competitive KPI method even when having infinitely many knee points. In particular, it obtains the best $\mathbb{I}(\mathrm{S})$ metric values on 12 out of 16 (75\%) test problem instances where 3 of them have been perfectly identified. In contrast, EMU$^r$ is the worst KPI method that mistakenly identifies many points outside the knee region. This might be because EMU$^r$ highly relies on the use of weight vectors evenly sampled from a canonical simplex whereas the shapes of the PF of neither PMOP13 nor PMOP14 follows a canonical simplex. Figs.~\ref{fig:PMOP13_M3_A1} to~\ref{fig:PMOP13_M5_A1} give some selected examples of knee points identified by different KPI methods. In particular, as for the 3-objective PMOP13 with two knee regions shown in~\pref{fig:PMOP13_M3_A3}, we can see that both CHIM and MMD can only identify knee points within one knee region. It is also interesting to see that all KPI methods, except EMU$^r$, perfectly identify the knee points on 5-objective PMOP13 with one knee region shown in~\pref{fig:PMOP13_M5_A1}.
\begin{figure}[htbp]
    \centering
    \includegraphics[width=.4\linewidth]{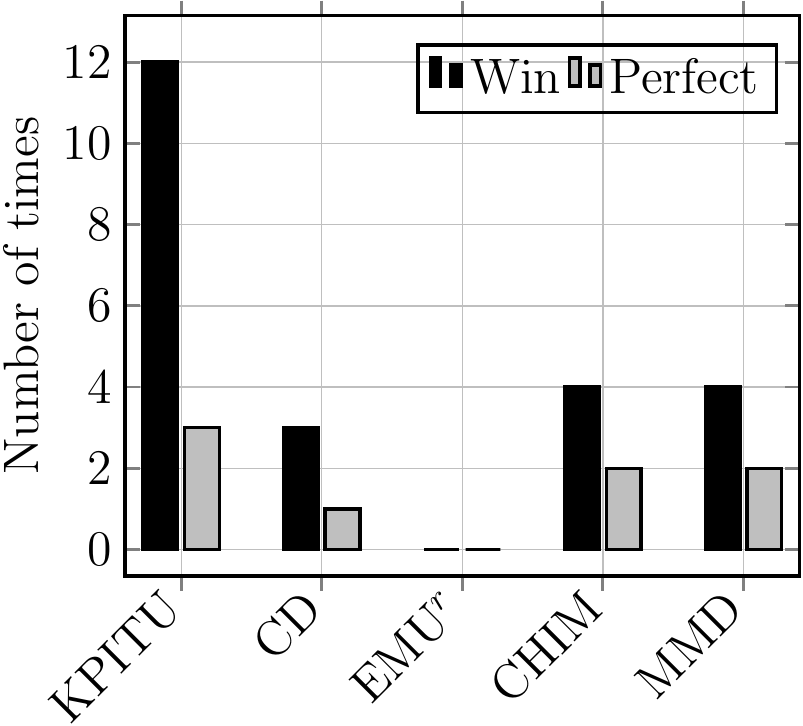}
    \caption{Bar graphs of the number of times that different algorithms obtain the best $\mathbb{I}(\mathrm{S})$ metric values and perfectly identify all knee points with in the underlying knee region(s).}
    \label{fig:bar_graph_d}
\end{figure}

\begin{table*}[htbp]
    \centering
    \caption{Comparison Results of KPITU with Other Four KPI Methods on Problems with Infinitely Many Knee Points}
    \resizebox{\textwidth}{!}{
    \begin{tabular}{c|c|ccccc||c|c|ccccc}
        \cline{2-7}\cline{9-14}
        & $m$     & \texttt{KPITU} & \texttt{CD}    & \texttt{EMU$^r$}  & \texttt{CHIM}  & \texttt{MMD}   &       & $m$     & \texttt{KPITU} & \texttt{CD}    & \texttt{EMU$^r$}  & \texttt{CHIM}  & \texttt{MMD} \\
        \hline
        \multirow{3}[4]{*}{\texttt{PMOP13$_1$}} & 3     & \cellcolor[rgb]{ .749,  .749,  .749}\textbf{0.000E+0} & 9.900E-2 & 1.420E-1 & 1.636E-3 & 1.636E-3 & \multirow{3}[4]{*}{\texttt{PMOP14$_1$}} & 3     & \cellcolor[rgb]{ .749,  .749,  .749}\textbf{0.000E+0} & 3.879E-1 & 5.476E-1 & 1.700E-3 & 1.700E-3 \\
        & 5     & \cellcolor[rgb]{ .749,  .749,  .749}\textbf{0.000E+0} & \cellcolor[rgb]{ .749,  .749,  .749}\textbf{0.000E+0} & 1.890E-1 & \cellcolor[rgb]{ .749,  .749,  .749}\textbf{0.000E+0} & \cellcolor[rgb]{ .749,  .749,  .749}\textbf{0.000E+0} &       & 5     & \cellcolor[rgb]{ .749,  .749,  .749}\textbf{3.380E-2} & 1.174E-1 & 2.778E-1 & 1.367E-1 & 1.367E-1 \\
        & 8     & \cellcolor[rgb]{ .749,  .749,  .749}\textbf{4.210E-1} & 4.561E-1 & 9.177E-1 & 4.842E-1 & 4.842E-1 &       & 8     & \cellcolor[rgb]{ .749,  .749,  .749}\textbf{7.990E-2} & 1.010E-1 & 2.052E-1 & 8.130E-2 & 8.130E-2 \\
        & 10    & \cellcolor[rgb]{ .749,  .749,  .749}\textbf{4.305E-1} & 4.541E-1 & 5.432E-1 & 4.564E-1 & 4.564E-1 &       & 10    & 8.687E-3 & \cellcolor[rgb]{ .749,  .749,  .749}\textbf{2.361E-3} & 9.214E-3 & \cellcolor[rgb]{ .749,  .749,  .749}\textbf{2.361E-3} & \cellcolor[rgb]{ .749,  .749,  .749}\textbf{2.361E-3} \\\hline\hline
        \multirow{3}[4]{*}{\texttt{PMOP13$_2$}} & 3     & \cellcolor[rgb]{ .749,  .749,  .749}\textbf{9.500E-3} & 1.090E-1 & 1.535E-1 & 7.995E-1 & 7.995E-1 & \multirow{3}[4]{*}{\texttt{PMOP14$_2$}} & 3     & 5.200E-3 & 3.360E-2 & 5.870E-2 & \cellcolor[rgb]{ .749,  .749,  .749}\textbf{0.000E+0} & \cellcolor[rgb]{ .749,  .749,  .749}\textbf{0.000E+0} \\
          & 5     & \cellcolor[rgb]{ .749,  .749,  .749}\textbf{7.700E-2} & 9.150E-2 & 4.875E-1 & 8.490E-1 & 8.490E-1 &       & 5     & \cellcolor[rgb]{ .749,  .749,  .749}\textbf{6.790E-2} & 7.910E-2 & 2.219E-1 & 2.232E-1 & 2.232E-1 \\
          & 8     & \cellcolor[rgb]{ .749,  .749,  .749}\textbf{2.711E-1} & 3.369E-1 & 8.639E-1 & 6.119E-1 & 6.119E-1 &       & 8     & \cellcolor[rgb]{ .749,  .749,  .749}\textbf{4.140E-2} & 2.001E-1 & 4.899E-1 & 1.855E-1 & 1.855E-1 \\
          & 10    & 9.156E-1 & \cellcolor[rgb]{ .749,  .749,  .749}\textbf{1.000E-2} & 1.217E+0 & 5.367E-1 & 5.367E-1 &       & 10    & 1.067E-1 & 4.182E-1 & 2.553E-1 & \cellcolor[rgb]{ .749,  .749,  .749}\textbf{3.960E-2} & \cellcolor[rgb]{ .749,  .749,  .749}\textbf{3.960E-2} \\
        \hline
    \end{tabular}
    }
\begin{tablenotes}
\item[1] \texttt{PMOP13}$_i$ and \texttt{PMOP14}$_i$ mean that the underlying test problem instance has $i\in\{1,2\}$ knee region(s).
\end{tablenotes} 
    \label{tab:PMOP1314}%
\end{table*}

\begin{figure*}[htbp]
    \includegraphics[width=1.0\linewidth]{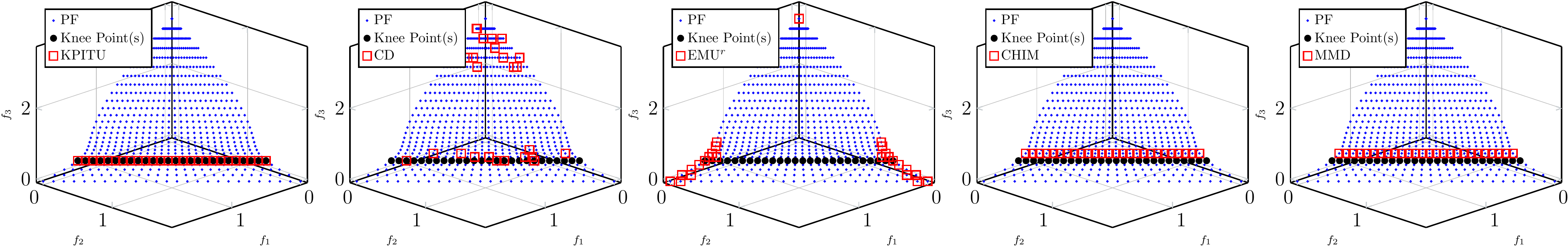}
    \caption{Knee points identified by different KPI methods on 3-objective PMOP13 with one knee region.}
    \label{fig:PMOP13_M3_A1}
\end{figure*}

\begin{figure*}[htbp]
    \includegraphics[width=1.0\linewidth]{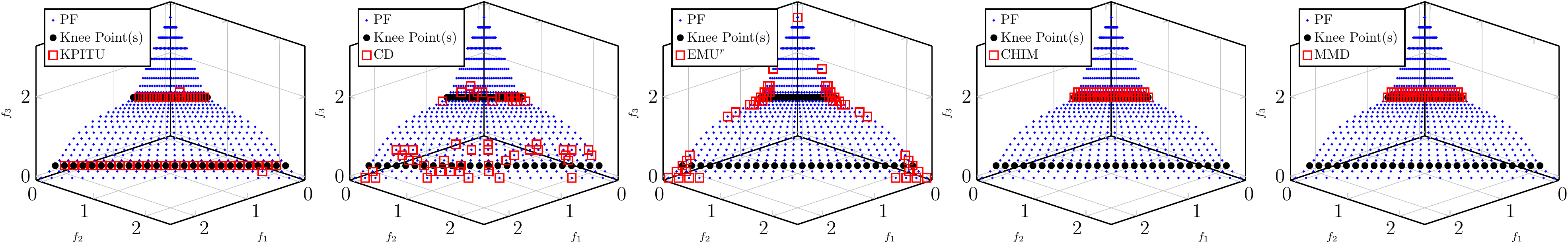}
    \caption{Knee points identified by different KPI methods on 3-objective PMOP13 with two knee regions.}
    \label{fig:PMOP13_M3_A3}
\end{figure*}

\begin{figure*}[htbp]
    \includegraphics[width=1.0\linewidth]{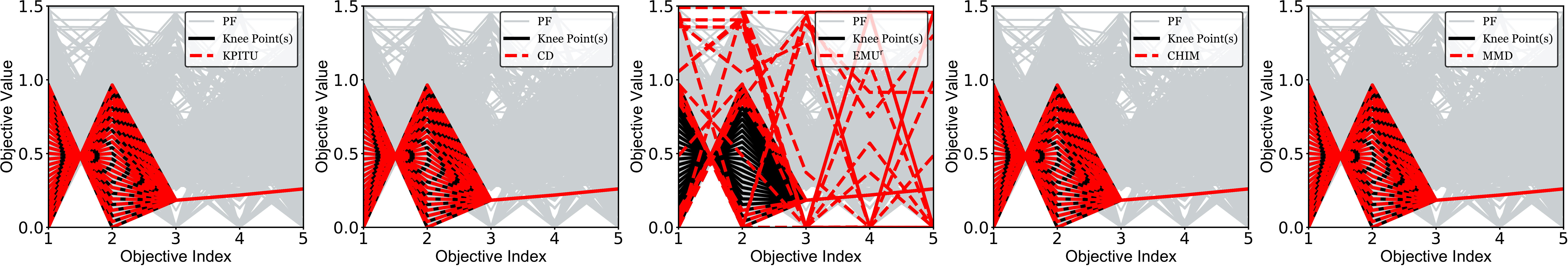}
    \caption{Knee points identified by different KPI methods on 5-objective PMOP13 with one knee region.}
    \label{fig:PMOP13_M5_A1}
\end{figure*}

\subsubsection{Incorporating KPITU in an EMO Algorithm}
\label{sec:nsga2_kpitu}

In the previous experiments, we have shown the effectiveness of KPITU as a standalone method for identifying knee point(s) from a set of trade-off alternatives. Given the simplicity of KPITU, we expect this method can be used to drive an EMO algorithm to search for knee point(s) directly. In this subsection, we use NSGA-II~\cite{DebAPM02}, one of the most popular EMO algorithms in the literature, as the baseline to validate this assertion. Our basic idea is to use KPITU to replace the crowding distance calculation in NSGA-II. \pref{alg:NSGA-II-KPITU} gives the pseudo-code of using KPITU to guide NSGA-II to search for knee point(s), dubbed $\mathtt{NSGA}$-$\mathtt{II}$-$\mathtt{KPITU}$. Specifically, let us consider the $t$-th generation of $\mathtt{NSGA}$-$\mathtt{II}$-$\mathtt{KPITU}$ where the parent population is denoted as $P_t$ and the offspring population is denoted as $Q_t$ (both of them have the same size $N$). The environmental selection first uses the non-dominated sorting to divide the hybrid population of the parents and offspring into several non-domination levels ($F_1$, $F_2$ and so on). Thereafter, solutions in the first several levels have a higher priority to be chosen to construct the next parent population until its size is equal to $N$ or for the first time exceeds $N$ (lines 5 to 7). Let us denote the last acceptable non-domination level as $F_l$. Instead of calculating the crowding distance as in the original NSGA-II, here we use $\mathtt{KPITU}(F_l)$ shown in~\pref{alg:KPITU} to identify the \textit{potential} knee point(s) (denoted as $\mathrm{K}$) from $F_l$. If the solutions in $\mathrm{K}$ have already filled the next parent population, the environmental selection at the $t$-th generation terminates (lines 14 and 15). On the other hand, if the size of $\mathrm{K}$ is smaller than the size of the remaining slot, we will use KPITU to identify the knee point(s) from $F_l\setminus\mathrm{K}$ (lines 19 to 28); otherwise, we use $\mathtt{Sort}(\mathrm{K})$ to sort the solutions in $\mathrm{K}$ according to their accumulative trade-off utility and the first several ones will be used to fill the remaining slot (lines 16 to 18).

In our experiments, we run $\mathtt{NSGA}$-$\mathtt{II}$-$\mathtt{KPITU}$ on all benchmark test problems considered in~\pref{sec:global_knee_results} and~\pref{sec:local_knee_results}, where $m\in\{2,3,5,8,10\}$. Three knee point driven algorithms, i.e., $\mathtt{NSGA}$-$\mathtt{II}$-$\mathtt{CD}$~\cite{Ramirez-Atencia17}, $\mathtt{NSGA}$-$\mathtt{II}$-$\mathtt{EMU}$~\cite{YuJO19a} and $\mathtt{NSGA}$-$\mathtt{II}$-$\mathtt{CHIM}$~\cite{YuJO19a}, are used as peer algorithms for comparison purpose. The population size is set to 100, 120, 160, 200 and 600 for different number of objectives, respectively. The settings of the number of generations can be found in Table 1 the supplementary document. As in~\pref{sec:global_knee_results} and~\pref{sec:local_knee_results}, the $\mathbb{I}(\mathrm{S})$ is used as the performance indicator. To have a statistical sound comparison, each algorithm is run for 31 times and the Wilcoxon's rank sum test with a 5\% significance level is applied to validate the statistical significance of the better/worse results achieved by $\mathtt{NSGA}$-$\mathtt{II}$-$\mathtt{KPITU}$. The median and interquartile range (IQR) of $\mathbb{I}(\mathrm{S})$ achieved by different algorithms are shown in~\pref{tab:nsga2_global} to~\pref{tab:nsga2_infinite}. From these results, we can see that our our proposed NSGA-II-KPITU is the most competitive one to search for knee point(s). In particular, it obtains the best $\mathbb{I}(\mathrm{S})$ median values in 95 out of 134 ($\approx$ 71\%) test problem instances whilst the second best algorithm is $\mathtt{NSGA}$-$\mathtt{II}$-$\mathtt{EMU}$ who obtains the best $\mathbb{I}(\mathrm{S})$ median values in 38 out of 134 ($\approx$ 28\%) test problem instances. \pref{fig:knea_PMOP7_M3_A1} to \pref{fig:knea_PMOP5_M10_A1} present the non-dominated solutions found by $\mathtt{NSGA}$-$\mathtt{II}$-$\mathtt{KPITU}$ and the other three peer algorithms whilst the complete results are put in the supplementary document. From these figures, we can see that although the $\mathbb{I}(\mathrm{S})$ values obtained by $\mathtt{NSGA}$-$\mathtt{II}$-$\mathtt{EMU}$ are competitive, its solutions tend to approximate the whole PF. In contrast, solutions obtained by the other three algorithms mainly focus on the knee point(s) whilst those obtained by $\mathtt{NSGA}$-$\mathtt{II}$-$\mathtt{CD}$ and $\mathtt{NSGA}$-$\mathtt{II}$-$\mathtt{CHIM}$ have some offsets with respect to the targeted knee point(s).

\begin{algorithm}[t]
    \caption{The $t$-th generation of $\mathtt{NSGA}$-$\mathtt{II}$-$\mathtt{KPITU}$}
    \label{alg:NSGA-II-KPITU}
    \KwIn{parent population $P_t$}
    \KwOut{$P_{t+1}$}
    $S_t\leftarrow\emptyset$,$i\leftarrow 1$;\\
    Use crossover and mutation to generate $Q_t$;\\
    $R_t\leftarrow P_t\cup Q_t$;\\
    Use non-dominated sorting to divide $R_t$ into several non-domination levels $F_1$, $F_2$, $\cdots$;\\
    \Repeat{$|S_t|\geq N$}{
        $S_t\leftarrow S_t\cup F_i$, $i\leftarrow i+1$;
    }
    $l\leftarrow i-1$;\\
    \uIf{$|S_t|==N$}{
        $P_{t+1}\leftarrow S_t$;
    }\Else{
        $P_{t+1}\leftarrow\cup_{i=1}^{l-1}F_i$;\\
        $\mathrm{K}\leftarrow\mathtt{KPITU}(F_l)$;\\
        \uIf{$|\mathrm{K}|==N-|P_{t+1}|$}{
            $P_{t+1}\leftarrow P_{t+1}\cup\mathrm{K}$;
        }\uElseIf{$|\mathrm{K}|>N-|P_{t+1}|$}{
            $F_l\leftarrow$ the first $N-|P_{t+1}|$ solutions in $\mathrm{K}$;\\
            $P_{t+1}\leftarrow P_{t+1}\cup F_l$;
        }\Else{
            \Repeat{$|\mathrm{K}|\geq N-|P_{t+1}|$}{
                $P_{t+1}\leftarrow P_{t+1}\cup\mathrm{K}$, $F_l\leftarrow F_l\setminus\mathrm{K}$;\\
                $\mathrm{K}\leftarrow\mathtt{KPITU}(F_l)$;
            }
            \uIf{$|\mathrm{K}|==N-|P_{t+1}|$}{
                $P_{t+1}\leftarrow P_{t+1}\cup\mathrm{K}$;
            }\Else{
                $F_l\leftarrow$ the first $N-|P_{t+1}|$ solutions in $\mathrm{K}$;\\
                $P_{t+1}\leftarrow P_{t+1}\cup F_l$;
            }
        }
    }
    \Return $P_{t+1}$
\end{algorithm}

\begin{table*}[htbp]
  \centering
  \caption{Comparison Results of $\mathbb{I}(\mathrm{S})$ Metric Values Obtained by $\mathtt{NSGA}$-$\mathtt{II}$-$\mathtt{KPITU}$ and Other Three Other Peer Algorithms on Problems with Only One Global Knee Point}
  \resizebox{\textwidth}{!}{
  \begin{threeparttable}
    \begin{tabular}{c|c|cccc||c|c|cccc}
	\cline{2-6}\cline{8-12}
          & $m$     & \texttt{KPITU} & \texttt{CD}    & \texttt{EMU}   & \texttt{CHIM}  &       & $m$     & \texttt{KPITU} & \texttt{CD}    & \texttt{EMU}   & \texttt{CHIM} \\
	\hline
    \texttt{DO2DK} & 2     & 1.711E-1(2.19E-4) & 5.126E-1(1.83E-3)$^{\dag}$ & \cellcolor[rgb]{ .753,  .753,  .753}\textbf{4.047E-4(6.42E-4)$^{\ddag}$} & 2.819E-1(2.34E+0)$^{\dag}$ &  \texttt{CKP}   & 2     & 4.995E-2(3.22E-4) & 6.691E-1(2.93E-3)$^{\dag}$ & \cellcolor[rgb]{ .753,  .753,  .753}\textbf{2.355E-2(3.03E-2)} & 2.692E-2(3.96E-2)$^{\dag}$ \\
    \hline
    \hline
     \texttt{DEB2DK} & 2     & \cellcolor[rgb]{ .753,  .753,  .753}\textbf{3.741E-5(1.13E-4)} & 3.780E-1(1.45E-3)$^{\dag}$ & 1.112E-3(1.51E-3)$^{\dag}$ & 3.522E-2(8.56E-2)$^{\dag}$ &  \texttt{DEB3DK} & 3     & \cellcolor[rgb]{ .753,  .753,  .753}\textbf{5.545E-2(6.14E-4)} & 1.605E-1(1.76E-2)$^{\dag}$ & 5.902E-2(2.44E-2) & 4.545E+0(7.71E-4)$^{\dag}$ \\
    \hline
    \hline
    \multirow{5}[2]{*}{\texttt{PMOP1}} & 2     & \cellcolor[rgb]{ .753,  .753,  .753}\textbf{1.052E-5(1.70E-5)} & 4.668E-2(3.26E-4)$^{\dag}$ & 8.274E-5(1.46E-4)$^{\dag}$ & 3.535E-3(6.06E-3)$^{\dag}$ & \multirow{5}[2]{*}{\texttt{PMOP7}} & 2     & \cellcolor[rgb]{ .753,  .753,  .753}\textbf{3.508E-5(7.10E-5)} & 2.707E+0(7.43E-3)$^{\dag}$ & 9.626E-4(9.40E-4)$^{\dag}$ & 3.588E-1(5.25E-1)$^{\dag}$ \\
          & 3     & \cellcolor[rgb]{ .753,  .753,  .753}\textbf{9.551E-3(2.76E-4)} & 8.357E-2(2.97E-3)$^{\dag}$ & 1.493E-2(1.09E-2) & 2.007E+0(1.93E-4)$^{\dag}$ &       & 3     & \cellcolor[rgb]{ .753,  .753,  .753}\textbf{8.067E-3(1.65E-4)} & 1.105E+0(2.11E-2)$^{\dag}$ & 3.717E-2(2.22E-2)$^{\dag}$ & 1.696E+0(6.17E-7)$^{\dag}$ \\
          & 5     & \cellcolor[rgb]{ .753,  .753,  .753}\textbf{3.254E-2(3.28E-3)} & 8.087E-2(9.18E-5) & 1.083E+0(6.23E-1)$^{\dag}$ & 2.726E+0(4.04E-4)$^{\dag}$ &       & 5     & \cellcolor[rgb]{ .753,  .753,  .753}\textbf{4.033E-2(2.81E-5)} & 5.320E-1(5.61E-2)$^{\dag}$ & 4.142E-1(2.82E-1)$^{\dag}$ & 1.195E+0(2.54E-1)$^{\dag}$ \\
          & 8     & \cellcolor[rgb]{ .753,  .753,  .753}\textbf{1.093E-1(7.70E-5)} & 2.587E-1(1.34E-1)$^{\dag}$ & 2.162E+0(8.94E-1)$^{\dag}$ & 3.507E+0(2.20E+0)$^{\dag}$ &       & 8     & \cellcolor[rgb]{ .753,  .753,  .753}\textbf{3.721E-2(3.57E-5)} & 5.513E-1(1.21E-1)$^{\dag}$ & 1.887E-1(6.74E-2)$^{\dag}$ & 1.147E+0(5.57E-8)$^{\dag}$ \\
          & 10    & \cellcolor[rgb]{ .753,  .753,  .753}\textbf{1.463E-1(3.52E-4)} & 6.014E-1(3.41E-1)$^{\dag}$ & 2.510E+0(4.38E-1)$^{\dag}$ & 2.951E+0(2.75E+0)$^{\dag}$ &       & 10    & \cellcolor[rgb]{ .753,  .753,  .753}\textbf{3.921E-2(6.41E-5)} & 3.000E-1(2.87E-1)$^{\dag}$ & 2.343E-1(5.27E-2)$^{\dag}$ & 1.224E+0(1.14E-7)$^{\dag}$ \\
    \hline
    \hline
    \multirow{5}[2]{*}{\texttt{PMOP2}} & 2     & 4.280E-1(1.71E-10) & 4.280E-1(1.38E-10) & \cellcolor[rgb]{ .753,  .753,  .753}\textbf{4.299E-3(7.39E-3)$^{\ddag}$} & 3.016E-1(4.17E-2)$^{\dag}$ & \multirow{5}[2]{*}{\texttt{PMOP8}} & 2     & \cellcolor[rgb]{ .753,  .753,  .753}\textbf{4.532E-6(1.12E-5)} & 1.322E-1(3.60E-4)$^{\dag}$ & 2.773E-4(4.84E-4)$^{\dag}$ & 5.589E-3(2.66E-2)$^{\dag}$ \\
          & 3     & 4.444E-1(2.19E-9) & 4.444E-1(3.71E-10)$^{\dag}$ & \cellcolor[rgb]{ .753,  .753,  .753}\textbf{4.554E-2(6.61E-2)$^{\ddag}$} & 4.443E-1(8.04E-2)$^{\dag}$ &       & 3     & \cellcolor[rgb]{ .753,  .753,  .753}\textbf{7.872E-3(4.98E-3)} & 1.618E-1(3.95E-3)$^{\dag}$ & 9.839E-2(3.06E-4)$^{\dag}$ & 1.869E-1(4.53E+0)$^{\dag}$ \\
          & 5     & 6.946E-1(1.05E-9) & 5.193E-1(6.70E-8)$^{\dag}$ & \cellcolor[rgb]{ .753,  .753,  .753}\textbf{1.540E-1(7.34E-2)$^{\ddag}$} & 7.224E-1(3.05E-2)$^{\dag}$ &       & 5     & \cellcolor[rgb]{ .753,  .753,  .753}\textbf{1.526E-1(1.99E-1)} & 2.599E-1(1.44E-2)$^{\dag}$ & 2.247E-1(6.14E-5) & 7.042E+0(4.33E-1)$^{\dag}$ \\
          & 8     & 5.992E-1(1.34E-9) & 6.747E-1(1.12E-5)$^{\dag}$ & \cellcolor[rgb]{ .753,  .753,  .753}\textbf{1.871E-1(4.02E-2)$^{\ddag}$} & 6.234E-1(1.52E-4)$^{\dag}$ &       & 8     & \cellcolor[rgb]{ .753,  .753,  .753}\textbf{3.892E-1(2.02E-4)} & 7.848E-1(5.12E-2)$^{\dag}$ & 4.431E-1(2.05E-2)$^{\dag}$ & 2.101E+1(2.89E-6)$^{\dag}$ \\
          & 10    & 5.771E-1(3.59E-9) & 6.497E-1(2.84E-2)$^{\dag}$ & \cellcolor[rgb]{ .753,  .753,  .753}\textbf{1.963E-1(4.81E-2)$^{\ddag}$} & 6.012E-1(9.06E-3)$^{\dag}$ &       & 10    & \cellcolor[rgb]{ .753,  .753,  .753}\textbf{1.820E+0(1.01E-1)} & 1.895E+0(2.01E-1) & 5.006E+0(3.90E+0)$^{\dag}$ & 4.919E+0(1.22E-1)$^{\dag}$ \\
    \hline
    \hline
    \multirow{5}[2]{*}{\texttt{PMOP3}} & 2     & \cellcolor[rgb]{ .753,  .753,  .753}\textbf{2.358E-5(1.51E-4)} & 1.221E+0(9.62E-4)$^{\dag}$ & 1.301E-3(2.20E-1)$^{\dag}$ & 9.719E-1(8.65E-1)$^{\dag}$ & \multirow{5}[2]{*}{\texttt{PMOP9}} & 2     & 8.774E-1(2.49E-7) & 8.774E-1(4.91E-5)$^{\dag}$ & \cellcolor[rgb]{ .753,  .753,  .753}\textbf{7.763E-1(6.46E-2)$^{\ddag}$} & 8.658E-1(2.08E-2)$^{\dag}$ \\
          & 3     & 5.079E-1(3.38E-4) & 4.792E-1(4.51E-1) & \cellcolor[rgb]{ .753,  .753,  .753}\textbf{1.179E-1(7.00E-2)$^{\ddag}$} & 2.460E+0(1.00E+0)$^{\dag}$ &       & 3     & \cellcolor[rgb]{ .753,  .753,  .753}\textbf{1.699E-6(1.22E-5)} & 6.799E-3(5.00E-5)$^{\dag}$ & 7.577E-5(1.14E-4)$^{\dag}$ & 4.983E-4(9.04E-4)$^{\dag}$ \\
          & 5     & \cellcolor[rgb]{ .753,  .753,  .753}\textbf{4.189E-1(4.62E-5)} & 4.907E-1(1.79E-2)$^{\dag}$ & 4.523E-1(3.07E-1)$^{\dag}$ & 2.113E+0(4.12E-1)$^{\dag}$ &       & 5     & \cellcolor[rgb]{ .753,  .753,  .753}\textbf{9.793E-3(7.33E-4)} & 1.083E-2(1.09E-4)$^{\dag}$ & 1.081E-2(7.89E-6)$^{\dag}$ & 1.074E-2(1.83E-1)$^{\dag}$ \\
          & 8     & 6.482E-1(2.34E-5) & 6.022E-1(2.08E-2)$^{\dag}$ & \cellcolor[rgb]{ .753,  .753,  .753}\textbf{2.842E-1(2.35E-1)$^{\ddag}$} & 1.873E+0(4.21E-1)$^{\dag}$ &       & 8     & \cellcolor[rgb]{ .753,  .753,  .753}\textbf{7.419E-3(1.76E-6)} & 5.860E-2(5.22E-2)$^{\dag}$ & 8.757E-2(3.59E-2)$^{\dag}$ & 1.471E-1(2.28E-1)$^{\dag}$ \\
          & 10    & 6.506E-1(1.32E-5) & 6.190E-1(1.29E-2)$^{\dag}$ & \cellcolor[rgb]{ .753,  .753,  .753}\textbf{3.874E-1(1.91E-1)$^{\ddag}$} & 1.820E+0(4.82E-1)$^{\dag}$ &       & 10    & \cellcolor[rgb]{ .753,  .753,  .753}\textbf{4.298E-1(3.32E-7)} & 5.297E-1(8.17E-5)$^{\dag}$ & 4.305E-1(4.30E-2)$^{\dag}$ & 5.212E-1(3.28E-3)$^{\dag}$ \\
    \hline
    \hline
    \multirow{5}[2]{*}{\texttt{PMOP4}} & 2     & \cellcolor[rgb]{ .753,  .753,  .753}\textbf{7.776E-3(1.19E-2)} & 1.253E-2(3.67E-2) & 2.723E-2(3.85E-2)$^{\dag}$ & 1.036E-2(2.90E-2) & \multirow{5}[2]{*}{\texttt{PMOP11}} & 2     & 2.397E-2(2.58E-5) & \cellcolor[rgb]{ .753,  .753,  .753}\textbf{3.215E-5(5.64E-5)$^{\ddag}$} & 2.129E-3(3.08E-3)$^{\dag}$ & 2.759E-2(8.23E-3) \\
          & 3     & \cellcolor[rgb]{ .753,  .753,  .753}\textbf{1.433E-2(1.42E-1)} & 2.802E+0(1.38E+0)$^{\dag}$ & 8.058E-1(5.27E-1)$^{\dag}$ & 9.575E-1(1.66E+0)$^{\dag}$ &       & 3     & 2.679E-2(7.69E-5) & 4.651E-2(7.41E-3)$^{\dag}$ & \cellcolor[rgb]{ .753,  .753,  .753}\textbf{4.077E-3(3.75E-3)$^{\ddag}$} & 1.976E-1(6.15E-5)$^{\dag}$ \\
          & 5     & \cellcolor[rgb]{ .753,  .753,  .753}\textbf{6.565E-1(1.58E-1)} & 7.145E-1(2.43E-1)$^{\dag}$ & 5.308E+0(2.08E+0)$^{\dag}$ & 2.389E+0(2.72E+0)$^{\dag}$ &       & 5     & 1.111E-2(1.09E-5) & 1.637E-2(1.72E-1)$^{\dag}$ & \cellcolor[rgb]{ .753,  .753,  .753}\textbf{8.390E-3(8.00E-3)} & 8.442E-2(4.71E-7)$^{\dag}$ \\
          & 8     & \cellcolor[rgb]{ .753,  .753,  .753}\textbf{3.656E+0(4.01E-1)} & 5.458E+0(4.60E+0)$^{\dag}$ & 1.179E+1(4.08E+0)$^{\dag}$ & 8.519E+0(5.89E+0)$^{\dag}$ &       & 8     & \cellcolor[rgb]{ .753,  .753,  .753}\textbf{1.442E-3(1.04E-6)} & 1.683E-3(7.70E-4)$^{\dag}$ & 6.197E-3(1.46E-2)$^{\dag}$ & 1.245E-2(1.43E-9)$^{\dag}$ \\
          & 10    & \cellcolor[rgb]{ .753,  .753,  .753}\textbf{3.020E+0(9.28E-1)} & 6.417E+0(1.41E-1)$^{\dag}$ & 2.188E+1(1.89E+1)$^{\dag}$ & 1.522E+1(8.29E+0)$^{\dag}$ &       & 10    & \cellcolor[rgb]{ .753,  .753,  .753}\textbf{3.288E-4(4.01E-7)} & 3.786E-4(1.08E-4)$^{\dag}$ & 5.970E-3(5.23E-3)$^{\dag}$ & 2.909E-3(1.63E-9)$^{\dag}$ \\
    \hline
    \hline
    \multirow{5}[2]{*}{\texttt{PMOP6}} & 2     & \cellcolor[rgb]{ .753,  .753,  .753}\textbf{3.262E-6(1.97E-5)} & 6.100E-2(5.09E-2)$^{\dag}$ & 1.270E-4(2.66E-4)$^{\dag}$ & 3.417E-3(5.11E-3)$^{\dag}$ & \multirow{5}[2]{*}{\texttt{PMOP12}} & 2     & \cellcolor[rgb]{ .753,  .753,  .753}\textbf{1.360E-5(3.20E-5)} & 2.884E-1(6.67E-4)$^{\dag}$ & 4.549E-4(6.09E-4)$^{\dag}$ & 3.095E-1(3.04E-1)$^{\dag}$ \\
          & 3     & \cellcolor[rgb]{ .753,  .753,  .753}\textbf{8.298E-3(6.04E-5)} & 4.629E-2(2.48E-1)$^{\dag}$ & 1.205E-2(1.35E-3)$^{\dag}$ & 7.946E-1(3.46E-1)$^{\dag}$ &       & 3     & 4.141E-2(5.22E-5) & 7.941E-3(1.57E-3)$^{\dag}$ & \cellcolor[rgb]{ .753,  .753,  .753}\textbf{6.692E-3(4.21E-3)$^{\ddag}$} & 3.623E-1(4.44E-2)$^{\dag}$ \\
          & 5     & \cellcolor[rgb]{ .753,  .753,  .753}\textbf{2.055E-2(2.38E-5)} & 2.102E-2(6.21E-1)$^{\dag}$ & 4.840E-1(2.82E-1)$^{\dag}$ & 1.604E+0(1.23E+0)$^{\dag}$ &       & 5     & \cellcolor[rgb]{ .753,  .753,  .753}\textbf{1.137E-2(2.98E-5)} & 1.249E-2(2.65E-3) & 1.551E-2(7.21E-3)$^{\dag}$ & 1.133E-1(5.34E-2)$^{\dag}$ \\
          & 8     & \cellcolor[rgb]{ .753,  .753,  .753}\textbf{1.195E-1(1.88E-5)} & 7.942E-1(1.08E+0) & 1.681E+0(7.77E-1)$^{\dag}$ & 2.375E+0(3.51E+0)$^{\dag}$ &       & 8     & \cellcolor[rgb]{ .753,  .753,  .753}\textbf{5.414E-3(4.66E-5)} & 6.920E-3(6.23E-4)$^{\dag}$ & 1.065E-2(3.87E-3)$^{\dag}$ & 5.026E-2(3.86E-5)$^{\dag}$ \\
          & 10    & \cellcolor[rgb]{ .753,  .753,  .753}\textbf{5.754E-1(5.46E-5)} & 2.780E+0(1.67E+0) & 3.873E+0(4.08E+0)$^{\dag}$ & 1.072E+1(8.49E+0)$^{\dag}$ &       & 10    & \cellcolor[rgb]{ .753,  .753,  .753}\textbf{5.193E-3(2.70E-4)} & 5.213E-3(7.48E-4) & 8.158E-3(4.72E-3)$^{\dag}$ & 3.648E-2(2.72E-5)$^{\dag}$ \\
    \hline
    \end{tabular}%
    
    \begin{tablenotes}
        \item[1] $^{\dag}$ denotes the performance of $\mathtt{NSGA}$-$\mathtt{II}$-$\mathtt{KPITU}$ is significantly better than the other peers according to the Wilcoxon's rank sum test at a 0.05 significance level whilst $^{\ddag}$ denotes the opposite case.
      \end{tablenotes}
    \end{threeparttable}}
   
  \label{tab:nsga2_global}
\end{table*}

\begin{table*}[htbp]
  \centering
  \caption{Comparison Results of $\mathbb{I}(\mathrm{S})$ Metric Values Obtained by $\mathtt{NSGA}$-$\mathtt{II}$-$\mathtt{KPITU}$ and Other Three Other Peer Algorithms on Problems with Local Knee Points}
  \resizebox{\textwidth}{!}{
  \begin{threeparttable}
    \begin{tabular}{c|c|cccc||c|c|cccc}
    \cline{2-6}\cline{8-12}
          & $m$     & \texttt{KPITU} & \texttt{CD}    & \texttt{EMU}   & \texttt{CHIM}  &       & $m$     & \texttt{KPITU} & \texttt{CD}    & \texttt{EMU}   & \texttt{CHIM} \\
    \hline
    \texttt{DO2DK} & 2     & 1.256E-2(1.26E-4) & 1.473E+0(1.03E-3)$^{\dag}$ & \cellcolor[rgb]{ .753,  .753,  .753}\textbf{6.668E-3(5.11E-3)$^{\ddag}$} & 1.533E+0(2.05E-1)$^{\dag}$ & \texttt{CKP}   & 2     & 1.533E+0(1.39E-3) & 1.802E+0(3.37E-4)$^{\dag}$ & \cellcolor[rgb]{ .753,  .753,  .753}\textbf{5.803E-2(9.54E-2)$^{\ddag}$} & 2.043E+0(1.63E-1)$^{\dag}$ \\
    \hline
    \hline
    \texttt{DEB2DK} & 2     & \cellcolor[rgb]{ .753,  .753,  .753}\textbf{1.783E-2(4.23E-4)} & 1.787E+0(3.07E-4)$^{\dag}$ & 1.928E-2(3.01E-2)$^{\dag}$ & 1.840E+0(1.65E-1)$^{\dag}$ & \texttt{DEB3DK} & 3     & \cellcolor[rgb]{ .753,  .753,  .753}\textbf{4.102E-1(2.33E-2)} & 2.449E+0(1.07E-1)$^{\dag}$ & 7.976E-1(6.46E-1) & 4.628E+0(4.81E-4)$^{\dag}$ \\
    \hline
    \hline
    \multirow{5}[2]{*}{\texttt{PMOP1}} & 2     & \cellcolor[rgb]{ .753,  .753,  .753}\textbf{4.744E-4(1.18E-4)} & 5.850E-1(6.05E-4)$^{\dag}$ & 4.014E-3(3.37E-3)$^{\dag}$ & 4.996E-1(2.92E-2)$^{\dag}$ & \multirow{5}[2]{*}{\texttt{PMOP7}} & 2     & \cellcolor[rgb]{ .753,  .753,  .753}\textbf{4.168E-3(1.59E-6)} & 2.365E+0(8.37E-4)$^{\dag}$ & 5.648E-3(1.16E-2)$^{\dag}$ & 1.481E+0(6.29E-2)$^{\dag}$ \\
          & 3     & \cellcolor[rgb]{ .753,  .753,  .753}\textbf{5.546E-3(1.62E-4)} & 8.955E-1(2.01E-2)$^{\dag}$ & 1.357E-1(2.18E-1)$^{\dag}$ & 2.012E+0(2.71E-4)$^{\dag}$ &       & 3     & \cellcolor[rgb]{ .753,  .753,  .753}\textbf{6.227E-2(2.96E-4)} & 1.078E+0(1.06E-1)$^{\dag}$ & 8.634E-2(4.80E-2) & 1.314E+0(2.15E-6)$^{\dag}$ \\
          & 5     & \cellcolor[rgb]{ .753,  .753,  .753}\textbf{9.937E-1(3.58E-1)} & 1.642E+0(1.22E-1)$^{\dag}$ & 1.182E+0(4.02E-1) & 3.083E+0(1.35E+0)$^{\dag}$ &       & 5     & \cellcolor[rgb]{ .753,  .753,  .753}\textbf{1.847E-1(7.19E-3)} & 8.088E-1(1.38E-1)$^{\dag}$ & 2.047E-1(6.37E-2) & 9.349E-1(1.60E-1)$^{\dag}$ \\
          & 8     & \cellcolor[rgb]{ .753,  .753,  .753}\textbf{6.162E-1(2.12E-1)} & 3.178E+0(4.93E-1)$^{\dag}$ & 2.393E+0(8.20E-1)$^{\dag}$ & 5.704E+0(2.25E+0)$^{\dag}$ &       & 8     & \cellcolor[rgb]{ .753,  .753,  .753}\textbf{1.514E-1(1.21E-2)} & 7.629E-1(9.13E-2)$^{\dag}$ & 1.830E-1(5.50E-2) & 8.202E-1(1.02E-1)$^{\dag}$ \\
          & 10    & \cellcolor[rgb]{ .753,  .753,  .753}\textbf{1.417E+0(3.19E-1)} & 3.931E+0(7.98E-1)$^{\dag}$ & 1.854E+0(3.73E-1) & 4.407E+0(2.33E+0)$^{\dag}$ &       & 10    & \cellcolor[rgb]{ .753,  .753,  .753}\textbf{1.743E-1(3.92E-2)} & 6.880E-1(7.16E-2)$^{\dag}$ & 2.146E-1(9.06E-2) & 7.930E-1(1.37E-1)$^{\dag}$ \\
    \hline
    \hline
    \multirow{5}[2]{*}{\texttt{PMOP2}} & 2     & 2.952E-1(1.74E-7) & 5.783E-1(1.56E-4)$^{\dag}$ & \cellcolor[rgb]{ .753,  .753,  .753}\textbf{8.646E-3(8.83E-3)$^{\ddag}$} & 3.685E-1(5.80E-2)$^{\dag}$ & \multirow{5}[2]{*}{\texttt{PMOP8}} & 2     & 1.303E-2(5.23E-5) & 6.232E-1(2.26E-4)$^{\dag}$ & \cellcolor[rgb]{ .753,  .753,  .753}\textbf{5.996E-3(3.70E-3)$^{\ddag}$} & 7.343E-1(3.11E-2)$^{\dag}$ \\
          & 3     & 4.505E-1(1.15E-7) & 6.521E-1(7.03E-3)$^{\dag}$ & \cellcolor[rgb]{ .753,  .753,  .753}\textbf{5.136E-2(2.83E-2)$^{\ddag}$} & 5.695E-1(4.00E-2)$^{\dag}$ &       & 3     & 5.833E-1(7.48E-3) & 5.988E-1(3.35E-3)$^{\dag}$ & \cellcolor[rgb]{ .753,  .753,  .753}\textbf{3.982E-1(5.44E-3)$^{\ddag}$} & 2.217E+0(2.50E-3)$^{\dag}$ \\
          & 5     & 2.112E-1(2.72E-9) & 6.726E-1(1.61E-3)$^{\dag}$ & \cellcolor[rgb]{ .753,  .753,  .753}\textbf{1.916E-1(4.33E-2)} & 4.731E-1(1.43E-1) &       & 5     & \cellcolor[rgb]{ .753,  .753,  .753}\textbf{6.746E-1(2.32E-5)} & 7.826E-1(3.68E-3)$^{\dag}$ & 6.970E-1(4.27E-1)$^{\dag}$ & 7.631E-1(1.46E-4)$^{\dag}$ \\
          & 8     & \cellcolor[rgb]{ .753,  .753,  .753}\textbf{1.135E-1(2.36E-9)} & 5.453E-1(1.81E-2)$^{\dag}$ & 1.955E-1(4.03E-2)$^{\dag}$ & 4.244E-1(6.03E-2)$^{\dag}$ &       & 8     & \cellcolor[rgb]{ .753,  .753,  .753}\textbf{4.067E-1(3.13E-5)} & 7.138E-1(7.57E-3)$^{\dag}$ & 5.834E-1(2.85E-1)$^{\dag}$ & 8.778E-1(5.09E-7)$^{\dag}$ \\
          & 10    & \cellcolor[rgb]{ .753,  .753,  .753}\textbf{1.315E-1(2.33E-4)} & 5.026E-1(4.15E-2)$^{\dag}$ & 1.332E-1(3.88E-2) & 4.697E-1(6.27E-2)$^{\dag}$ &       & 10    & \cellcolor[rgb]{ .753,  .753,  .753}\textbf{4.826E-1(1.12E-4)} & 1.025E+0(5.56E-3)$^{\dag}$ & 6.040E-1(1.50E-1) & 4.040E+0(3.06E-4)$^{\dag}$ \\
    \hline
    \hline
    \multirow{5}[2]{*}{\texttt{PMOP3}} & 2     & 5.103E-1(5.78E-5) & 1.580E+0(5.43E-5)$^{\dag}$ & \cellcolor[rgb]{ .753,  .753,  .753}\textbf{5.739E-3(1.69E-1)$^{\ddag}$} & 1.539E+0(3.56E-1)$^{\dag}$ & \multirow{5}[2]{*}{\texttt{PMOP9}} & 2     & 5.443E-1(9.11E-5) & 5.441E-1(4.64E-5)$^{\dag}$ & \cellcolor[rgb]{ .753,  .753,  .753}\textbf{2.236E-4(2.76E-4)$^{\ddag}$} & 5.441E-1(4.02E-3) \\
          & 3     & 8.211E-1(2.29E-4) & 1.220E+0(8.00E-2)$^{\dag}$ & \cellcolor[rgb]{ .753,  .753,  .753}\textbf{9.003E-2(7.60E-2)$^{\ddag}$} & 2.108E+0(4.17E-1)$^{\dag}$ &       & 3     & 2.760E-1(1.21E-3) & 4.450E-1(7.55E-2)$^{\dag}$ & \cellcolor[rgb]{ .753,  .753,  .753}\textbf{1.748E-2(6.94E-3)$^{\ddag}$} & 7.342E-1(3.74E-1)$^{\dag}$ \\
          & 5     & \cellcolor[rgb]{ .753,  .753,  .753}\textbf{8.040E-2(7.66E-5)} & 9.756E-1(9.24E-3)$^{\dag}$ & 3.430E-1(1.97E-1)$^{\dag}$ & 1.809E+0(9.14E-1)$^{\dag}$ &       & 5     & 2.203E-1(2.89E-5) & 4.698E-1(4.41E-3)$^{\dag}$ & \cellcolor[rgb]{ .753,  .753,  .753}\textbf{1.710E-1(5.23E-2)$^{\ddag}$} & 4.033E-1(9.92E-3)$^{\dag}$ \\
          & 8     & \cellcolor[rgb]{ .753,  .753,  .753}\textbf{1.013E-1(2.12E-5)} & 9.955E-1(3.96E-3)$^{\dag}$ & 2.575E-1(1.89E-1)$^{\dag}$ & 1.260E+0(7.46E-1)$^{\dag}$ &       & 8     & 7.014E-1(6.48E-6) & 7.910E-1(9.00E-4)$^{\dag}$ & \cellcolor[rgb]{ .753,  .753,  .753}\textbf{6.378E-1(1.27E-1)} & 7.631E-1(6.56E-1) \\
          & 10    & \cellcolor[rgb]{ .753,  .753,  .753}\textbf{1.209E-1(5.84E-6)} & 9.960E-1(1.37E-3)$^{\dag}$ & 2.079E-1(4.34E-2)$^{\dag}$ & 1.682E+0(7.98E-1)$^{\dag}$ &       & 10    & 8.901E-1(1.77E-5) & 1.287E+0(2.28E-4)$^{\dag}$ & \cellcolor[rgb]{ .753,  .753,  .753}\textbf{8.345E-1(1.31E-1)$^{\ddag}$} & 1.022E+0(1.08E-2)$^{\dag}$ \\
    \hline
    \hline
    \multirow{5}[2]{*}{\texttt{PMOP4}} & 2     & \cellcolor[rgb]{ .753,  .753,  .753}\textbf{1.640E-2(5.19E-3)} & 1.014E-1(1.19E-1)$^{\dag}$ & 1.825E-2(2.02E-2) & 3.357E-1(1.05E-2)$^{\dag}$ & \multirow{5}[2]{*}{\texttt{PMOP10}} & 2     & \cellcolor[rgb]{ .753,  .753,  .753}\textbf{5.310E-3(2.89E-6)} & 1.470E+0(4.83E-1)$^{\dag}$ & 7.706E-3(5.18E-2) & 2.292E-1(5.99E-2)$^{\dag}$ \\
          & 3     & \cellcolor[rgb]{ .753,  .753,  .753}\textbf{1.390E-1(8.96E-2)} & 2.013E+0(1.56E+0)$^{\dag}$ & 2.867E-1(7.40E-2) & 9.302E-1(7.46E-1)$^{\dag}$ &       & 3     & \cellcolor[rgb]{ .753,  .753,  .753}\textbf{4.878E-2(7.27E-5)} & 1.277E+0(4.19E-1)$^{\dag}$ & 7.711E-2(4.09E-2) & 8.525E-1(3.62E-1) \\
          & 5     & \cellcolor[rgb]{ .753,  .753,  .753}\textbf{9.060E-1(4.53E-2)} & 3.311E+0(1.17E+0)$^{\dag}$ & 1.034E+0(1.57E-2) & 2.635E+0(1.39E+0)$^{\dag}$ &       & 5     & \cellcolor[rgb]{ .753,  .753,  .753}\textbf{1.234E-1(1.41E-5)} & 1.461E-1(3.08E-3)$^{\dag}$ & 1.497E-1(1.46E-2)$^{\dag}$ & 5.695E-1(5.44E-1)$^{\dag}$ \\
          & 8     & \cellcolor[rgb]{ .753,  .753,  .753}\textbf{2.124E+0(5.50E-2)} & 7.703E+0(2.28E+0)$^{\dag}$ & 6.240E+0(5.65E-2)$^{\dag}$ & 6.007E+0(2.19E+0)$^{\dag}$ &       & 8     & \cellcolor[rgb]{ .753,  .753,  .753}\textbf{6.835E-2(8.18E-6)} & 8.871E-2(1.75E-3)$^{\dag}$ & 9.911E-2(2.36E-2)$^{\dag}$ & 9.185E-1(2.34E-7)$^{\dag}$ \\
          & 10    & \cellcolor[rgb]{ .753,  .753,  .753}\textbf{2.820E+0(7.30E-1)} & 1.335E+1(4.32E+0)$^{\dag}$ & 8.042E+0(2.95E-1)$^{\dag}$ & 1.454E+1(1.08E-2)$^{\dag}$ &       & 10    & \cellcolor[rgb]{ .753,  .753,  .753}\textbf{4.719E-2(3.85E-6)} & 6.694E-2(4.07E-4)$^{\dag}$ & 8.133E-2(1.28E-2)$^{\dag}$ & 7.044E-1(1.75E-5)$^{\dag}$ \\
    \hline
    \hline
    \multirow{5}[2]{*}{\texttt{PMOP5}} & 2     & \cellcolor[rgb]{ .753,  .753,  .753}\textbf{1.254E-2(4.35E-3)} & 1.590E-1(2.27E-2)$^{\dag}$ & 1.789E-2(3.20E-2) & 1.649E-1(2.75E-2)$^{\dag}$ & \multirow{5}[2]{*}{\texttt{PMOP11}} & 2     & \cellcolor[rgb]{ .753,  .753,  .753}\textbf{1.470E-2(2.93E-5)} & 1.969E-1(2.86E-3)$^{\dag}$ & 3.219E-2(1.59E-1) & 2.070E-1(2.86E-3)$^{\dag}$ \\
          & 3     & 3.829E-2(2.94E-2) & 1.038E-1(3.20E-2)$^{\dag}$ & \cellcolor[rgb]{ .753,  .753,  .753}\textbf{3.411E-2(3.33E-2)} & 1.556E-1(4.32E-1)$^{\dag}$ &       & 3     & 7.565E-2(1.15E-4) & 3.120E-1(2.44E-3)$^{\dag}$ & \cellcolor[rgb]{ .753,  .753,  .753}\textbf{3.731E-2(6.09E-2)$^{\ddag}$} & 3.832E-1(3.30E-3)$^{\dag}$ \\
          & 5     & \cellcolor[rgb]{ .753,  .753,  .753}\textbf{4.202E-2(6.72E-3)} & 5.912E-2(3.03E-2) & 1.228E-1(6.63E-2)$^{\dag}$ & 2.704E-1(7.59E-2)$^{\dag}$ &       & 5     & \cellcolor[rgb]{ .753,  .753,  .753}\textbf{6.965E-2(1.00E-5)} & 2.477E-1(1.55E-3)$^{\dag}$ & 8.800E-2(2.73E-2)$^{\dag}$ & 2.963E-1(3.03E-5)$^{\dag}$ \\
          & 8     & \cellcolor[rgb]{ .753,  .753,  .753}\textbf{2.231E-2(5.25E-3)} & 2.533E-2(8.54E-3) & 4.842E-2(6.13E-2)$^{\dag}$ & 1.194E-1(6.91E-2)$^{\dag}$ &       & 8     & \cellcolor[rgb]{ .753,  .753,  .753}\textbf{8.198E-2(3.78E-6)} & 1.046E-1(2.08E-2)$^{\dag}$ & 8.997E-2(1.73E-2)$^{\dag}$ & 1.286E-1(3.21E-8)$^{\dag}$ \\
          & 10    & \cellcolor[rgb]{ .753,  .753,  .753}\textbf{1.007E-2(1.82E-3)} & 1.199E-2(3.47E-4) & 1.892E-2(6.15E-2)$^{\dag}$ & 5.454E-2(2.18E-2)$^{\dag}$ &       & 10    & \cellcolor[rgb]{ .753,  .753,  .753}\textbf{2.523E-2(1.13E-5)} & 4.970E-2(7.61E-3)$^{\dag}$ & 2.652E-2(6.41E-3) & 6.226E-2(6.44E-5)$^{\dag}$ \\
    \hline
    \hline
    \multirow{5}[2]{*}{\texttt{PMOP6}} & 2     & \cellcolor[rgb]{ .753,  .753,  .753}\textbf{4.134E-3(8.91E-6)} & 8.616E-1(1.63E-2)$^{\dag}$ & 4.861E-3(2.10E-3)$^{\dag}$ & 8.427E-1(1.24E-2)$^{\dag}$ & \multirow{5}[2]{*}{\texttt{PMOP12}} & 2     & 4.026E-1(1.68E-5) & 4.004E-1(9.91E-5)$^{\dag}$ & \cellcolor[rgb]{ .753,  .753,  .753}\textbf{1.017E-3(2.80E-4)$^{\ddag}$} & 3.950E-1(1.44E-1)$^{\dag}$ \\
          & 3     & 9.777E-2(2.38E-4) & 7.560E-1(5.45E-2)$^{\dag}$ & \cellcolor[rgb]{ .753,  .753,  .753}\textbf{6.959E-2(3.00E-2)} & 1.335E+0(2.85E-1)$^{\dag}$ &       & 3     & 8.165E-2(1.88E-5) & 1.159E-1(1.25E-2) & \cellcolor[rgb]{ .753,  .753,  .753}\textbf{6.997E-3(3.61E-3)$^{\ddag}$} & 1.829E-1(4.74E-2)$^{\dag}$ \\
          & 5     & \cellcolor[rgb]{ .753,  .753,  .753}\textbf{2.297E-1(1.68E-4)} & 1.112E+0(5.64E-2)$^{\dag}$ & 5.264E-1(2.49E-1)$^{\dag}$ & 1.989E+0(1.02E-1)$^{\dag}$ &       & 5     & \cellcolor[rgb]{ .753,  .753,  .753}\textbf{8.176E-3(8.35E-5)} & 3.012E-2(1.19E-4)$^{\dag}$ & 2.280E-2(1.46E-2)$^{\dag}$ & 7.627E-2(1.12E-2)$^{\dag}$ \\
          & 8     & \cellcolor[rgb]{ .753,  .753,  .753}\textbf{1.834E+0(5.09E-1)} & 1.121E+1(3.44E+0)$^{\dag}$ & 2.145E+0(1.59E-1)$^{\dag}$ & 1.766E+1(4.78E+0)$^{\dag}$ &       & 8     & \cellcolor[rgb]{ .753,  .753,  .753}\textbf{6.208E-3(1.99E-5)} & 1.191E-2(2.23E-5)$^{\dag}$ & 9.411E-3(1.67E-3) & 2.711E-2(7.56E-5)$^{\dag}$ \\
          & 10    & 9.023E+2(2.83E+1) & 2.440E+2(1.39E+1)$^{\dag}$ & \cellcolor[rgb]{ .753,  .753,  .753}\textbf{7.069E+1(2.97E+1)} & 4.447E+2(9.24E+1)$^{\dag}$ &       & 10    & \cellcolor[rgb]{ .753,  .753,  .753}\textbf{2.071E-3(1.14E-5)} & 8.014E-3(1.64E-5)$^{\dag}$ & 3.877E-3(1.35E-3) & 1.793E-2(3.83E-5)$^{\dag}$ \\
    \hline
    \end{tabular}
    
    \begin{tablenotes}
        \item[1] $^{\dag}$ denotes the performance of $\mathtt{NSGA}$-$\mathtt{II}$-$\mathtt{KPITU}$ is significantly better than the other peers according to the Wilcoxon's rank sum test at a 0.05 significance level whilst $^{\ddag}$ denotes the opposite case.
      \end{tablenotes}
    \end{threeparttable}}
  \label{tab:nsga2_local}
\end{table*}

\begin{table*}[htbp]
  \centering
  \caption{Comparison Results of $\mathbb{I}(\mathrm{S})$ Metric Values Obtained by $\mathtt{NSGA}$-$\mathtt{II}$-$\mathtt{KPITU}$ and Other Three Other Peer Algorithms on Problems with Infinitely Many Knee Points}
  \resizebox{\textwidth}{!}{
  \begin{threeparttable}
    \begin{tabular}{c|c|cccc||c|c|cccc}
    \cline{2-6}\cline{8-12}
          & $m$     & \texttt{KPITU} & \texttt{CD}    & \texttt{EMU}   & \texttt{CHIM}  &       & $m$     & \texttt{KPITU} & \texttt{CD}    & \texttt{EMU}   & \texttt{CHIM} \\
    \hline
    \multirow{4}[2]{*}{\texttt{PMOP13}} & 3     & \cellcolor[rgb]{ .753,  .753,  .753}\textbf{1.831E-1(2.77E-4)} & 7.411E-1(1.65E-2)$^{\dag}$ & 6.508E-1(4.79E-1)$^{\dag}$ & 1.410E+0(2.58E-1)$^{\dag}$ & \multirow{4}[2]{*}{\texttt{PMOP14}} & 3     & \cellcolor[rgb]{ .753,  .753,  .753}\textbf{1.189E-1(8.23E-5)} & 2.462E-1(2.11E-3)$^{\dag}$ & 1.525E-1(2.76E-2)$^{\dag}$ & 1.089E+0(1.55E-4)$^{\dag}$ \\
          & 5     & \cellcolor[rgb]{ .753,  .753,  .753}\textbf{3.296E-1(8.94E-3)} & 1.193E+0(5.23E-3)$^{\dag}$ & 8.893E-1(5.75E-1)$^{\dag}$ & 1.487E+0(5.12E-1)$^{\dag}$ &       & 5     & 1.639E-1(2.68E-5) & 1.403E-1(5.10E-3) & \cellcolor[rgb]{ .753,  .753,  .753}\textbf{1.591E-1(2.28E-2)} & 9.933E-1(2.01E-1)$^{\dag}$ \\
          & 8     & \cellcolor[rgb]{ .753,  .753,  .753}\textbf{5.107E-1(7.52E-3)} & 1.244E+0(6.19E-2)$^{\dag}$ & 1.045E+0(8.24E-1)$^{\dag}$ & 1.721E+0(8.32E-1)$^{\dag}$ &       & 8     & \cellcolor[rgb]{ .753,  .753,  .753}\textbf{4.590E-2(1.54E-5)} & 4.746E-2(4.00E-2) & 5.470E-2(3.33E-2) & 6.300E-2(1.63E-2)$^{\dag}$ \\
          & 10    & \cellcolor[rgb]{ .753,  .753,  .753}\textbf{4.977E-1(1.94E-3)} & 9.362E-1(1.42E-1)$^{\dag}$ & 1.306E+0(5.69E-1)$^{\dag}$ & 1.613E+0(6.05E-1)$^{\dag}$ &       & 10    & \cellcolor[rgb]{ .753,  .753,  .753}\textbf{9.602E-2(2.65E-5)} & 1.029E-1(1.95E-2) & 4.445E-1(3.95E-1)$^{\dag}$ & 2.338E-1(6.21E-2)$^{\dag}$ \\
    \hline
    \hline
    \multirow{4}[2]{*}{\texttt{PMOP13}} & 3     & 3.339E-1(4.23E-4) & 1.432E+0(1.25E-2)$^{\dag}$ & \cellcolor[rgb]{ .753,  .753,  .753}\textbf{2.326E-1(1.67E-2)$^{\ddag}$} & 3.653E+0(1.09E-2)$^{\dag}$ & \multirow{4}[2]{*}{\texttt{PMOP14}} & 3     & 2.526E-1(1.40E-4) & 5.318E-1(1.69E-2)$^{\dag}$ & \cellcolor[rgb]{ .753,  .753,  .753}\textbf{1.145E-1(1.23E-2)$^{\ddag}$} & 7.794E-1(8.02E-2)$^{\dag}$ \\
          & 5     & \cellcolor[rgb]{ .753,  .753,  .753}\textbf{3.973E-1(7.52E-2)} & 1.909E+0(1.19E-2)$^{\dag}$ & 6.884E-1(3.25E-1)$^{\dag}$ & 2.448E+0(8.25E-1)$^{\dag}$ &       & 5     & \cellcolor[rgb]{ .753,  .753,  .753}\textbf{1.495E-1(9.99E-2)} & 5.409E-1(1.41E-1)$^{\dag}$ & 2.843E-1(1.46E-1)$^{\dag}$ & 5.626E-1(2.94E-1)$^{\dag}$ \\
          & 8     & \cellcolor[rgb]{ .753,  .753,  .753}\textbf{8.846E-1(1.74E-2)} & 2.436E+0(1.94E-1)$^{\dag}$ & 1.487E+0(7.99E-1) & 2.812E+0(1.37E+0)$^{\dag}$ &       & 8     & \cellcolor[rgb]{ .753,  .753,  .753}\textbf{2.981E-1(8.21E-2)} & 6.041E-1(1.24E-2)$^{\dag}$ & 3.335E-1(1.11E-1) & 5.134E-1(6.53E-2)$^{\dag}$ \\
          & 10    & \cellcolor[rgb]{ .753,  .753,  .753}\textbf{1.345E+0(4.63E-2)} & 2.564E+0(3.29E-1)$^{\dag}$ & 1.626E+0(8.14E-1) & 3.094E+0(1.68E+0)$^{\dag}$ &       & 10    & 3.621E-1(6.29E-2) & 4.926E-1(5.99E-2)$^{\dag}$ & \cellcolor[rgb]{ .753,  .753,  .753}\textbf{3.381E-1(1.40E-1)} & 5.407E-1(4.25E-2)$^{\dag}$ \\
    \hline
    \end{tabular}
    
    \begin{tablenotes}
        \item[1] $^{\dag}$ denotes the performance of $\mathtt{NSGA}$-$\mathtt{II}$-$\mathtt{KPITU}$ is significantly better than the other peers according to the Wilcoxon's rank sum test at a 0.05 significance level whilst $^{\ddag}$ denotes the opposite case.
      \end{tablenotes}
    \end{threeparttable}}
  \label{tab:nsga2_infinite}
\end{table*}

\begin{figure*}[htbp]
    \includegraphics[width=1.0\linewidth]{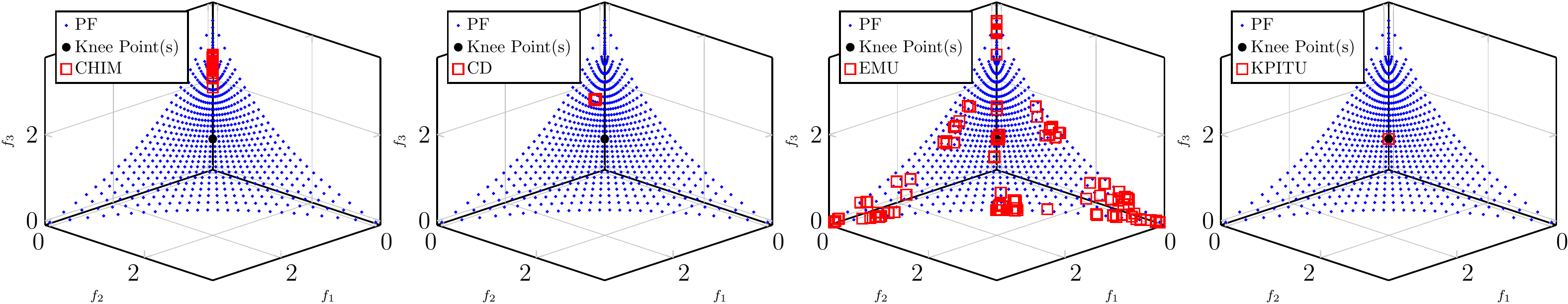}
    \caption{Population distribution of solutions found by NSGA-II-KPITU, NSGA-II-EMU, NSGA-II-CHIM and NSGA-II-CD on 3-objective PMOP7 with one global knee point.}
    \label{fig:knea_PMOP7_M3_A1}
\end{figure*}

\begin{figure*}[htbp]
    \includegraphics[width=1.0\linewidth]{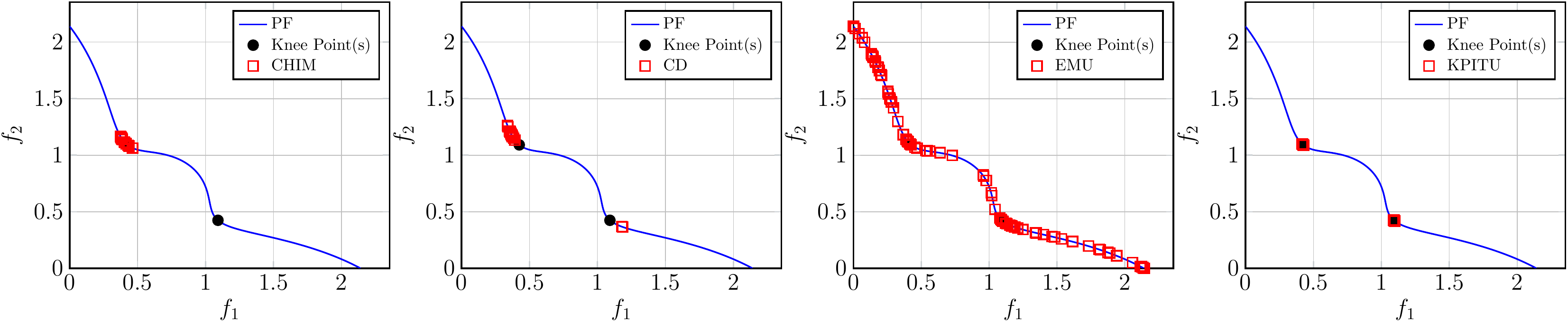}
    \caption{Population distribution of solutions found by NSGA-II-KPITU, NSGA-II-EMU, NSGA-II-CHIM and NSGA-II-CD on 2-objective PMOP1 with local knee points.}
    \label{fig:knea_PMOP1_M2_A2}
\end{figure*}

\begin{figure*}[htbp]
    \includegraphics[width=1.0\linewidth]{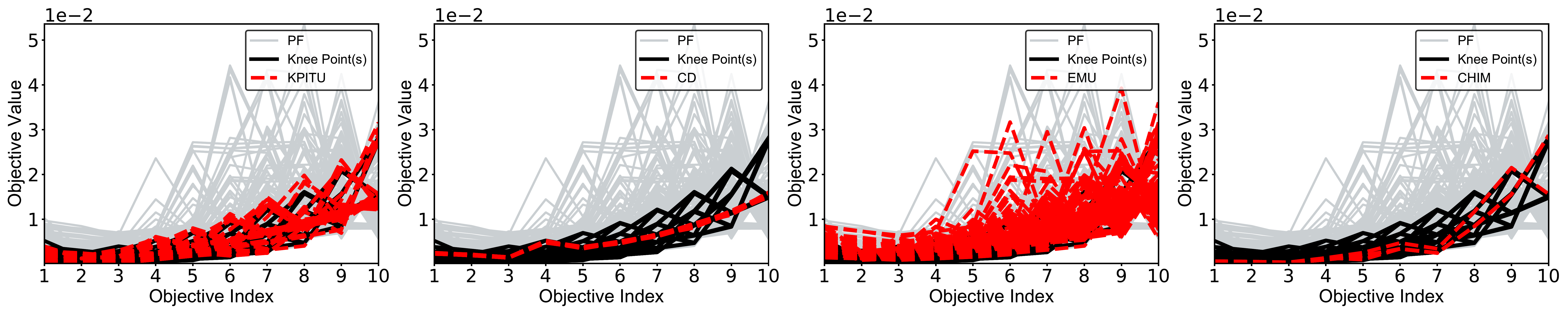}
    \caption{Population distribution of solutions found by NSGA-II-KPITU, NSGA-II-EMU, NSGA-II-CHIM and NSGA-II-CD on 10-objective PMOP5 with local knee points.}
    \label{fig:knea_PMOP5_M10_A1}
\end{figure*}



\section{Conclusion}
\label{sec:conclusion}

Knee region, where knee point(s) are located in, is a part of the PF experiencing the smallest trade-off loss at all objectives. Due to this defining characteristic, knee point(s) is(are) of a distinctive importance in MCDM. In contrast, Pareto-optimal solutions outside the knee region(s) are less attractive to DMs since a small improvement on one objective can lead to a significant degradation on at least one of the other objectives. This paper proposed a simple and effective KPI method based on trade-off utility, dubbed KPITU, to help DMs identify knee point(s) from a given set of trade-off solutions. Generally speaking, the KPI process of KPITU is mainly based on pair-wise comparisons between each solution of the trade-off solution set and solutions within its neighbourhood. In particular, a solution is a knee point if and only if it has the best trade-off utility among its neighbours. This vanilla KPI process of KPITU is independent from each solution and has a quadratic worst-case complexity. To amend this, we implement a GPU version which is able to carry out the KPI process in a parallel manner and it reduces the worst-case complexity from quadratic to linear. Experiments on 134 test problems instances clearly demonstrate the outstanding performance of KPITU against five state-of-the-art KPI methods especially on problems with many local knee points. It is worth noting that KPITU does not have any additional control parameters and is naturally scalable to any number of objectives. In addition, we validate the usefulness of KPITU as an operator to guide an EMO algorithm to search for knee point(s) by incorporating it into NSGA-II.

Supporting a better and explainable decision-making when encountering multiple conflicting objectives has become increasingly important to bridge the gap between EMO research and its wider industrial applications. As a piece of future directions, in addition to a numerical representation of raw objective function values, it is interesting and important to develop a meaningful visualisation method to assist the DM's cognitive understanding of trade-off alternatives and thus facilitate the MCDM. Furthermore, trade-off solutions obtained by EMO algorithms usually imply abundant knowledge including the relationship between objective functions and decision variables. An effective data mining over those trade-off solutions is able to uncover innovative design principles thus automate the decision-making process and improve the interpretability of the optimisation itself~\cite{Deb13}.

\noindent\textbf{Data availability:} The supplementary document and source codes can be found from our project page: \url{https://github.com/COLA-Laboratory/kpi}.

\section*{Acknowledgment}
K. Li was supported by UKRI Future Leaders Fellowship (Grant No. MR/S017062/1) and Royal Society (Grant No. IEC/NSFC/170243). X. Yao was supported by the Guangdong Provincial Key Laboratory of Brain-inspired Intelligent Computation, the Program for Guangdong Introducing Innovative and Enterpreneurial Teams (Grant No. 2017ZT07X386), Shenzhen Science and Technology Program (Grant No. KQTD2016112514355531) and the Program for University Key Laboratory of Guangdong Province (Grant No. 2017KSYS008).

\bibliographystyle{IEEEtran}
\bibliography{IEEEabrv,knee}

\begin{thebibliography}{10}
\providecommand{\url}[1]{#1}
\csname url@samestyle\endcsname
\providecommand{\newblock}{\relax}
\providecommand{\bibinfo}[2]{#2}
\providecommand{\BIBentrySTDinterwordspacing}{\spaceskip=0pt\relax}
\providecommand{\BIBentryALTinterwordstretchfactor}{4}
\providecommand{\BIBentryALTinterwordspacing}{\spaceskip=\fontdimen2\font plus
\BIBentryALTinterwordstretchfactor\fontdimen3\font minus
  \fontdimen4\font\relax}
\providecommand{\BIBforeignlanguage}[2]{{%
\expandafter\ifx\csname l@#1\endcsname\relax
\typeout{** WARNING: IEEEtran.bst: No hyphenation pattern has been}%
\typeout{** loaded for the language `#1'. Using the pattern for}%
\typeout{** the default language instead.}%
\else
\language=\csname l@#1\endcsname
\fi
#2}}
\providecommand{\BIBdecl}{\relax}
\BIBdecl

\bibitem{Jacoby74}
J.~Jacoby, D.~E. Speller, and C.~A. Kohn, ``Brand choice behavior as a function
  of information load,'' \emph{Journal of Marketing Research}, vol.~11, no.~1,
  pp. 63--69, 1974.

\bibitem{LiDY18}
K.~Li, K.~Deb, and X.~Yao, ``R-metric: Evaluating the performance of
  preference-based evolutionary multiobjective optimization using reference
  points,'' \emph{{IEEE} Trans. Evolutionary Computation}, vol.~22, no.~6, pp.
  821--835, 2018.

\bibitem{Li19}
K.~Li, ``Progressive preference learning: Proof-of-principle results in
  {MOEA/D},'' in \emph{EMO'19: Proc. of the 10th International Conference
  Evolutionary Multi-Criterion Optimization}, 2019, pp. 631--643.

\bibitem{LiCSY19}
K.~Li, R.~Chen, D.~A. Savic, and X.~Yao, ``Interactive decomposition
  multiobjective optimization via progressively learned value functions,''
  \emph{{IEEE} Trans. Fuzzy Systems}, vol.~27, no.~5, pp. 849--860, 2019.

\bibitem{DebG11}
K.~Deb and S.~Gupta, ``Understanding knee points in bicriteria problems and
  their implications as preferred solution principles,'' \emph{Engineering
  Optimization}, vol.~43, no.~11, pp. 1175--1204, 2011.

\bibitem{ChiuYJ16}
W.~Chiu, G.~G. Yen, and T.~Juan, ``Minimum manhattan distance approach to
  multiple criteria decision making in multiobjective optimization problems,''
  \emph{{IEEE} Trans. Evolutionary Computation}, vol.~20, no.~6, pp. 972--985,
  2016.

\bibitem{BhattacharjeeSR17}
K.~S. Bhattacharjee, H.~K. Singh, M.~Ryan, and T.~Ray, ``Bridging the gap:
  Many-objective optimization and informed decision-making,'' \emph{{IEEE}
  Trans. Evolutionary Computation}, vol.~21, no.~5, pp. 813--820, 2017.

\bibitem{Das99}
I.~Das, ``On characterizing the "knee" of the {Pareto} curve based on
  normal-boundary intersection,'' \emph{Structural Optimization}, vol.~18, no.
  2-3, pp. 107--115, 1999.

\bibitem{SchutzeLC08}
O.~Sch{\"{u}}tze, M.~Laumanns, and C.~A.~C. Coello, ``Approximating the knee of
  an {MOP} with stochastic search algorithms,'' in \emph{PPSN'08: Proc. of the
  10th International Conference on Parallel Problem Solving from Nature}, 2008,
  pp. 795--804.

\bibitem{YuJO18}
G.~Yu, Y.~Jin, and M.~Olhofer, ``A method for a posteriori identification of
  knee points based on solution density,'' in \emph{CEC'18: Proc. of 2018
  {IEEE} Congress on Evolutionary Computation}, 2018, pp. 1--8.

\bibitem{BrankeDDO04}
J.~Branke, K.~Deb, H.~Dierolf, and M.~Osswald, ``Finding knees in
  multi-objective optimization,'' in \emph{PPSN'04: Proc. of the 8th
  International Conference on Parallel Problem Solving from Nature}, 2004, pp.
  722--731.

\bibitem{Ikeda2001}
K.~Ikeda, H.~Kita, and S.~Kobayashi, ``Failure of pareto-based {MOEAs}: Does
  non-dominated really mean near to optimal?'' in \emph{CEC'01: Proc. of the
  2001 Congress on Evolutionary Computation}, vol.~2, 2001, pp. 957--962.

\bibitem{RachmawatiS06}
L.~Rachmawati and D.~Srinivasan, ``A multi-objective evolutionary algorithm
  with weighted-sum niching for convergence on knee regions,'' in
  \emph{GECCO'06: Proc. of the 2006 Genetic and Evolutionary Computation
  Conference}, 2006, pp. 749--750.

\bibitem{RachmawatiS06a}
------, ``A multi-objective genetic algorithm with controllable convergence on
  knee regions,'' in \emph{CEC'06: Proc. of the 2006 {IEEE} International
  Conference on Evolutionary Computation}, 2006, pp. 1916--1923.

\bibitem{BechikhSG10}
S.~Bechikh, L.~B. Said, and K.~Gh{\'{e}}dira, ``Searching for knee regions in
  multi-objective optimization using mobile reference points,'' in
  \emph{SAC'10: Proc. of the 2010 {ACM} Symposium on Applied Computing}, 2010,
  pp. 1118--1125.

\bibitem{BechikhSG11}
------, ``Searching for knee regions of the pareto front using mobile reference
  points,'' \emph{Soft Comput.}, vol.~15, no.~9, pp. 1807--1823, 2011.

\bibitem{ZhangTJ15}
X.~Zhang, Y.~Tian, and Y.~Jin, ``A knee point-driven evolutionary algorithm for
  many-objective optimization,'' \emph{{IEEE} Trans. Evolutionary Computation},
  vol.~19, no.~6, pp. 761--776, 2015.

\bibitem{SudengW15}
S.~Sudeng and N.~Wattanapongsakorn, ``A preference-based multiobjective
  evolutioary algorithm for finding knee solutions,'' \emph{Applied Mechanics
  \& Materials}, vol. 781, p.~5, 2015.

\bibitem{Ramirez-Atencia17}
C.~Ram{\'{\i}}rez{-}Atencia, S.~Mostaghim, and D.~Camacho, ``A knee point based
  evolutionary multi-objective optimization for mission planning problems,'' in
  \emph{GECCO'17: Proc. of the Genetic and Evolutionary Computation
  Conference}, 2017, pp. 1216--1223.

\bibitem{YuJO19}
G.~Yu, Y.~Jin, and M.~Olhofer, ``An a priori knee identification
  multi-objective evolutionary algorithm based on $\alpha$-dominance,'' in
  \emph{GECCO'19: Proc. of the 2019 Genetic and Evolutionary Computation
  Conference Companion}, 2019, pp. 241--242.

\bibitem{YuJO19a}
------, ``Benchmark problems and performance indicators for search of knee
  points in multiobjective optimization,'' \emph{IEEE Trans. Cybernetics},
  2019, accepted for publication.

\bibitem{DebAPM02}
K.~Deb, S.~Agrawal, A.~Pratap, and T.~Meyarivan, ``A fast and elitist
  multiobjective genetic algorithm: {NSGA-II},'' \emph{{IEEE} Trans.
  Evolutionary Computation}, vol.~6, no.~2, pp. 182--197, 2002.

\bibitem{RachmawatiS09}
L.~Rachmawati and D.~Srinivasan, ``Multiobjective evolutionary algorithm with
  controllable focus on the knees of the pareto front,'' \emph{{IEEE} Trans.
  Evolutionary Computation}, vol.~13, no.~4, pp. 810--824, 2009.

\bibitem{BhattacharjeeSR16}
K.~S. Bhattacharjee, H.~K. Singh, and T.~Ray, ``A study on performance metrics
  to identify solutions of interest from a trade-off set,'' in \emph{ACALCI'16:
  Proc. of 2016 Artificial Life and Computational Intelligence}, 2016, pp.
  66--77.

\bibitem{LiuGZ14}
H.~Liu, F.~Gu, and Q.~Zhang, ``Decomposition of a multiobjective optimization
  problem into a number of simple multiobjective subproblems,'' \emph{{IEEE}
  Trans. Evolutionary Computation}, vol.~18, no.~3, pp. 450--455, 2014.

\bibitem{DasD98}
I.~Das and J.~E. Dennis, ``Normal-boundary intersection: {A} new method for
  generating the pareto surface in nonlinear multicriteria optimization
  problems,'' \emph{{SIAM} Journal on Optimization}, vol.~8, no.~3, pp.
  631--657, 1998.

\bibitem{BosmanT03}
P.~A.~N. Bosman and D.~Thierens, ``The balance between proximity and diversity
  in multiobjective evolutionary algorithms,'' \emph{{IEEE} Trans. Evolutionary
  Computation}, vol.~7, no.~2, pp. 174--188, 2003.

\bibitem{LiZZL09}
K.~Li, J.~Zheng, C.~Zhou, and H.~Lv, ``An improved differential evolution for
  multi-objective optimization,'' in \emph{CSIE'09: Proc. of 2009 {WRI} World
  Congress on Computer Science and Information Engineering}, 2009, pp.
  825--830.

\bibitem{LiZLZL09}
K.~Li, J.~Zheng, M.~Li, C.~Zhou, and H.~Lv, ``A novel algorithm for
  non-dominated hypervolume-based multiobjective optimization,'' in
  \emph{SMC'09: Proc. of 2009 the {IEEE} International Conference on Systems,
  Man and Cybernetics}, 2009, pp. 5220--5226.

\bibitem{LiKWCR12}
K.~Li, S.~Kwong, R.~Wang, J.~Cao, and I.~J. Rudas, ``Multi-objective
  differential evolution with self-navigation,'' in \emph{SMC'12: Proc. of the
  2012 {IEEE} International Conference on Systems, Man, and Cybernetics}, 2012,
  pp. 508--513.

\bibitem{LiKCLZS12}
K.~Li, S.~Kwong, J.~Cao, M.~Li, J.~Zheng, and R.~Shen, ``Achieving balance
  between proximity and diversity in multi-objective evolutionary algorithm,''
  \emph{Inf. Sci.}, vol. 182, no.~1, pp. 220--242, 2012.

\bibitem{CaoKWL12}
J.~Cao, S.~Kwong, R.~Wang, and K.~Li, ``A weighted voting method using minimum
  square error based on extreme learning machine,'' in \emph{ICMLC'12: Proc. of
  the 2012 International Conference on Machine Learning and Cybernetics}, 2012,
  pp. 411--414.

\bibitem{LiKWTM13}
K.~Li, S.~Kwong, R.~Wang, K.~Tang, and K.~Man, ``Learning paradigm based on
  jumping genes: {A} general framework for enhancing exploration in
  evolutionary multiobjective optimization,'' \emph{Inf. Sci.}, vol. 226, pp.
  1--22, 2013.

\bibitem{LiK14}
K.~Li and S.~Kwong, ``A general framework for evolutionary multiobjective
  optimization via manifold learning,'' \emph{Neurocomputing}, vol. 146, pp.
  65--74, 2014.

\bibitem{CaoKWL14}
J.~Cao, S.~Kwong, R.~Wang, and K.~Li, ``{AN} indicator-based selection
  multi-objective evolutionary algorithm with preference for multi-class
  ensemble,'' in \emph{ICMLC'14: Proc. of the 2014 International Conference on
  Machine Learning and Cybernetics}, 2014, pp. 147--152.

\bibitem{WuKZLWL15}
M.~Wu, S.~Kwong, Q.~Zhang, K.~Li, R.~Wang, and B.~Liu, ``Two-level stable
  matching-based selection in {MOEA/D},'' in \emph{SMC'15: Proc. of the 2015
  {IEEE} International Conference on Systems, Man, and Cybernetics}, 2015, pp.
  1720--1725.

\bibitem{LiKZD15}
K.~Li, S.~Kwong, Q.~Zhang, and K.~Deb, ``Interrelationship-based selection for
  decomposition multiobjective optimization,'' \emph{{IEEE} Trans.
  Cybernetics}, vol.~45, no.~10, pp. 2076--2088, 2015.

\bibitem{LiKD15}
K.~Li, S.~Kwong, and K.~Deb, ``A dual-population paradigm for evolutionary
  multiobjective optimization,'' \emph{Inf. Sci.}, vol. 309, pp. 50--72, 2015.

\bibitem{LiDZ15}
K.~Li, K.~Deb, and Q.~Zhang, ``Evolutionary multiobjective optimization with
  hybrid selection principles,'' in \emph{CEC'15: Proc. of the 2015 {IEEE}
  Congress on Evolutionary Computation}, 2015, pp. 900--907.

\bibitem{LiODY16}
K.~Li, M.~N. Omidvar, K.~Deb, and X.~Yao, ``Variable interaction in
  multi-objective optimization problems,'' in \emph{PPSN XIV: Proc. of the 14th
  International Conference Parallel Problem Solving from Nature}, vol.
  9921.\hskip 1em plus 0.5em minus 0.4em\relax Springer, 2016, pp. 399--409.

\bibitem{LiDZZ17}
K.~Li, K.~Deb, Q.~Zhang, and Q.~Zhang, ``Efficient nondomination level update
  method for steady-state evolutionary multiobjective optimization,''
  \emph{{IEEE} Trans. Cybernetics}, vol.~47, no.~9, pp. 2838--2849, 2017.

\bibitem{WuKJLZ17}
M.~Wu, S.~Kwong, Y.~Jia, K.~Li, and Q.~Zhang, ``Adaptive weights generation for
  decomposition-based multi-objective optimization using gaussian process
  regression,'' in \emph{Proceedings of the Genetic and Evolutionary
  Computation Conference, {GECCO} 2017, Berlin, Germany, July 15-19, 2017},
  2017, pp. 641--648.

\bibitem{WuLKZZ17}
M.~Wu, K.~Li, S.~Kwong, Y.~Zhou, and Q.~Zhang, ``Matching-based selection with
  incomplete lists for decomposition multiobjective optimization,''
  \emph{{IEEE} Trans. Evolutionary Computation}, vol.~21, no.~4, pp. 554--568,
  2017.

\bibitem{LiWKC13}
K.~Li, R.~Wang, S.~Kwong, and J.~Cao, ``Evolving extreme learning machine
  paradigm with adaptive operator selection and parameter control,''
  \emph{International Journal of Uncertainty, Fuzziness and Knowledge-Based
  Systems}, vol.~21, pp. 143--154, 2013.

\bibitem{ChenLY18}
R.~Chen, K.~Li, and X.~Yao, ``Dynamic multiobjectives optimization with a
  changing number of objectives,'' \emph{{IEEE} Trans. Evolutionary
  Computation}, vol.~22, no.~1, pp. 157--171, 2018.

\bibitem{KumarBCLB18}
S.~Kumar, R.~Bahsoon, T.~Chen, K.~Li, and R.~Buyya, ``Multi-tenant cloud
  service composition using evolutionary optimization,'' in \emph{ICPADS'18:
  Proc. of the 24th {IEEE} International Conference on Parallel and Distributed
  Systems}, 2018, pp. 972--979.

\bibitem{WuLKZ20}
M.~Wu, K.~Li, S.~Kwong, and Q.~Zhang, ``Evolutionary many-objective
  optimization based on adversarial decomposition,'' \emph{{IEEE} Trans.
  Cybernetics}, vol.~50, no.~2, pp. 753--764, 2020.

\bibitem{ChenLBY18}
T.~Chen, K.~Li, R.~Bahsoon, and X.~Yao, ``{FEMOSAA:} feature-guided and
  knee-driven multi-objective optimization for self-adaptive software,''
  \emph{{ACM} Trans. Softw. Eng. Methodol.}, vol.~27, no.~2, pp. 5:1--5:50,
  2018.

\bibitem{LiCFY19}
K.~Li, R.~Chen, G.~Fu, and X.~Yao, ``Two-archive evolutionary algorithm for
  constrained multiobjective optimization,'' \emph{{IEEE} Trans. Evolutionary
  Computation}, vol.~23, no.~2, pp. 303--315, 2019.

\bibitem{WuLKZZ19}
M.~Wu, K.~Li, S.~Kwong, Q.~Zhang, and J.~Zhang, ``Learning to decompose: {A}
  paradigm for decomposition-based multiobjective optimization,'' \emph{{IEEE}
  Trans. Evolutionary Computation}, vol.~23, no.~3, pp. 376--390, 2019.

\bibitem{LiuLC19}
M.~Liu, K.~Li, and T.~Chen, ``Security testing of web applications: a
  search-based approach for detecting {SQL} injection vulnerabilities,'' in
  \emph{GECCO'19: Proc. of the 2019 Genetic and Evolutionary Computation
  Conference Companion}, 2019, pp. 417--418.

\bibitem{LiXT19}
K.~Li, Z.~Xiang, and K.~C. Tan, ``Which surrogate works for empirical
  performance modelling? {A} case study with differential evolution,'' in
  \emph{CEC'19: Proc. of the 2019 {IEEE} Congress on Evolutionary Computation},
  2019, pp. 1988--1995.

\bibitem{ZouJYZZL19}
J.~Zou, C.~Ji, S.~Yang, Y.~Zhang, J.~Zheng, and K.~Li, ``A knee-point-based
  evolutionary algorithm using weighted subpopulation for many-objective
  optimization,'' \emph{Swarm and Evolutionary Computation}, vol.~47, pp.
  33--43, 2019.

\bibitem{BillingsleyLMMG19}
J.~Billingsley, K.~Li, W.~Miao, G.~Min, and N.~Georgalas, ``A formal model for
  multi-objective optimisation of network function virtualisation placement,''
  in \emph{EMO'19: Proc. of the 10th International Conference Evolutionary
  Multi-Criterion Optimization}, 2019, pp. 529--540.

\bibitem{GaoNL19}
H.~Gao, H.~Nie, and K.~Li, ``Visualisation of pareto front approximation: {A}
  short survey and empirical comparisons,'' in \emph{CEC'19: Proc. of the 2019
  {IEEE} Congress on Evolutionary Computation}, 2019, pp. 1750--1757.

\bibitem{Deb13}
K.~Deb, ``Innovization: Discovery of innovative solution principles using
  multi-objective optimization,'' in \emph{EMO'13: Proc. of the 7th
  International Conference Evolutionary Multi-Criterion Optimization}, 2013,
  pp. 4--5.

\end{thebibliography}

\end{document}